\documentclass[a4paper, reqno, twoside, 10pt]{amsart}

\usepackage{amsmath,amssymb,amscd,mathtools,mathrsfs}
\usepackage{booktabs,multirow,array,longtable,float}
\usepackage{afterpage}
\usepackage{enumitem}
\usepackage[T1]{fontenc}
\usepackage{newtxtext,newtxmath}
\usepackage{helvet}
\usepackage{courier}

\usepackage[hidelinks]{hyperref}
\hypersetup{pdftitle={Rational homology disk smoothings of weighted homogeneous surface singularities},pdfauthor={Heesang Park and Dongsoo Shin}}

\newtheorem{theorem}{Theorem}[section]
\newtheorem{maintheorem}{Theorem}

\newtheorem{lemma}[theorem]{Lemma}
\newtheorem{proposition}[theorem]{Proposition}
\newtheorem*{proposition*}{Proposition}
\newtheorem{corollary}[theorem]{Corollary}

\newtheorem*{conjecture*}{Conjecture}

\theoremstyle{definition}

\newtheorem{example}[theorem]{Example}

\newtheorem{remark}[theorem]{Remark}
\newtheorem*{remark*}{Remark}

\numberwithin{equation}{section}

\DeclareMathOperator{\Spec}{Spec}
\DeclareMathOperator{\Proj}{Proj}

\begin{document}

\title[WHSS admitting QHDS]{Weighted homogeneous surface singularities admitting rational homology disk smoothings}

\author[H. Park]{Heesang Park}
\address{Department of Mathematics, Konkuk University, Seoul, Republic of Korea}
\email{HeesangPark@konkuk.ac.kr}

\author[D. Shin]{Dongsoo Shin}
\address{Department of Mathematics, Chungnam National University, Daejeon, Republic of Korea}
\email{dsshin@cnu.ac.kr}



\begin{abstract}
We study rational homology disk smoothings through canonical-cover representations and specialization of reflexive canonical powers, and recover the known dual graph and analytic classification. The geometric-genus representation of the canonical covering group, together with its complex conjugate, is a direct summand of the reduced regular representation. Specialization forces the canonical Riemann--Roch corrections to vanish; explicit formulas give equations for the Seifert invariants. For three arms and central weight two, these equations, the character restrictions, and the numerical identities determine the Seifert data without assuming the existence of a smoothing. Projective compactifications treat the remaining cases, and elliptic pencils explain the four-arm attachment points and the parameter independence of their cross ratios. We construct strong \(\mathbb{Q}\)-Gorenstein smoothings for all listed analytic structures by deforming canonical-cover equations and lifting generators of the full first syzygy module. The numerical Milnor formula gives rational homology disk fibers.
\end{abstract}

\maketitle

\section{Introduction}\label{sec:introduction}

Let \((V,0)\) be a normal complex surface singularity. A \emph{rational homology disk smoothing} (QHD smoothing) of \((V,0)\) is a flat deformation \(\pi \colon \mathcal{V} \to \mathbb{D}\) over a disc \(\mathbb{D}\) with smooth nearby fibers whose Milnor fiber \(F\) satisfies \(H_i(F,\mathbb{Q})=0\) for every \(i>0\). Its Milnor number is therefore \(\mu=b_2(F)=0\), and the existence of such a smoothing implies that \((V,0)\) is rational by Durfee's formula. Wahl~\cite{Wahl1981} constructed cyclic quotient examples and further weighted homogeneous examples. Among cyclic quotient singularities, precisely those of type \(\frac1{p^2}(1,pq-1)\), with \(p>q>0\) and \(\gcd(p,q)=1\), admit such smoothings.

Stipsicz--Szab\'{o}--Wahl~\cite{SSW} obtained necessary conditions on the dual graphs of the minimal resolutions by applying Donaldson's diagonalization theorem and classifying lattice embeddings satisfying the adjunction equalities. Their classification separates the possible embeddings into cases and proceeds by induction on the number of vertices. Bhupal--Stipsicz~\cite{BS} then classified the noncyclic star-shaped resolution graphs by studying symplectic fillings of the links and applying McDuff's theorem to configurations of pseudoholomorphic curves. Their argument treats the remaining graph families through successive blow-downs to plane curve configurations, with reverse blow-up constructions verifying existence through Pinkham's negative-weight smoothing method. Fowler~\cite{Fowler} determined the corresponding analytic structures by studying configurations of rational curves on projective rational surfaces arising from Pinkham's compactification, using Riemann--Roch and nef divisors to locate the \((-1)\)-curves needed for their blow-downs; cf.~Wahl~\cite{Wahl2011}. More recently, the authors, jointly with Beke and Stipsicz~\cite{BPSS2026} and separately in~\cite{PS}, proved by different methods that every normal complex surface singularity admitting a QHD smoothing is weighted homogeneous. Combined with the results above, these theorems complete the analytic classification of normal complex surface singularities admitting QHD smoothings.

Our approach combines canonical-cover cohomology, specialization of reflexive canonical powers, and projective compactifications. It gives explicit equations for the Seifert invariants, characterwise restrictions on canonical-cover cohomology, and an elliptic-pencil description of the four-arm attachment points. Equivariant polynomial presentations of the canonical covers and their deformations include lifts of generators of the full first syzygy module. Together these arguments recover the known dual graph and analytic classification, including the existence of the QHD smoothings.

The following obstruction can be stated without reference to the classification. The existence of the canonical cover follows from the strong \(\mathbb{Q}\)-Gorenstein property established in Park--Shin~\cite{PS}.

\begin{maintheorem}[Characters of the canonical cover; Theorem~\ref{obs:equivariant-hodge}]\label{intro:character-obstruction}
Let \((V,0)\) be a normal complex surface singularity admitting a rational homology disk smoothing. Let \(B\) be its canonical cover, with group \(G=\boldsymbol{\mu}_r\), and put \(\mathcal{P}=H^1(X_B,\mathcal{O}_{X_B})\) for a \(G\)-equivariant resolution over a sufficiently small representative. Then \(\mathcal{P}\oplus\overline{\mathcal{P}}\) is isomorphic to a direct summand of \(\mathbb{C}[G]/\mathbb{C}\) as a complex \(G\)-representation. Equivalently, for any faithful character \(\chi\), the multiplicities \(p_j\) of \(\chi^j\) in \(\mathcal{P}\) satisfy
\[
p_j+p_{-j}\le1-\delta_{j,0}\qquad(j\in\mathbb{Z}/r\mathbb{Z}).
\]
\end{maintheorem}

Thus each character occurs at most once, and a character and its inverse cannot both occur. Summing the inequalities gives \(2p_g(B)\le r-1\), but the individual inequalities retain information lost in this sum. Example~\ref{obs:hodge-example} gives a rational Seifert datum satisfying \(K^2+N=0\), the square-determinant condition, and this genus bound, but violating the character restriction. The proof of Theorem~\ref{intro:character-obstruction} uses the structural theorem of Park--Shin and does not use a classification of the resolution graphs.

For a noncyclic rational weighted homogeneous singularity, let \(t\) denote the number of arms in its minimal good resolution and \(-d\) the self-intersection of the central component. Tables~\ref{table:threearm} and~\ref{table:fourarm} list the Seifert invariants and arm continued fractions of the noncyclic families in the following classification.

\begin{maintheorem}[Classification]\label{completion:classification}
The normal complex surface singularities that admit a rational homology disk smoothing are precisely the following:
\begin{enumerate}
\item the cyclic quotient singularities
\[
 \frac1{p^2}(1,pq-1), p>q>0, \gcd(p,q)=1,
\]

\item the weighted homogeneous singularities determined by the taut graphs in Table~\ref{table:threearm},

\item the weighted homogeneous singularities determined by the dual graphs and the normalized cross ratio in Table~\ref{table:fourarm}.
\end{enumerate}
\end{maintheorem}

\begin{proof}
The cyclic case is recalled above. For noncyclic singularities, Theorem~\ref{foundation:PS} gives rationality and weighted homogeneity, and Theorem~\ref{thm:arm-weight-bounds} reduces the problem to six pairs \((t,d)\). Theorems~\ref{d2:classification}, \ref{d3:classification}, \ref{d4:classification}, and~\ref{four:classification} determine the graphs; Corollary~\ref{cor:three-arm-taut} and Theorem~\ref{thm:cross-ratio} determine the analytic structures. These results are proved below. Theorem~\ref{cqg:main} gives sufficiency.
\end{proof}

\begin{table}[p]
\centering
\fontsize{9}{10}\selectfont
\renewcommand{\arraystretch}{0.97}
\setlength{\tabcolsep}{5pt}
\setlength{\aboverulesep}{1.5pt}
\setlength{\belowrulesep}{1.5pt}
\setlength{\defaultaddspace}{3pt}
%
\caption{The three-arm families.}\label{table:threearm}
\vspace{-8pt}
\begin{tabular}{@{}clll@{}}
\toprule
\((t,d)\) & Seifert invariant \((n_i,a_i)\) & Arm & BS \\
\midrule
\multirow{3}{*}{\((3,4)\)}
  & \(((q+1)(p+2)+1,\ q(p+2)+1)\) & \([2^q,p+3]\) & \multirow{3}{*}{1(a)} \\
  & \(((r+1)(q+2)+1,\ r(q+2)+1)\) & \([2^r,q+3]\) & \\
  & \(((p+1)(r+2)+1,\ p(r+2)+1)\) & \([2^p,r+3]\) & \\
\midrule
\multirow{3}{*}{\((3,3)\)}
  & \((p+3,\ 1)\)                  & \([p+3]\)          & \multirow{3}{*}{1(b)} \\
  & \(((r+1)(q+3)+1,\ r(q+3)+1)\) & \([2^r,q+4]\)      & \\
  & \((n_b,\ n_b-h_b)\)            & \([2^q,3,2^p,r+3]\) & \\
\addlinespace
  & \((2,\ 1)\)                    & \([2]\)       & \multirow{3}{*}{1(c)} \\
  & \(((q+1)(r+3)+1,\ q(r+3)+1)\) & \([2^q,r+4]\) & \\
  & \(((r+1)(q+3)+1,\ r(q+3)+1)\) & \([2^r,q+4]\) & \\
\midrule
\multirow{3}{*}{\((3,2)\)}
  & \((2,\ 1)\)          & \([2]\)             & \multirow{3}{*}{1(d)} \\
  & \((r+4,\ 1)\)        & \([r+4]\)           & \\
  & \((n_d,\ n_d-h_d)\)  & \([2^q,3,2^r,q+5]\) & \\
\addlinespace
  & \((3,\ 1)\)          & \([3]\)             & \multirow{3}{*}{1(e)} \\
  & \((p+3,\ 1)\)        & \([p+3]\)           & \\
  & \((n_e,\ n_e-h_e)\)  & \([2^q,4,2^p,q+4]\) & \\
\addlinespace
  & \((2,\ 1)\)                    & \([2]\)       & \multirow{3}{*}{1(f)} \\
  & \((3,\ 1)\)                    & \([3]\)       & \\
  & \((q^2+6q+6,\ q^2+5q+1)\)      & \([2^q,q+6]\) & \\
\addlinespace
  & \((r+4,\ 1)\)        & \([r+4]\)                 & \multirow{3}{*}{1(g)} \\
  & \((p+3,\ 1)\)        & \([p+3]\)                 & \\
  & \((n_g,\ n_g-h_g)\)  & \([2^q,3,2^r,3,2^p,q+4]\) & \\
\addlinespace
  & \((4,\ 1)\)                        & \([4]\)         & \multirow{3}{*}{1(h)} \\
  & \((4,\ 1)\)                        & \([4]\)         & \\
  & \((2(q+2)^2,\ 2(q+1)(q+2)-1)\)     & \([2^q,3,q+3]\) & \\
\addlinespace
  & \((6,\ 1)\)                        & \([6]\)           & \multirow{3}{*}{1(i)} \\
  & \((3,\ 1)\)                        & \([3]\)           & \\
  & \(((q+2)^2+1,\ (q+1)(q+2)+1)\)     & \([2^{q+1},q+3]\) & \\
\addlinespace
  & \((2,\ 1)\)                        & \([2]\)           & \multirow{3}{*}{1(j)} \\
  & \((6,\ 1)\)                        & \([6]\)           & \\
  & \(((q+2)(q+3)+1,\ (q+2)^2)\)       & \([2^{q+1},q+4]\) & \\
\bottomrule
\end{tabular}
\vspace{4pt}
%
\caption{The four-arm families.}\label{table:fourarm}
\vspace{-8pt}
\begin{tabular}{@{}clllc@{}}
\toprule
\((t,d)\) & Seifert invariant \((n_i,a_i)\) & Arm & Cross ratio \(\lambda\) & BS \\
\midrule
\multirow{4}{*}{\((4,3)\)}
  & \((3,\ 1)\)                    & \([3]\)         & \multirow{4}{*}{\(\lambda^2-\lambda+1=0\)} & \multirow{4}{*}{2(a)} \\
  & \((3,\ 1)\)                    & \([3]\)         & & \\
  & \((3,\ 1)\)                    & \([3]\)         & & \\
  & \((3p^2+9p+7,\ 3p^2+6p+2)\)    & \([2^p,4,p+2]\) & & \\
\midrule
\multirow{4}{*}{\((4,3)\)}
  & \((2,\ 1)\)                    & \([2]\)         & \multirow{4}{*}{\(\lambda=1/2\)} & \multirow{4}{*}{2(b)} \\
  & \((4,\ 1)\)                    & \([4]\)         & & \\
  & \((4,\ 1)\)                    & \([4]\)         & & \\
  & \((2p^2+6p+5,\ 2(p+1)^2)\)     & \([2^p,3,p+2]\) & & \\
\midrule
\multirow{4}{*}{\((4,3)\)}
  & \((2,\ 1)\)                    & \([2]\)         & \multirow{4}{*}{\(\lambda=9\)} & \multirow{4}{*}{2(c)} \\
  & \((3,\ 1)\)                    & \([3]\)         & & \\
  & \((6,\ 1)\)                    & \([6]\)         & & \\
  & \((p^2+3p+3,\ (p+1)^2+1)\)     & \([2^p,2,p+2]\) & & \\
\bottomrule
\end{tabular}
%
\par\vspace{6pt}
\begin{minipage}{\linewidth}
\fontsize{8.5}{9.5}\selectfont
\setlength{\parindent}{0pt}
\setlength{\abovedisplayskip}{3pt}
\setlength{\belowdisplayskip}{0pt}
\setlength{\abovedisplayshortskip}{3pt}
\setlength{\belowdisplayshortskip}{0pt}
\textit{Notes.} The parameters \(p,q,r\) are nonnegative integers. Each line lists one Seifert invariant and the corresponding arm; arms are read from the central component outward and may be permuted. In negative continued fractions, \(2^q\) denotes \(q\) consecutive entries equal to \(2\), omitted when \(q=0\); thus \([2^3,b]=[2,2,2,b]\) and \([2^0,b]=[b]\). The BS column identifies the corresponding graph in Figures~1 and~2 of Bhupal--Stipsicz~\cite{BS}. The cross ratio in Table~\ref{table:fourarm} is normalized as in Theorem~\ref{thm:cross-ratio}. For the families 1(b), 1(d), 1(e), and 1(g), put
\[
\begin{aligned}
n_b&=(q+2)h_b-r-2,              & h_b&=(p+2)(r+2)+1,\\
n_d&=(q+2)h_d-q-4,              & h_d&=(r+2)(q+4)+1,\\
n_e&=(q+1)h_e+(p+1)(q+3)+1,     & h_e&=(2p+3)(q+3)+2,\\
n_g&=(q+2)h_g-(p+2)(q+3)-1,     & h_g&=(r+3)((p+2)(q+3)+1)-q-3.
\end{aligned}
\]
\end{minipage}
\end{table}

\afterpage{\clearpage}

\subsection{Canonical corrections and character restrictions}

The structural theorem of Park--Shin~\cite[Theorem~10.1 and Corollary~10.2]{PS}, recalled in Section~\ref{sec:foundations}, supplies rationality, weighted homogeneity, and the strong \(\mathbb{Q}\)-Gorenstein property. Its proof does not use the graph classifications of Stipsicz--Szab\'{o}--Wahl or Bhupal--Stipsicz.

Let \(R\) be the total local algebra, \(s\) its parameter, and \(A=R/(s)\). For their common index \(r\) and canonical character \(\kappa\), the specialization cokernels \(C_j=\operatorname{coker}(\omega_R^{[j]}/s\omega_R^{[j]}\to\omega_A^{[j]})\) form the representation \(\mathcal{C}=\bigoplus_{j=0}^{r-1}C_j\otimes\mathbb{C}_{\kappa^{-j}}\). The specialization formula of Park--Shin and the strong \(\mathbb{Q}\)-Gorenstein property give
\[
\dim_{\mathbb{C}}C_j=c_A(jK_A)=0\qquad(j\in\mathbb{Z}).
\]
Proposition~\ref{obs:CSRC} gives this vanishing for every rational surface singularity admitting a strong \(\mathbb{Q}\)-Gorenstein smoothing. Proposition~\ref{obs:correction-identification} evaluates the corrections as the quantities \(\gamma_j\) in \eqref{obs:canonical-correction}, turning specialization into a finite system of explicit equations for the Seifert invariants.

For \(\chi=\kappa^{-1}\), formula~\eqref{obs:genus-character} computes the multiplicities \(p_j\) in Theorem~\ref{intro:character-obstruction}. Both \(p_j\) and \(\gamma_j\) are independent of the attachment points. For three arms and central weight two, the numerical identities, canonical-correction equations, and character inequalities determine the normalized negative-definite Seifert data without assuming the existence of a smoothing; see Theorem~\ref{d2:classification}. These conditions force exactly two arms to have one component, apart from \((2;(2,1),(3,1),(6,1))\). Before classifying the graphs, the same formulas show that every noncyclic QHD singularity has a nonrational canonical cover and index at least three; see Corollary~\ref{obs:noncyclic-cover}.

\subsection{Compactifications and attachment points}

Pinkham's compactification gives a smooth projective general fiber with boundary \(C+D\), where \(C\cong\mathbb{P}^1\), \(C^2=d-t\), and \(D\) consists of the complementary Hirzebruch--Jung chains. Contracting \(D\) gives a rational homology projective plane. Hwang--Keum's theorem, Artin's criterion, and \(|C|\) give the bounds \(3\le t\le4\), \(t-1\le d\le t+1\) in Theorem~\ref{thm:arm-weight-bounds}.

For \((t,d)=(3,3)\) and \((3,4)\), respectively, \(|C|\) gives a ruling or a birational morphism to \(\mathbb{P}^2\); the fibers or line configurations determine the arms. Three-arm tautness follows from Donaldson's theorem and structural lattice lemmas of Stipsicz--Szab\'{o}--Wahl, without their graph lists. For four arms, the lattice reduction and Belousov--Hwang~\cite{BH} give \(d=3\) and Table~\ref{table:fourarm}.

The four-arm analytic restriction has a further interpretation. Contracting the parameter-dependent part of the complementary boundary leaves the marked central curve unchanged and gives a weak del Pezzo surface of degree one. Blowing up the base point of its anticanonical pencil gives a rational elliptic surface with the respective configurations of singular fibers
\[
I_3^4,\qquad I_4^2I_2^2,\qquad I_6I_3I_2I_1.
\]
Its section identifies the attachment points with the singular parameters of the pencil. Beauville's classification~\cite{Beauville} then gives the cross ratios of Theorem~\ref{thm:cross-ratio}. These values were determined by Fowler~\cite{Fowler}; the elliptic-pencil construction explains why they are independent of the graph parameter. Every strong \(\mathbb{Q}\)-Gorenstein smoothing of a weighted homogeneous singularity with a graph in Tables~\ref{table:threearm}--\ref{table:fourarm} is QHD. Thus the four-arm cross-ratio conditions characterize strong \(\mathbb{Q}\)-Gorenstein smoothability itself; see Corollary~\ref{completion:strong-smoothability}.

\subsection{Explicit equivariant constructions}

The existence argument supplies defining equations, group actions, and lifted relations for every parameter value in the tables.

\begin{maintheorem}[Equivariant constructions]\label{intro:equivariant-constructions}\label{cqg:main}
Every analytic structure listed in Tables~\ref{table:threearm} and~\ref{table:fourarm}, with the prescribed attachment points, admits a strong \(\mathbb{Q}\)-Gorenstein rational homology disk smoothing obtained by deforming polynomial defining equations of its canonical cover and lifting generators of their full first syzygy module.
\end{maintheorem}
\begin{proof}
Propositions~\ref{cqg:w}, \ref{cqg:four:main}, \ref{cqg:i:existence}, \ref{cqg:j:existence}, and~\ref{cqg:m:existence}, proved below, construct equivariant deformations of the canonical covers for all listed analytic structures. They verify the common criteria of Section~\ref{cqg:section}, including the polynomial syzygy lifts. Corollary~\ref{cqg:qhd-verification} gives the strong \(\mathbb{Q}\)-Gorenstein rational homology disk smoothings.
\end{proof}

Table~\ref{table:smoothing-covers} records the finite Gorenstein surface covers used in the constructions, including the classical cyclic case.

\begin{table}[p]
\centering
\fontsize{9}{10}\selectfont
\renewcommand{\arraystretch}{0.97}
\setlength{\tabcolsep}{5pt}
\setlength{\aboverulesep}{1.5pt}
\setlength{\belowrulesep}{1.5pt}
\setlength{\defaultaddspace}{3pt}
\caption{Surface covers used in the QHD smoothing constructions.}\label{table:smoothing-covers}
\vspace{-8pt}
\begin{tabular}{@{}clll@{}}
\toprule
Family & Surface cover \(W\) & \(\Gamma\) & Construction \\
\midrule
Cyclic & \(\{xy-z^p=0\}\) & \(\mu_p\) & \shortstack[l]{Classical quotient\\\S\,\ref{cqg:general-foundation}} \\
\midrule
1(a) & \(\{xy^a+yz^b+zx^c=0\}\) & \(\mu_L\) & \shortstack[l]{Invariant function\\\S\,\ref{cqg:wn:section}} \\
\addlinespace
1(b),(c) & \(\left\{\begin{gathered}x^b+z^c-yw=0,\\xw-y^az=0\end{gathered}\right\}\) & \(\mu_L\) & \shortstack[l]{Invariant function\\\S\,\ref{cqg:wn:section}} \\
\addlinespace
1(d)--(g) & \(\{x^P+y^Q+z^n=0\}/H\) & \(\mu_N\) & \shortstack[l]{Adjoining fractions\\\S\,\ref{cqg:m}} \\
\addlinespace
1(h) & \(\left\{\begin{gathered}zw+x^{2m}+\zeta y^{2m}=0,\\xw+yz=0\end{gathered}\right\}\) & \(\Gamma_h\) & \shortstack[l]{Invariant function\\\S\,\ref{cqg:wn:section}} \\
\addlinespace
1(i) & \(\{u^3+v^3+vz^n=0\}/\mu_m(1,1,1)\) & \(\mu_{3m}\) & \shortstack[l]{Regular sequence\\\S\,\ref{cqg:i:section}} \\
\addlinespace
1(j) & \(\{u^2+v^4+vw^n=0\}/\mu_m(2,1,1)\) & \(\mu_{2m}\) & \shortstack[l]{Regular sequence\\\S\,\ref{cqg:j:section}} \\
\midrule
2(a)--(c) & \(\{f_n=0\}\subset Z_n,\quad n=3,4,6\) & \(\Gamma_{n,P}\) & \shortstack[l]{Invariant function\\\S\,\ref{cqg:four:section}} \\
\bottomrule
\end{tabular}
\par\vspace{6pt}
\begin{minipage}{\linewidth}
\fontsize{8.5}{9.5}\selectfont
\setlength{\parindent}{0pt}
\setlength{\abovedisplayskip}{3pt}
\setlength{\belowdisplayskip}{3pt}
\setlength{\abovedisplayshortskip}{3pt}
\setlength{\belowdisplayshortskip}{3pt}
\setlength{\jot}{1pt}
\textit{Notes.} Each row has \(V=W/\Gamma\); the group actions and the deformations are specified in the indicated constructions. The notation \(\mu_m(c_1,\ldots,c_e)\) denotes a diagonal action with those characters. The cyclic row has \(p>q>0\), \(\gcd(p,q)=1\), and action \(\mu_p(1,-1,q)\). In the other rows, \(p,q,r\) are the parameters of Tables~\ref{table:threearm} and~\ref{table:fourarm}.

For 1(a), put \((a,b,c)=(p+2,q+2,r+2)\) and \(L=abc+1\). For 1(b),(c), put \((b,c)=(q+2,r+2)\), with \(a=p+2\) and \(a=1\), respectively, and \(L=a(b+1)c+b\). The arms may be permuted. For 1(h),(i),(j), put \(m=q+2,q+3,q+4\), respectively. In 1(i), take \(n=m^2-2m+2\); in 1(j), take \(n=m^2-3m+3\). In 1(h), \(\zeta=\omega^m\) for a primitive \(2m(m+1)\)-th root \(\omega\), and \(|\Gamma_h|=4m(m+1)\).

For 1(d)--(g), the parameters and the auxiliary quotient are
\[
\begin{array}{c|ccc}
\text{row}&P&Q&k\\\hline
1(d)&r+4&2&q+3\\
1(e)&3&p+3&q+3\\
1(f)&3&2&q+3\\
1(g)&r+4&p+3&q+3
\end{array}
\qquad
\begin{gathered}
\ell=PQ-P-Q,\quad N=P+k\ell,\\
s=kQ-k+1,\quad n=(k-1)N-kQ,\\
H=\mu_N(-s,-k,1).
\end{gathered}
\]
For 2(a)--(c), take \(P=p+2\) and \((n,q)=(3,3),(4,2),(6,1)\), respectively. Put \(d=qP\), \(m=qP(P-1)+1\), and choose \(\zeta\) of order \(n\) as in Section~\ref{cqg:four:section}. The ambient threefolds \(Z_n=\Spec D_n\) and the invariant functions are
\[
\begin{gathered}
D_3=\mathbb{C}[x_1,x_2,x_3],\quad
D_4=\mathbb{C}[x_1,x_2,x_3,x_4]/(x_1x_3-x_2x_4),\\
D_6=\mathbb{C}[x_1,\ldots,x_6,z]/I_6,\quad
I_6=(x_ix_{i+2}-x_{i+1}z\ (i\bmod6),\ x_ix_{i+3}-z^2\ (1\le i\le3)),\\
f_n=\sum_{j=0}^{n-1}\zeta^j x_{j+1}^{d-1}x_{j+2},\qquad |\Gamma_{n,P}|=nqmP.
\end{gathered}
\]
The indices on \(x_i\) are cyclic. For \(n=6\), the surface \(W\) is a Cartier section of the Gorenstein toric threefold \(Z_6\). In 1(i),(j), \(W\) is the canonical cover itself; the construction deforms the displayed quotient surface. In 1(d)--(g), it deforms the \(H\)-quotient by adjoining fractions.
\end{minipage}
\end{table}
\clearpage

Each singularity in Table~\ref{table:smoothing-covers} is \(V=W/\Gamma\), where \(W\) admits a \(\Gamma\)-equivariant smoothing \(\mathcal{W}\to\mathbb{D}\) with free action on noncentral fibers and QHD quotient. Thus
\[
\chi(\mathcal{W}_t)=|\Gamma|,\qquad \mu(\mathcal{W}_t)=|\Gamma|-1,
\]
by multiplicativity of Euler characteristic and the vanishing of the first rational homology and homology above degree two of a normal surface smoothing fiber. The canonical cover is \(W/\ker\kappa\), where \(\kappa\) is the character of a canonical generator; see Corollary~\ref{cqg:larger-cover}. The table includes canonical and larger Gorenstein covers.

The constructions in Table~\ref{table:smoothing-covers} deform an invariant function, a regular sequence in a toric algebra, or successive adjunctions of fractions. Section~\ref{cqg:section} supplies the common algebraic steps, including lifts of generators of the full first syzygy module. For the invariant-function families we use known quotient equations and graphs and compute their canonical characters and invariant presentations. For families \(1(d)\)--\(1(g)\), \(1(i)\), and \(1(j)\), we also identify the central canonical-cover algebras. Their equivariant deformations use fraction adjunctions followed by passage to invariants, or regular sequences in toric algebras. For every parameter value, the central-fiber identifications, syzygy lifts, and smoothness and freeness checks verify the constructions without invoking earlier existence conclusions.

\subsection*{Organization of the paper}

Sections~\ref{sec:foundations}--\ref{sec:three-arm-tautness} give the structural input, arm and weight bounds, and three-arm tautness. Section~\ref{sec:classification} proves the character obstruction and computes the canonical corrections. Sections~\ref{sec:three-arm-classification}--\ref{sec:cross-ratios} determine the graphs and attachment points. Section~\ref{cqg:section} gives the common deformation arguments, applied in Sections~\ref{cqg:invariant-sections}--\ref{cqg:m}. Tables~\ref{table:threearm} and~\ref{table:fourarm} agree with Figures~1 and~2 of Bhupal--Stipsicz~\cite{BS}; their graph lists are used for comparison, not in the proofs of necessity.

\subsection*{Statements and Declarations}

\begin{itemize}[labelindent=0pt, leftmargin=*, align=left]
\item \textbf{Funding.} Heesang Park was supported by the Basic Science Research Program through the National Research Foundation of Korea (NRF) (RS-2026-25590534). Dongsoo Shin was supported by the Basic Science Research Program through the National Research Foundation of Korea (NRF) (NRF-2023R1A2C1007293).


\item \textbf{Use of artificial intelligence.} The authors used ChatGPT-6 Astra for assistance with the constructions using deformations of equations and lifts of their relations, checking calculations, and language editing. The authors take full responsibility for the correctness of its contents.
\end{itemize}

\section{Properties of rational homology disk smoothings}\label{sec:foundations}

Let \(\pi: (X,E) \to (V,0)\) be the minimal good resolution of a normal complex surface singularity \((V,0)\), let \(E=\sum_v E_v\) be the reduced exceptional divisor, and put \(S_X=\Theta_X(-\log E)\). Write \(E_v^2=-b_v\), let \(N\) be the number of components of \(E\), and denote the numerical canonical cycle by \(K\).

\begin{theorem}[Park--Shin, {\cite[Theorem~10.1 and its proof, and Corollary~10.2]{PS}}]\label{foundation:PS}
Let \((V,0)\) be a normal complex surface singularity which is not smooth. If \((V,0)\) admits a rational homology disk smoothing, then it is rational, non-Gorenstein, and weighted homogeneous. Every such smoothing is strong \(\mathbb{Q}\)-Gorenstein and occurs over a one-dimensional component of the semiuniversal deformation base.
\end{theorem}

For a QHD smoothing as in Theorem~\ref{foundation:PS}, every reflexive canonical power on the total space specializes to the corresponding power on the central fiber. The canonical indices agree, and the smoothing lifts to an equivariant smoothing of the index-one cover on which the covering group acts trivially on the base. Since \((V,0)\) is rational, each \(E_v\) is a smooth rational curve and \(K\cdot E_v=b_v-2\). Wahl's smoothing-component formula \cite[Theorem~8.1]{Wahl2011} and the numerical Milnor formula of Park--Shin~\cite[Theorem~8.3 and Corollary~8.4]{PS} give
\begin{equation}\label{foundation:dimension-eq}
1=h^1(S_X)+\sum_v(b_v-3),\quad K^2+N=0.
\end{equation}

More generally, let a rational normal surface singularity of canonical index \(r\) admit a strong \(\mathbb{Q}\)-Gorenstein smoothing. Apply Park--Shin~\cite[Theorem~8.3]{PS} with \(i=r\) if \(r>1\), and with \(i=2\) if \(r=1\). The specialization defect and the local correction term vanish, so the numerical Milnor formula gives
\begin{equation}\label{foundation:strong-milnor}
\mu=K^2+N.
\end{equation}
Its Milnor fiber has vanishing first rational homology by Greuel--Steenbrink~\cite{GreuelSteenbrink1983} and the homotopy type of a finite complex of dimension at most two. Thus \(K^2+N=0\) suffices to conclude that such a smoothing is QHD.

\section{Compactification and the six noncyclic cases}\label{sec:bounds}

We use Pinkham's compactification and its boundary curves to bound the number of arms and the central weight.

\subsection{Seifert invariants and complementary chains}\label{sec:notation}
Assume that \((V,0)\) is noncyclic and admits a QHD smoothing. By Theorem~\ref{foundation:PS}, its minimal good resolution graph is a star-shaped tree of smooth rational curves. Write \(E_0^2=-d\) for the central self-intersection, and let the \(i\)-th arm have weights \(-b_{i1},\ldots,-b_{i\ell_i}\), ordered from the central component outwards. We use the negative continued fraction
\[
\frac{n_i}{a_i}=[b_{i1},\ldots,b_{i\ell_i}]
=b_{i1}-\frac{1}{b_{i2}-\dfrac{1}{\ddots-\dfrac{1}{b_{i\ell_i}}}},
\quad 0<a_i<n_i,\quad \gcd(n_i,a_i)=1.
\]
Every \(b_{ij}\) is at least two. The Seifert data are denoted by \((d;(n_1,a_1),\ldots,(n_t,a_t))\). An arm has one vertex if and only if \(a_i=1\). In a continued fraction, \(2^m\) denotes a string of \(m\) entries equal to two; \(2^0\) is the empty string.

Let \(a_i^*\) be the integer in \(\{1,\ldots,n_i-1\}\) satisfying \(a_i a_i^*\equiv1\pmod{n_i}\). We put
\begin{equation*}
e=d-\sum_i\frac{a_i}{n_i},\quad
\eta=t-2-\sum_i\frac{1}{n_i}.
\end{equation*}
Negative definiteness is equivalent to \(e>0\).

The complementary chain to \(n_i/a_i\) has continued fraction \(n_i/(n_i-a_i)\). The components of each chain are ordered starting at the end adjacent to the central curve. Reversing this order replaces the denominator by its inverse modulo the numerator. Put \(W([c_1,\ldots,c_s])=\sum_j(c_j-3)\). Hirzebruch--Jung duality gives
\begin{equation*}
W\left(\frac{n}{a}\right)+W\left(\frac{n}{n-a}\right)=-2.
\end{equation*}
This follows from the dot diagram for complementary continued fractions.

\subsection{Negative-weight smoothings and projective compactifications}\label{sec:compactification}
All weights are function weights; a positive parameter weight gives a negative-weight deformation.

\begin{lemma}\label{cap:positive-weight}
Let \((V,0)\) be a normal complex surface singularity which is not smooth and has a good \(\mathbb{C}^*\)-action. The parameter of an equivariant one-parameter smoothing of \((V,0)\) has positive weight.
\end{lemma}

\begin{proof}
Let \(A\) be the normal graded coordinate ring, with \(A_0=\mathbb{C}\). Any two distinct members \(x,y\) of a minimal homogeneous generating set are algebraically independent. Otherwise the kernel of \(\mathbb{C}[X,Y]\to A\) is generated by an irreducible weighted homogeneous polynomial, which has the form \(X^p-Y^q\) after scaling, with \(p,q\) relatively prime. Choose integers \(u,v\) with \(qu+pv=1\). Then \(h=x^uy^v\) satisfies \(h^q=x\) and \(h^p=y\). Normality implies \(h\in A\), and its weight is positive. Unless \(p=q=1\), one of \(x,y\) belongs to \(A_+^2\). If \(p=q=1\), they are linearly dependent. Either conclusion contradicts minimality.

Order the minimal generator weights as \(a_1\le\cdots\le a_e\). Since the surface is singular, \(e\ge3\). Every defining relation has degree greater than \(a_3\). Indeed, minimality excludes linear terms, so a relation of degree at most \(a_3\) could involve only the first two generators, contrary to their algebraic independence. Flatness, reductivity of \(\mathbb{C}^*\), and Nakayama's lemma give homogeneous lifted equations \(F_j(X,s)\) of degrees \(D_j>a_3\). Suppose that \(w=\operatorname{wt}(s)\le0\). No monomial \(s^k\) can occur in \(F_j\), so \(X=0\) defines a section. A term \(s^kX_i\) would require \(a_i=D_j-kw>a_3\). The first three columns of the relative Jacobian therefore vanish along the section. Every fiber has tangent dimension at least three there, contrary to the smoothness of the nearby surface fibers.
\end{proof}

By Theorem~\ref{foundation:PS}, the smoothing component is one-dimensional. The grading torus preserves this component and acts on its normalization, which is a smooth curve. Pullback to the normalization gives an equivariant smoothing with the same general Milnor fiber. Its parameter has positive weight by Lemma~\ref{cap:positive-weight}. Since the generator weights and parameter weight are positive, the homogeneous formal equations are polynomials.

Pinkham's construction \cite{Pinkham} compactifies this negative-weight smoothing with locally trivial boundary. If \(\mathcal{R}\) is the graded total algebra and \(u\) the parameter of weight \(w>0\), the compactified general fiber is
\[
\Proj(\mathcal{R}[s]/(u-s^w)),\quad\deg s=1.
\]
Its affine chart is the fiber at \(u=1\), with the homotopy type of the Milnor fiber. The boundary is \(\overline C=\Proj A\cong\mathbb{P}^1\), with cyclic quotient points of types \(\frac1{n_i}(1,n_i-a_i)\). Resolving these points gives a smooth projective surface \(Y\) with boundary \(C+D\), where \(D\) consists of the complementary chains \(n_i/(n_i-a_i)\), read from \(C\) outwards. Its affine complement \(U=Y\setminus(C\cup D)\) is a rational homology disk.

\subsubsection{The boundary intersection matrix}

Write \(N_D\) for the number of components of \(D\). Since \(U=Y\setminus(C\cup D)\) is a QHD, Poincar\'e duality gives \(H_c^j(U,\mathbb{Q})=0\) for \(j=1,2,3\). The long exact sequence for cohomology with compact support gives
\[
H^1(Y,\mathbb{Q})=0,\quad
H^2(Y,\mathbb{Q})\cong H^2(C\cup D,\mathbb{Q})\cong\mathbb{Q}^{N_D+1}.
\]
Thus \(q(Y)=0\) and \(b_2(Y)=N_D+1\). The boundary intersection matrix is nondegenerate, since its Schur complement at \(C\) is \(e>0\) and every complementary chain is negative definite. Its \(N_D+1\) component classes therefore form a basis of \(H^2(Y,\mathbb{Q})\), and \(\rho(Y)=b_2(Y)=N_D+1\).

The rational self-intersection at infinity is \(\overline C^2=e=d-\sum_i a_i/n_i\). Resolving the singularity with chain \(n_i/(n_i-a_i)\) subtracts \((n_i-a_i)/n_i\) from the self-intersection of the central curve. Consequently
\[
C^2=e-\sum_i\frac{n_i-a_i}{n_i}=d-t.
\]
For each \(i\), the end component of \(D_i\) adjacent to \(C\) meets \(C\) transversely once. These \(t\) components are pairwise disjoint, and all other components of \(D\) are disjoint from \(C\).

\subsection{The number of arms and the central weight}

\begin{theorem}
\label{thm:arm-weight-bounds}
Let \((V,0)\) be a noncyclic rational weighted homogeneous normal surface singularity admitting a QHD smoothing on a one-dimensional smoothing component. Let its minimal good resolution have \(t\) arms and central self-intersection \(-d\). Then
\[
3\le t\le4,\quad t-1\le d\le t+1.
\]
Thus the possible pairs \((t,d)\) for a noncyclic singularity are
\[
(t,d)=(3,2),(3,3),(3,4),(4,3),(4,4),(4,5).
\]
\end{theorem}

\begin{proof}
A noncyclic rational weighted homogeneous singularity has a star-shaped minimal good resolution with at least three arms. The assertion follows from Lemmas~\ref{lem:lower-central-bound}, \ref{lem:four-arms-bound}, and~\ref{lem:upper-central-bound}, which are proved below.
\end{proof}

Contracting \(D\) gives a projective surface \(\overline Y\) with \(t\) cyclic quotient singularities and the rational homology of \(\mathbb{P}^2\).

\begin{lemma}
\label{lem:lower-central-bound}
For a rational surface singularity with a star-shaped minimal good resolution as above, one has \(d\ge t-1\).
\end{lemma}
\begin{proof}
Let \(E\) be the reduced exceptional divisor of the minimal good resolution and let \(E_0\) be its central component. The divisor \(E\) is a tree of smooth rational curves, so \(p_a(E)=0\). The arithmetic genus addition formula gives
\[
p_a(E+E_0)=p_a(E)+p_a(E_0)+E\cdot E_0-1=t-d-1.
\]
Artin's numerical criterion for rationality \cite{Artin1966} says that every nonzero effective exceptional cycle has arithmetic genus at most zero. Applying it to \(E+E_0\) proves the assertion.
\end{proof}

\begin{lemma}
\label{lem:four-arms-bound}
Under the hypotheses of Theorem~\ref{thm:arm-weight-bounds}, one has \(t\le4\).
\end{lemma}
\begin{proof}
By Hwang--Keum~\cite[Corollary~1.2]{HwangKeum2011}, a rational homology projective plane with quotient singularities has at most four singular points, unless it has five rational double points of type \(3A_1+2A_3\) and its minimal resolution is an Enriques surface. Apply this result to \(\overline{Y}\). Since every component of \(D\) has self-intersection at most \(-2\), the morphism \(Y\to\overline{Y}\) is the minimal resolution. Lemma~\ref{lem:lower-central-bound} gives \(C^2\ge-1\). The canonical class of an Enriques surface is numerically trivial, so adjunction gives self-intersection \(-2\) for every smooth rational curve. Hence \(Y\) is not an Enriques surface, and \(t\le4\).
\end{proof}

\begin{lemma}\label{lem:rational-curve-system}
Let \(Y\) be a smooth projective surface with \(q(Y)=0\), and let \(C\cong\mathbb{P}^1\) be a curve with \(C^2=m\ge0\). Then \(h^0(Y,\mathcal{O}_Y(C))=m+2\), and \(|C|\) is base point free. For \(m=0\), it defines a \(\mathbb{P}^1\)-fibration over \(\mathbb{P}^1\). For \(m\ge1\), it defines a birational morphism onto a surface of degree \(m\) in \(\mathbb{P}^{m+1}\); in particular the image is \(\mathbb{P}^2\) when \(m=1\).
\end{lemma}
\begin{proof}
The restriction sequence
\[
0\to\mathcal{O}_Y\to\mathcal{O}_Y(C)\to\mathcal{O}_{\mathbb{P}^1}(m)\to0
\]
gives \(h^0(Y,\mathcal{O}_Y(C))=m+2\) and the surjectivity of restriction, since \(H^1(Y,\mathcal{O}_Y)=0\). The lifted sections generate along \(C\), and the section defining \(C\) does not vanish outside \(C\). Hence \(|C|\) is base point free. If \(m=0\), the reduced irreducible fiber \(C\) implies that the finite morphism in the Stein factorization has degree one. The general fiber is smooth of genus \(p_a(C)=0\). If \(m>0\), the image \(Z\subset\mathbb{P}^{m+1}\) is a nondegenerate surface, and
\[
m=\deg(\phi)\deg(Z),\quad \deg(Z)\ge m.
\]
Thus \(\deg(\phi)=1\) and \(\deg(Z)=m\).
\end{proof}

For a composition \(\sigma:Y\to Z\) of point blowups, we consider each positive-dimensional fiber with its reduced structure.

\begin{lemma}\label{lem:one-outside-cluster}
Let \(\sigma:Y\to Z\) be a composition of point blowups between smooth projective surfaces. Let \(D\) be a disjoint union of linear chains whose \(\sigma\)-exceptional components have self-intersection at most \(-2\). Every positive-dimensional exceptional fiber contains an exceptional component outside \(D\). Suppose smooth curves \(A,B\subset Z\) meet transversely at \(p\), their strict transforms belong to distinct connected components of \(D\), and the exceptional fiber over \(p\) has exactly one exceptional component \(E\) outside \(D\). Then the exceptional fiber, together with these strict transforms, has the linear dual graph
\[
A-L_A-E-L_B-B,
\]
where \(E^2=-1\) and \(L_A,L_B\) are possibly empty exceptional subchains of \(D\). If \(L_A\) is empty and \(L_B\) has length \(r\), then \(L_B=[2^r]\), and the blowups over this point lower the self-intersections of \(A,B\) by \(r+1,1\), respectively; the symmetric assertion also holds. In particular, if the strict transform of \(B\) contains at least two successive blowup centers over this point, then \(L_A\) is nonempty.
\end{lemma}
\begin{proof}
The last blowup creates an exceptional \((-1)\)-curve, which cannot belong to \(D\). At each stage, the exceptional \((-1)\)-curves are pairwise disjoint. A further blowup therefore destroys at most one such curve and creates one, so their number cannot decrease. If the final exceptional fiber has only one component outside \(D\), every blowup after the first must be centered on its unique exceptional \((-1)\)-curve. This curve must lie on the path between the strict transforms of \(A\) and \(B\). Once a blowup at a smooth point moves the unique \((-1)\)-curve off that path, subsequent centers remain on the resulting branch. The final path would then consist entirely of components of \(D\), contrary to the assumption that \(A\) and \(B\) belong to distinct connected components. Thus every subsequent center is an intersection on the path, and the dual graph remains linear. If \(L_A\) is empty, each subsequent center is the intersection of the strict transform of \(A\) with the unique \((-1)\)-curve. Induction gives \(A-E-[2^r]-B\) and decreases the self-intersections of \(A,B\) by \(r+1,1\), respectively. Interchanging \(A\) and \(B\) proves the symmetric assertion.
\end{proof}

\begin{lemma}
\label{lem:upper-central-bound}
Under the hypotheses of Theorem~\ref{thm:arm-weight-bounds}, one has \(C^2\le1\), and hence \(d\le t+1\).
\end{lemma}
\begin{proof}
Suppose that \(m=C^2\ge2\). Lemma~\ref{lem:rational-curve-system} gives a birational morphism \(\phi:Y\to Z\subset\mathbb{P}^{m+1}\) onto a surface of minimal degree. Each component of \(D\) meeting \(C\) maps isomorphically to a line, because its intersection with \(C\) is one; every other component of \(D\) is contracted.

By the classification of surfaces of minimal degree by Eisenbud--Harris~\cite{EisenbudHarris1987}, \(Z\) is a rational normal scroll, possibly a cone, or the Veronese surface. The Veronese surface contains no line, whereas \(Z\) contains the images of the components of \(D\) meeting \(C\). Thus \(Z\) is a scroll. Each morphism below is an isomorphism near \(C\), since the contracted curves are disjoint from \(C\).

Suppose first that \(Z\) is a smooth scroll. Write \(Z=\mathbb{F}_e\), with negative section \(T\) and fiber \(F\). Its hyperplane class is \(T+(e+a)F\), where \(a\ge1\). Except for the smooth quadric, its lines are the ruling fibers and, when \(a=1\), the section \(T\). Hence at most one component meeting \(C\) maps to a line which is not a fiber. Since \(\rho(Y)=N_D+1\) and \(\rho(Z)=2\), the morphism \(Y\to Z\) has \(N_D-1\) exceptional components. Exactly \(N_D-t\) of them belong to \(D\), so \(t-1\) exceptional components lie outside \(D\).

If all components meeting \(C\) map to fibers, their images are distinct, since they meet the image of \(C\) at distinct points. Each of these \(t\) fibers contains a point over which \(Y\to Z\) has a positive-dimensional fiber, because its strict transform has self-intersection at most \(-2\). Exceptional fibers over different fibers are disjoint. Lemma~\ref{lem:one-outside-cluster} gives at least \(t\) exceptional components outside \(D\), contrary to their number \(t-1\).

Otherwise one component meeting \(C\) maps to a section line and the remaining \(t-1\) components meeting \(C\) map to distinct fibers. Each intersection of the section with one of these fibers is blown up, since their strict transforms belong to disjoint chains on \(Y\). The \(t-1\) resulting exceptional fibers each contain an exceptional component outside \(D\). The number of these components implies that every exceptional fiber contains exactly one, and that there are no further exceptional fibers. Each fiber is blown up at least twice, since its final self-intersection is at most \(-2\). By Lemma~\ref{lem:one-outside-cluster}, a nonempty subchain of \(D\) is attached to the component mapping to the section over each of the \(t-1\) intersections. Since \(t\ge3\), this component meets at least two other components of \(D_i\), contrary to the fact that it is an end of its chain \(D_i\).

If \(Z\) is the smooth quadric, its lines belong to its two rulings. Since \(3\le t\le4\), one ruling contains at most one image of a component meeting \(C\), unless each ruling contains two. In the first case, the preceding arguments apply after choosing the ruling fibers. In the second, the four intersection points between lines of different rulings are blown up. They give four distinct exceptional fibers, and hence at least four exceptional components outside \(D\), contrary to their number \(t-1=3\).

Finally suppose that \(Z\) is a cone of degree \(m\). The morphism \(\phi\) factors through the minimal resolution \(\mathbb{F}_m\to Z\). On \(\mathbb{F}_m\), the image of \(C\) is linearly equivalent to \(T+mF\) and is disjoint from the negative section \(T\). Every component meeting \(C\) maps to a distinct fiber. Put \(\varepsilon=1\) if the strict transform of \(T\) belongs to \(D\), and put \(\varepsilon=0\) otherwise. The components of \(D\) not contracted by \(Y\to\mathbb{F}_m\) are precisely the \(t\) components meeting \(C\) and, when \(\varepsilon=1\), the strict transform of \(T\). Thus the number of exceptional components outside \(D\) is
\[
(N_D-1)-(N_D-t-\varepsilon)=t+\varepsilon-1.
\]
Each of these fibers contains a point over which \(Y\to\mathbb{F}_m\) has a positive-dimensional fiber. If \(\varepsilon=0\), Lemma~\ref{lem:one-outside-cluster} contradicts the number \(t-1\) of exceptional components outside \(D\). If \(\varepsilon=1\), over each of these ruling fibers there is exactly one positive-dimensional fiber of \(Y\to\mathbb{F}_m\), and each contains exactly one exceptional component outside \(D\), and there are no other exceptional fibers. The strict transform of \(T\) belongs to a unique chain \(D_i\), which contains a path from its end meeting \(C\) to \(T\). For \(j\ne i\), the strict transforms of \(T\) and the \(j\)-th fiber are disjoint, so their intersection is the point over which the unique positive-dimensional fiber lying above that ruling fiber is contracted. The fiber is blown up at least twice. Lemma~\ref{lem:one-outside-cluster} therefore gives a nonempty subchain of \(D_i\) attached to \(T\) over that intersection. Together with the path to the end of \(D_i\) meeting \(C\), these \(t-1\) distinct subchains give \(T\) valency at least \(t\ge3\) in \(D_i\), contrary to the linearity of \(D_i\).
\end{proof}

\section{Diagonal embeddings and star-shaped graphs}
\label{sec:three-arm-tautness}

An embedding of a tree with \(N\) vertices satisfying the adjunction equalities consists of vectors \(v\) in \((\mathbb{Z}^{N},-I)\), with square \(-b_v\le-2\), pairing one on edges and zero otherwise, and satisfying
\[
v^2+\mathbf{k}\cdot v=-2,\quad \mathbf{k}=\sum_{j=1}^{N}e_j.
\]
By Stipsicz--Szab\'{o}--Wahl~\cite[Proposition~2.4(2), (4)]{SSW}, every resolution graph of a singularity admitting a QHD smoothing admits such an embedding. The characteristic vector has square \(-N\), so a change of signs of the diagonal basis makes it equal to \(\mathbf{k}\).

We use the following coefficient properties from Stipsicz--Szab\'{o}--Wahl~\cite[Section~3]{SSW}. Each vertex vector, and the sum of the vectors over any connected subtree, has one of the forms
\begin{equation}\label{lat:subtree-forms}
e_i-\sum_{j\in J}e_j,\quad -2e_i-\sum_{j\in J}e_j,\quad i\notin J.
\end{equation}
These forms follow by writing the adjunction equation as \(\sum_jx_j(x_j+1)=2\). For the whole minimal tree, the first form is excluded. At most one vertex has the negative form, and each positive coefficient occurs in a unique vertex. For \(N>1\), the unique column of total coefficient \(-2\) has entries either \((-1,-1)\) or \((1,-1,-1,-1)\). We distinguish these two patterns by their lists of nonzero entries. Following Stipsicz--Szab\'{o}--Wahl~\cite[Section~3]{SSW}, a tree is full when its total vector has no zero coordinate, hence is \(-2e_i-\sum_{j\ne i}e_j\).

For a positive vertex \(v=e_i-\sum_{j\in J}e_j\), call \(v\) full if the total coefficient in column \(i\) is negative. We also use two lemmas. In the case of a column with nonzero entries \((1,-1,-1,-1)\), if every leaf is full, then the graph is full by Stipsicz--Szab\'{o}--Wahl~\cite[Proposition~5.2]{SSW}. If every vertex has a positive coefficient and some leaf is not full, then there is a vertex \(v=e_i+T\) with
\begin{equation}\label{lat:reducible}
T=-\sum_{j\in J}e_j\ne0,\quad i\notin J,\quad T\cdot w=0\quad(w\ne v)
\end{equation}
\cite[Lemma~6.6]{SSW}. These two lemmas are proved by coefficient calculations and induction on the number of vertices. Their proofs do not use the graph classification. The weight identity for star-shaped graphs is proved below.

\begin{lemma}\label{lat:minimal-definite}
A connected tree of integral vectors satisfying the adjunction equalities in a negative diagonal lattice, with all squares at most \(-2\), has negative definite intersection matrix.
\end{lemma}
\begin{proof}
The intersection matrix is negative semidefinite. If it were singular, connectedness and the Perron--Frobenius theorem would give a kernel vector with strictly positive entries \(m_v\). Then \(\sum_v m_vv=0\), and adjunction gives \(\sum_vm_v(b_v-2)=0\). Hence \(b_v=2\) for every \(v\). Formula~\eqref{lat:subtree-forms} shows that every vertex vector is a root \(e_i-e_j\). Two distinct roots cannot have the same positive index or the same negative index, since their pairing would be negative. Each root has at most two neighbors. Thus the tree is a chain of \((-2)\)-vertices, whose intersection matrix is negative definite, a contradiction.
\end{proof}

The following structural result is due to Stipsicz--Szab\'{o}--Wahl~\cite[Propositions~4.2 and~4.3]{SSW}. Its proof does not use their classification theorem.

\begin{lemma}\label{lat:double-column}
Let a minimal tree be embedded in a diagonal lattice so that the adjunction equalities hold. If one column has nonzero entries \((-1,-1)\), then the tree is a full chain. The two vertices containing that column are its ends.
\end{lemma}

\begin{samepage}
\begin{lemma}\label{lat:star-reduction}
Let a minimal star-shaped tree have \(t=3\) or \(4\) arms and central weight \(-d\), where \(d\ge t-1\). Suppose it is embedded in a diagonal lattice so that the adjunction equalities hold. Then
\[
\sum_v(b_v-3)=4-t.
\]
If \(t=4\), exactly three arms have one vertex, with orders
\[
(3,3,3),\quad (2,4,4),\quad (2,3,6).
\]
\end{lemma}
\end{samepage}
\begin{proof}
The graph is not a chain, so Lemma~\ref{lat:double-column} excludes a column with nonzero entries \((-1,-1)\). Let \(s\) be the column with nonzero entries \((1,-1,-1,-1)\). If the graph is full, its total vector is \(V=-2e_s-\sum_{j\ne s}e_j\). Thus \(V^2=-N-3\), and the identity \(V^2=-\sum_vb_v+2(N-1)\) gives \(\sum_v(b_v-3)=1\). Adjunction also gives
\[
\operatorname{val}(v)-2=(V+\mathbf{k})\cdot v=-e_s\cdot v=v_s.
\]
Every leaf has coefficient \(-1\) in column \(s\), which has exactly three negative entries. A full graph therefore has at most three leaves. This proves the weight identity when \(t=3\) and excludes a full graph when \(t=4\).

Suppose that the graph is not full. If every vertex has a positive coefficient, Stipsicz--Szab\'{o}--Wahl~\cite[Proposition~5.2 and Lemma~6.6]{SSW} give a vertex satisfying \eqref{lat:reducible}. Otherwise exactly \(N-1\) vertices have a positive coefficient, at distinct indices. The remaining coordinate is nonzero in some vertex, since the embedding has full rank, and has no positive entry. Its total coefficient is not \(-2\), since that coefficient already occurs in the column with nonzero entries \((1,-1,-1,-1)\). Hence it occurs in exactly one vertex \(v\), with coefficient \(-1\); denote this coordinate by \(e_j\). Removing \(-e_j\) preserves adjunction and the pairings between distinct vertices. If \(v^2\le-3\), the result would be \(N\) vectors of square at most \(-2\) satisfying the adjunction equalities in a lattice of rank \(N-1\), contrary to Lemma~\ref{lat:minimal-definite}. Thus \(v=e_i-e_j\), and \(T=-e_j\) satisfies \eqref{lat:reducible}.

Replace \(v=e_i+T\) by \(e_i\). Since \(T\cdot e_i=0\), all pairings between distinct vertices are unchanged, and the square of this vertex becomes \(-1\). The new vectors lie in \(T^\perp\), and the other \(N-1\) vectors remain linearly independent. The intersection matrix is therefore negative semidefinite of corank one, with a unique \((-1)\)-vertex.

Contract the unique \((-1)\)-vertex whenever its valency is one or two. A vector satisfying the adjunction equality and having square \(-1\) is a diagonal basis vector \(e_j\). Every neighbor contains \(-e_j\), and every nonneighbor has zero coefficient in this coordinate. Add \(e_j\) to each neighbor and omit coordinate \(j\). This contraction preserves adjunction and corank one. At most one new \((-1)\)-vertex appears. Two could appear only after a valency-two contraction, and would be adjacent, whereas two vectors satisfying the adjunction equalities and having square \(-1\) have nonpositive pairing. At least one appears by Lemma~\ref{lat:minimal-definite}. The vector \(T\) has zero coefficient in each omitted coordinate and remains orthogonal to every vector in the contracted graph.

A \((-1)\)-vertex has valency at most three. Otherwise the sum of this vector and its neighbors would have a coefficient at most \(-3\), contrary to \eqref{lat:subtree-forms}. The contractions terminate at a \((-1)\)-vertex of valency three. Since these contractions do not create nodes, this vertex is the image of the original node. Three chains of weights at most \(-2\) are attached to it. Write their fractions as \(n_i/a_i\), and let \(Z=\sum m_vv\) be a positive kernel cycle with central coefficient \(m_0\). Solving the chain intersection equations gives coefficients at the components adjacent to the center and at the outer ends \(m_0a_i/n_i\) and \(m_0/n_i\), respectively. The central intersection equation and adjunction give
\[
\sum_{i=1}^{3}\frac{a_i}{n_i}=1,
\quad
0=(\mathbf{k}+\sum_vv)\cdot Z
=m_0-\sum_{i=1}^{3}\frac{m_0}{n_i}.
\]
Subtracting these two equalities gives \(\sum_i(a_i-1)/n_i=0\). Since \(a_i\ge1\), each \(a_i=1\). The resulting graph has a central \((-1)\)-vertex and three arms with one vertex. After permutation, their orders \((l_1,l_2,l_3)\) are \((3,3,3)\), \((2,4,4)\), or \((2,3,6)\), the solutions of \(\sum_i1/l_i=1\) with \(l_i\ge2\).

Put \(m=-T^2\). The resulting diagonal lattice has rank four. Its central vector is \(e_s\), each neighbor contains \(-e_s\), and the total vector \(V\) has coefficient \(-2\) in column \(s\). All other coordinates of \(V\) are zero or \(-1\). Since
\[
V^2=5-(l_1+l_2+l_3),\quad T\cdot V=0,
\]
the support of \(T\) is contained in the zero coordinates of \(V\). There are \(12-l_1-l_2-l_3\) such coordinates. Adding \(T\) to the central vector gives a full-rank integral embedding of the four-vertex tree with central weight \(-(1+m)\). Indeed, \(T\) is nonzero and orthogonal to the original span, and the kernel cycle has nonzero central coefficient. The determinant of the negative intersection matrix is
\[
m l_1l_2l_3,
\]
and is a square because the four vectors are integral in a rank-four diagonal lattice. Since the nonzero coefficients of \(T\) are \(-1\), we obtain
\[
\begin{array}{c|c|c}
(l_1,l_2,l_3)&1\le m\le 12-l_1-l_2-l_3&m\\\hline
(3,3,3)&1\le m\le3&3\\
(2,4,4)&1\le m\le2&2\\
(2,3,6)&m=1&1
\end{array}
\]
Thus the four-vertex tree with central weight \(-(1+m)\) satisfies \(\sum_v(b_v-3)=1\).

Reverse the contractions, and then decrease the square of the unique \((-1)\)-vertex by \(m\). Every inverse contraction is a blowup at this vertex or an incident edge; the new exceptional vertex is the unique \((-1)\)-vertex. The node of the graph with three arms, each with one vertex persists. A second node cannot be created, since the final graph is star-shaped. The distance of the unique \((-1)\)-vertex from this node does not decrease. If the final graph has three arms, the first blowup cannot be at a smooth point of the node, since this would create a fourth arm. Hence the first blowup, if any, is at an edge. The new \((-1)\)-vertex has valency two, so a blowup at a smooth point of it would create a second node. Every subsequent blowup is therefore at an edge.

If the final graph has four arms, the first blowup is at a smooth point of the node. It creates the fourth arm and changes the central weight from \(-1\) to \(-2\). A first blowup at an edge would leave every subsequent \((-1)\)-vertex away from the node, so the fourth arm could not be created. Since \(d\ge3\), the second blowup is at the edge joining the node to its new leaf. A blowup at a smooth point of the new leaf would move the unique \((-1)\)-vertex away from the node and leave the central weight equal to \(-2\). After the second blowup the unique \((-1)\)-vertex has valency two. All remaining blowups are at edges in the fourth arm. The original three arms with one vertex are unchanged, and the fourth arm has at least two vertices.

A blowup at an edge preserves \(\sum_v(b_v-3)\), and a blowup at a smooth point decreases it by one. There are no blowups at smooth points when \(t=3\), and exactly one when \(t=4\). Decreasing the square of the unique \((-1)\)-vertex by \(m\) increases this sum by \(m\), independently of the number of contractions. Hence the sum is \(1\) when \(t=3\) and \(0\) when \(t=4\).
\end{proof}

The cohomology dimensions and weight identities below also appear in the proof of Wahl~\cite[Corollary~8.2]{Wahl2011}. Here they follow from the structural lattice results above, without using the classification of the graphs.

\begin{proposition}\label{prop:three-arm-lattice-vanishing}
For a noncyclic weighted homogeneous surface singularity admitting a QHD smoothing, with \(t\) arms, one has
\[
h^1(S_X)=t-3,\quad \sum_v(b_v-3)=4-t.
\]
In the four-arm case exactly three arms are arms with one vertex and orders \((3,3,3)\), \((2,4,4)\), or \((2,3,6)\).
\end{proposition}
\begin{proof}
Theorem~\ref{thm:arm-weight-bounds} gives \(t=3\) or \(4\) and \(d\ge t-1\). Apply Lemma~\ref{lat:star-reduction} to the embedding satisfying the adjunction equalities, and then use \eqref{foundation:dimension-eq}.
\end{proof}

The preceding vanishing gives three-arm tautness; see Laufer~\cite{Laufer1973} for the classification of taut rational surface singularities.

\begin{corollary}\label{cor:three-arm-taut}
Every three-arm weighted homogeneous singularity admitting a QHD smoothing is taut.
\end{corollary}
\begin{proof}
The singularity is rational by Theorem~\ref{foundation:PS}. If a rational weighted homogeneous surface singularity is not taut, then it has a nonzero first-order equisingular deformation Wahl~\cite[Corollary~(5.7)]{Wahl1988}. The tangent space of equisingular deformations is \(H^1(S_X)\), which vanishes by Proposition~\ref{prop:three-arm-lattice-vanishing}.
\end{proof}
\section{Canonical-cover obstructions and Seifert invariants}\label{sec:classification}

We establish a character obstruction and compute it, together with the canonical Riemann--Roch corrections, from the Seifert data.

\subsection{A character obstruction on the canonical cover}

\begin{theorem}[Character obstruction]\label{obs:equivariant-hodge}
Let \((V,0)\) be a normal complex surface singularity admitting a rational homology disk smoothing. Let \(B\) be its canonical cover, with group \(G=\boldsymbol{\mu}_r\), and let \(\mathcal{P}=H^1(X_B,\mathcal{O}_{X_B})\), where \(X_B\to\operatorname{Spec}B\) is a \(G\)-equivariant resolution over a sufficiently small representative. Then \(\mathcal{P}\oplus\overline{\mathcal{P}}\) is isomorphic, as a complex \(G\)-representation, to a direct summand of the reduced regular representation \(\mathbb{C}[G]/\mathbb{C}\). Equivalently, if \(p_j\) is the multiplicity of \(\chi^j\) in \(\mathcal{P}\), for a faithful character \(\chi\) of \(G\), then
\begin{equation}\label{obs:hodge}
p_j+p_{-j}\leq 1-\delta_{j,0}
\qquad\left(j\in\mathbb{Z}/r\mathbb{Z}\right).
\end{equation}
\end{theorem}

\begin{proof}
If \((V,0)\) is smooth, then \(r=1\) and \(\mathcal{P}=0\), so the assertion is immediate. Otherwise, Theorem~\ref{foundation:PS} gives weighted homogeneity, a one-dimensional QHD smoothing component, and the strong \(\mathbb{Q}\)-Gorenstein property. Since \(\mathcal{P}\) depends only on the central canonical cover, we choose the smoothing over the normalization of this component, as in Section~\ref{sec:compactification}. Write \(R\) for its graded total algebra and \(\tau\) for its parameter. The grading is good and \(\deg\tau>0\).

The strong \(\mathbb{Q}\)-Gorenstein property identifies the central fiber of the relative canonical-cover algebra with \(B\):
\[
\widetilde{R}=\bigoplus_{j=0}^{r-1}\omega_R^{[j]},
\qquad
\widetilde{R}/\tau\widetilde{R}\simeq B.
\]
A homogeneous trivialization of \(\omega_R^{[r]}\) defines a rational grading on \(\widetilde{R}\) compatible with multiplication. Multiplying all degrees by a positive integer makes this grading integral. It commutes with the deck group \(G\), which acts trivially on \(\tau\). This grading is good: \(\widetilde{R}\) is a finite graded domain over \(R\), and a homogeneous element of negative degree would satisfy a homogeneous monic equation whose nonleading coefficients all have negative degree in \(R\). It would therefore be nilpotent. A homogeneous element of degree zero satisfies a monic equation over \(R_0=\mathbb{C}\), and hence belongs to \(\mathbb{C}\). Thus \(\widetilde{R}\) has no negative degrees and \(\widetilde{R}_0=\mathbb{C}\); the induced grading on \(B\) is good as well.

Let \(\widetilde{F}\to F\) be the covering of Milnor fibers, on which \(G\) acts freely. The covering fiber is connected, has the homotopy type of a finite CW complex of real dimension at most two, and satisfies \(H^1(\widetilde{F},\mathbb{Q})=0\) by Greuel--Steenbrink~\cite{GreuelSteenbrink1983}. Since \(\chi(\widetilde{F})=r\chi(F)=r\), Wahl~\cite[Lemma~5.7]{Wahl2015} gives
\[
H^2(\widetilde{F},\mathbb{C})\simeq\mathbb{C}[G]/\mathbb{C}
\]
as complex \(G\)-representations.

Apply Pinkham's construction~\cite{Pinkham} to the covering smoothing. The construction is \(G\)-equivariant because the deck action commutes with the grading. Explicitly, if \(w=\deg\tau>0\), the projective family before resolving its boundary is
\[
\operatorname{Proj}_{\mathbb{C}[z]}
\left(\widetilde{R}[s,z]/(\tau-zs^w)\right),
\qquad \deg s=1,\quad\deg z=0,
\]
where \(G\) fixes \(s\) and \(z\). Its central fiber is \(\operatorname{Proj}(B[s])\), and its boundary is locally trivial over the base. Simultaneously resolve the cyclic quotient singularities along the boundary. Their minimal resolutions are unique, so this operation is \(G\)-equivariant. The resulting projective family remains flat, since this resolution is locally a product over the base. Its smooth projective general fiber \(Y\) has boundary \(C_B\cup D_B\), where \(C_B\simeq\operatorname{Proj}B\) is a smooth connected projective curve and \(D_B\) consists of rational chains attached once to \(C_B\). The complement \(U=Y\setminus(C_B\cup D_B)\) has the \(G\)-equivariant homotopy type of \(\widetilde{F}\). By the divisor Gysin sequence, \(H^1(Y,\mathbb{Q})\to H^1(U,\mathbb{Q})\) is injective, so \(H^1(Y,\mathcal{O}_Y)=H^3(Y,\mathbb{Q})=0\).

Let \(Y_0\) be the central compactified fiber after resolving the boundary, and let \(f\colon W\to Y_0\) resolve its vertex. The surface \(W\) is, up to blowups, equivariantly ruled over \(C_B\). Hence \(H^1(W,\mathcal{O}_W)\simeq H^1(C_B,\mathcal{O}_{C_B})\) equivariantly and \(H^2(W,\mathcal{O}_W)=0\). Moreover, \(f_*\mathcal{O}_W=\mathcal{O}_{Y_0}\), and \(R^1f_*\mathcal{O}_W\) is supported at the vertex with space of sections \(\mathcal{P}\). The equivariant Leray spectral sequence gives
\[
\chi_G(Y_0,\mathcal{O}_{Y_0})
=[\mathbb{C}]-[H^1(C_B,\mathcal{O}_{C_B})]+[\mathcal{P}]
\]
in the complex representation ring \(R(G)\). The equivariant Euler characteristic is constant in the flat projective family: the character projectors split its relative cohomology complex, and the Euler characteristic of each summand is constant. Since \(H^1(Y,\mathcal{O}_Y)=0\), comparison with the general fiber yields
\[
[\mathcal{P}]
=[H^2(Y,\mathcal{O}_Y)]+[H^1(C_B,\mathcal{O}_{C_B})].
\]

Put \(H=H^2(U,\mathbb{Q})\). Deligne's mixed Hodge structure and the weight spectral sequence are functorial, so they are \(G\)-equivariant~\cite[Theorem~3.2.5(iii) and Corollary~3.2.13(iii)]{Deligne1971}. If \(Z_\nu\) are the irreducible boundary components, the sequence gives
\[
\operatorname{Gr}^W_2 H
\simeq H^2(Y,\mathbb{Q})/
\langle[Z_\nu]\rangle,
\qquad
\operatorname{Gr}^W_3 H
\simeq H^1(C_B,\mathbb{Q})(-1).
\]
For the second identification, the rational components of \(D_B\) have zero first cohomology and \(H^3(Y,\mathbb{Q})=0\). Boundary divisor classes have type \((1,1)\), and the Tate twist has trivial \(G\)-action. Consequently, the \((0,2)\) and \((2,0)\) parts of \(\operatorname{Gr}^W_2 H_{\mathbb{C}}\), together with the \((1,2)\) and \((2,1)\) parts of \(\operatorname{Gr}^W_3 H_{\mathbb{C}}\), have representation
\[
H^2(Y,\mathcal{O}_Y)\oplus\overline{H^2(Y,\mathcal{O}_Y)}
\oplus H^1(C_B,\mathcal{O}_{C_B})\oplus\overline{H^1(C_B,\mathcal{O}_{C_B})}
\simeq\mathcal{P}\oplus\overline{\mathcal{P}}.
\]
Semisimplicity splits the weight filtration as \(G\)-representations, giving the required summand of \(H^2(U,\mathbb{C})\). The splitting need not be canonical or preserve the mixed Hodge structure.

Finally, \(H^2(U,\mathbb{C})\simeq\mathbb{C}[G]/\mathbb{C}\). Complex conjugation sends \(\chi^j\) to \(\chi^{-j}\), so the multiplicity of \(\chi^j\) in \(\mathcal{P}\oplus\overline{\mathcal{P}}\) is \(p_j+p_{-j}\). Its multiplicity in the reduced regular representation is \(1-\delta_{j,0}\), proving \eqref{obs:hodge}. Conversely, these multiplicity inequalities imply the asserted direct-summand property by semisimplicity.
\end{proof}

\begin{corollary}\label{obs:hodge-consequences}
Under the hypotheses of Theorem~\ref{obs:equivariant-hodge}, the representation \(\mathcal{P}\) is multiplicity-free and contains neither the trivial character nor a nontrivial character of order two. No character and its inverse both occur, and
\[
2p_g(B)\le r-1.
\]
\end{corollary}
\begin{proof}
The character inequalities give each assertion, and their sum gives the genus bound.
\end{proof}

The scalar bound \(2p_g(B)\le\mu(\widetilde F)=r-1\) is classical; see Stipsicz--Szab\'{o}--Wahl~\cite[Section~2.3]{SSW}. Theorem~\ref{obs:equivariant-hodge} restricts each inverse pair of characters independently of the graph classification.

\subsection{The three-arm numerical equations}\label{obs:section}

For three-arm graphs of singularities admitting QHD smoothings, the weight identity and \(K^2+N=0\) give two equations in the Seifert invariants. These equations also determine the cases with \(\eta=0\).

\begin{proposition}\label{obs:general-equations}\label{obs:orbifold-sign}
Let \((V,0)\) be a three-arm normal weighted homogeneous surface singularity admitting a QHD smoothing. Then
\begin{equation}\label{obs:reduced-graph-equations}
\sum_{i,j}(b_{ij}-3)=4-d,\quad e\left(1-\sum_i\frac{a_i^*}{n_i}\right)=\eta^2.
\end{equation}
Let a negative definite three-arm graph satisfy these equations with \(2\le d\le4\). Then \(0\le\eta/e<1\), and equality \(\eta=0\) occurs precisely for
\[
\left(4;(3,1),(3,1),(3,1)\right),\quad \left(3;(2,1),(4,1),(4,1)\right),\quad \left(2;(2,1),(3,1),(6,1)\right).
\]
\end{proposition}
\begin{proof}
The first identity follows from Proposition~\ref{prop:three-arm-lattice-vanishing}. To prove the second identity, note that \((K+E)E_0=1\), that \((K+E)E_{i\ell_i}=-1\) for each outer end, and that \(K+E\) has intersection zero with every other component. The positive intersection matrix \(A_i\) of an arm satisfies
\[
(A_i^{-1})_{11}=\frac{a_i}{n_i},\quad (A_i^{-1})_{1\ell_i}=\frac1{n_i},\quad (A_i^{-1})_{\ell_i\ell_i}=\frac{a_i^*}{n_i}.
\]
Taking the Schur complement of the arm matrices gives \((K+E)^2=-\sum_i a_i^*/n_i-\eta^2/e\). Using \(E^2=-\sum_vb_v+2(N-1)\) and \(KE=\sum_v(b_v-2)\), we obtain
\begin{equation}\label{obs:canonical-square}
K^2+N=5-d-\sum_{i,j}(b_{ij}-3)-\sum_i\frac{a_i^*}{n_i}-\frac{\eta^2}{e}.
\end{equation}
Substitution of the first identity and \(K^2+N=0\) proves the second identity.

For the remaining assertions, assume \eqref{obs:reduced-graph-equations}. Since \(a_i^*\ge1\),
\[
0\le\frac{\eta^2}{e}=1-\sum_i\frac{a_i^*}{n_i}\le\eta.
\]
Thus \(\eta\ge0\), and \(\eta>0\) implies \(\eta/e\le1\). Equality would force \(a_i^*=a_i=1\) for all \(i\), contradicting \(e-\eta=d-1>0\). If \(\eta=0\), every arm has one vertex and \(\sum_i1/n_i=1\). For \(2\le n_1\le n_2\le n_3\), this gives \((3,3,3)\), \((2,4,4)\), or \((2,3,6)\): one has \(n_1\le3\), and \(n_1=2\) reduces the equation to \((n_2-2)(n_3-2)=4\). The first identity determines the stated central weights.
\end{proof}

\subsection{Computing the characters from Seifert data}\label{obs:characters}

Let \(A\) be the graded coordinate ring of a rational noncyclic weighted homogeneous surface singularity, let \(C=\mathbb{P}^1\) be the central curve, and let \(P_i\in C\) be the points corresponding to the arms. Choose an integral divisor \(Q\) of degree \(d\) and put
\[
\mathcal{D}=Q-\sum_i\frac{a_i}{n_i}P_i,\quad \Xi=K_C+\sum_i\left(1-\frac1{n_i}\right)P_i.
\]
Thus \(\deg\mathcal{D}=e\) and \(\deg\Xi=\eta\). The description of the graded ring and its reflexive canonical powers in Wahl~\cite[Theorems~4.1 and 4.4]{Wahl2013} gives
\[
(N_j)_k=H^0\left(C,\mathcal{O}_C\left(\left\lfloor k\mathcal{D}+j\Xi\right\rfloor\right)\right),\quad N_j=\omega_A^{[j]}.
\]
Let \(r\) be the least positive integer for which \(c=r\frac{\eta}{e}\) is integral and
\[
n_i\mid r-ca_i\quad\text{for every }i.
\]
This is the canonical index by Wahl~\cite[Corollary~4.5]{Wahl2013}: the divisor \(r\Xi-c\mathcal{D}\) is integral precisely when the displayed congruences hold, and an integral divisor of degree zero on \(C\) is principal. Let \(\kappa\) be the character of a canonical generator on the cover and put \(\chi=\kappa^{-1}\). The summand \(N_j=\omega_A^{[j]}\) of the canonical-cover algebra has character \(\chi^j\). We index these characters by \(j\in\mathbb{Z}/r\mathbb{Z}\). For an integer representative \(j\), set
\begin{equation*}
D_j(k)=dk+(t-2)j-\sum_i\left\lceil\frac{ka_i+j}{n_i}\right\rceil.
\end{equation*}
The divisor \(\lfloor k\mathcal{D}+j\Xi\rfloor\) has degree \(D_j(k)\). Applying the cohomology formula of Pinkham~\cite{Pinkham1977} to the lifted grading on the canonical cover and taking the \(j\)-th eigensummand gives the multiplicity of the corresponding character in the cohomology of an equivariant resolution
\begin{equation}\label{obs:genus-character}
p_j=\sum_{k=-\lfloor j\frac{\eta}{e}\rfloor}^{\lfloor(1-j)\frac{\eta}{e}\rfloor}[-D_j(k)-1]_+,\quad [u]_+=\max\{u,0\}.
\end{equation}
The summand is the dimension of the first cohomology of a line bundle of degree \(D_j(k)\) on \(\mathbb{P}^1\). To determine the summation range, write \(\frac{\eta}{e}=p/q\) in lowest terms. A section with indices \((j,k)\) has weight \(qk+pj\), and the canonical weight is \(p\). The identity
\[
D_{1-j}(-k)=-D_j(k)-2
\]
identifies its Serre dual with a section of weight \(p-(qk+pj)\). Hence only indices in the interval \(0\le qk+pj\le p\) contribute. The congruences defining \(r\) identify the formula for \(j+r\) with that for \(j\) after translating \(k\) by an integer.

\begin{corollary}\label{obs:noncyclic-cover}
Let \((V,0)\) be noncyclic and admit a rational homology disk smoothing. Then \(0\le\eta/e<1\). With the character indexing in \eqref{obs:genus-character}, its canonical cover satisfies
\[
p_1=1,\qquad p_{-1}=0.
\]
In particular, the canonical cover is not rational, and the canonical index is at least three.
\end{corollary}
\begin{proof}
For three arms, Proposition~\ref{obs:orbifold-sign} gives \(0\le\eta/e<1\). For four arms, Lemma~\ref{lat:star-reduction} gives three arms with one vertex whose reciprocal orders sum to one. Write \(n/a\) for the fourth arm. Then \(d\ge3\), \(1<a<n\), and
\[
\eta=1-\frac1n>0,\qquad e-\eta=d-2-\frac{a-1}{n}>0.
\]
At \(j=1\), the summation interval in \eqref{obs:genus-character} consists only of \(k=0\), and \(D_1(0)=(t-2)-t=-2\). Hence \(p_1=1\). Theorem~\ref{obs:equivariant-hodge} gives \(p_{-1}=0\). Thus the cover has positive geometric genus, and the characters indexed by \(1\) and \(-1\) are distinct. Its index is therefore at least three.
\end{proof}

This consequence precedes the classification and uses only the numerical equations and the structural lattice reduction of Section~\ref{sec:three-arm-tautness}.

\begin{example}\label{obs:hodge-example}
The character inequalities are stronger than the scalar bound \(2p_g(B)\le r-1\), even for rational weighted homogeneous singularities satisfying \(K^2+N=0\) and having square determinant. Consider the Seifert datum
\[
\left(2;(3,2),(7,4),(22,3)\right).
\]
Its arms are \([2,2]\), \([2,4]\), and \([8,2,2]\). The corresponding graph is negative definite, since
\[
e=\frac{289}{462}>0,\qquad \eta=\frac{221}{462},\qquad \frac{\eta}{e}=\frac{13}{17}.
\]
Its fundamental cycle, with the coefficients listed along the arms from the central component outward, is
\[
Z=(3;2,1;2,1;1,1,1).
\]
One obtains this cycle from the reduced exceptional divisor by adding the central component, the first components of the first and second arms, and the central component again. At each step the component added has intersection one with the current cycle. The resulting cycle is anti-nef, with \(Z^2=-10\) and \(KZ=8\). Hence \(p_a(Z)=0\), so the singularity is rational. The numerical canonical divisor has coefficients
\[
K=-\frac1{17}(30;20,10;22,14;18,12,6).
\]
Consequently, \(N=8\), \(K^2=-8\), and the absolute determinant is \(289=17^2\). The canonical index is \(r=17\): the denominator of \(\eta/e\) forces \(17\mid r\), and \(r=17\), \(c=13\) satisfy the defining congruences. Formula~\eqref{obs:genus-character} gives
\[
p_1=p_{16}=1,\qquad p_j=0\quad(j\not\equiv1,16\pmod{17}).
\]
Thus \(p_g(B)=2\) satisfies the scalar bound, whereas \(p_1+p_{-1}=2\) violates Theorem~\ref{obs:equivariant-hodge}. This datum therefore admits no QHD smoothing. The example compares the character inequalities with the stated numerical tests; it does not separate them from all the restrictions used in the classification. In fact, its arm-weight sum is \(1\), so it also fails the weight identity in \eqref{obs:reduced-graph-equations}.
\end{example}

\subsection{Reflexive canonical sheaves and Riemann--Roch correction terms}\label{obs:CSRC-section}
We use the graded coordinate ring \(A\) and its reflexive canonical powers \(N_j\) as above, and localize at the vertex when considering a smoothing. If \([u]_n\) denotes the least nonnegative residue of \(u\) modulo \(n\), put
\[
T(a,n;z)=\sum_{v=0}^{n-1}(n-2v)[-av-z]_n.
\]
For \(f=\lfloor j\eta/e\rfloor\) and \(\theta=j\eta/e-f\), define
\begin{equation}\label{obs:canonical-correction}
\gamma_j=p_j+\frac e2\theta\left(1+\frac{\eta}{e}-\theta\right)+\sum_i\frac{T(a_i,n_i;0)-T(a_i,n_i;j-a_if)}{2n_i^2}.
\end{equation}

The local correction terms are those of Riemann--Roch for normal surfaces~\cite{Blache1995}. We calculate them using the graded cohomology of the Seifert construction~\cite{Pinkham1977} and the canonical modules described in Wahl~\cite[Section~4]{Wahl2013}.

\begin{proposition}\label{obs:correction-identification}
For a rational weighted homogeneous surface singularity, the quantity \(\gamma_j\) in \eqref{obs:canonical-correction} is the local Riemann--Roch correction \(c_A(jK_A)\), with the normalization
\[
\chi(\mathcal{O}_Y(D))=\chi(\mathcal{O}_Y)+\frac{D(D-K_Y)}2+\sum_{x\in\operatorname{Sing}Y}c_x(D).
\]
\end{proposition}

\begin{proof}
Put \(\alpha=\eta/e\) and consider the normalized Hilbert series
\[
H_j(z)=\sum_{k\in\mathbb{Z}}\dim_{\mathbb{C}}(N_j)_k\exp(-(k+j\alpha)z).
\]
If \(k<-f\), the orbifold degree \(ek+j\eta\) is negative, so there are no sections. For \(k=v-f\), with \(v\ge0\), set \(c_i=j-a_if\) and
\[
\epsilon_i(v;c_i)=\left\lceil\frac{a_iv+c_i}{n_i}\right\rceil-\frac{a_iv+c_i}{n_i}.
\]
The sequence \(\epsilon_i\) has period \(n_i\) and mean \((n_i-1)/(2n_i)\). The degree formula gives
\[
D_j(v-f)+1=e(v+\theta)+1-\sum_i\epsilon_i(v;c_i).
\]
The differences between the dimensions of sections and the Euler characteristics are zero for all sufficiently large \(v\), and their sum is \(p_j\), by \eqref{obs:genus-character} and Serre duality. Expanding the linear and periodic geometric series at \(z=0\), we obtain the constant coefficient
\[
F_j=p_j-\frac e2\left(\theta^2-\theta+\frac16\right)+\frac12+\frac{\eta\theta}{2}-\sum_i\frac{T(a_i,n_i;c_i)}{2n_i^2}.
\]
Indeed, the constant contribution of a periodic sequence is
\[
\sum_{v=0}^{n_i-1}\epsilon_i(v;c_i)\left(\frac12-\frac{v+\theta}{n_i}\right)
=\frac{T(a_i,n_i;c_i)}{2n_i^2}-\theta\frac{n_i-1}{2n_i}.
\]
Since \(A\) is rational, \(p_0=0\). Subtracting \(F_0\) gives \(F_j-F_0=\gamma_j\).

We compute this difference by Riemann--Roch on the projective cone \(Y=\Proj(A[z_0])\), where \(\deg z_0=1\), and put \(C_\infty=(z_0=0)\). Its singularities consist of the vertex and cyclic quotient points on \(C_\infty\). The identity for canonical modules \(\omega_{A[z_0]}=\omega_A[z_0](-1)\) identifies the sheaf associated with \(N_j[z_0]\) as
\[
\mathcal{F}_j=(\omega_Y^{[j]}(jC_\infty))^{**}.
\]
Numerically,
\[
C_\infty^2=e,\quad K_Y\equiv(\alpha-1)C_\infty.
\]
These identities follow from the Seifert construction, or from the leading terms of the degree formula and \(D_r(k)=D_0(k+r\alpha)\). For an integer \(m\) sufficiently large, Serre vanishing and Riemann--Roch give
\begin{equation}\label{obs:cone-RR}
\sum_{k\le m}\dim_{\mathbb{C}}(N_j)_k
=\chi(\mathcal{O}_Y)+\frac e2(m+j\alpha)(m+j\alpha+1-\alpha)+c_A(jK_A)+\sum_i c_i(m,j).
\end{equation}
Here \(c_i(m,j)\) are the corrections at infinity. Riemann--Roch for this normal surface is obtained by resolving and applying ordinary Riemann--Roch, as in the calculation with full pullbacks in Park--Shin~\cite[Section 8]{PS}; see also Blache~\cite{Blache1995}.

At the \(i\)-th boundary point, \(\mathcal{O}_Y(C_\infty)\) generates the cyclic local divisor class group. In the Seifert chart the stabilizer acts faithfully on its normal coordinate. Thus \(c_i(m,j)\), as \(m\) varies through a period, takes the values of the corrections for all local divisor classes. Its mean is independent of \(j\). Multiply the generating series of \eqref{obs:cone-RR} by \(1-\exp(-z)\), with the shift \(\exp(-j\alpha z)\). The finitely many terms below the range in which Serre's vanishing theorem applies contribute \(O(z)\). The shifted quadratic, linear, and constant sums have principal parts \(2z^{-3}\), \(z^{-2}\), and \(z^{-1}\), respectively, independent of the shift. Each periodic correction contributes its mean to the resulting constant coefficient. It follows that \(F_j-c_A(jK_A)\) is independent of \(j\). Since \(c_A(0)=0\), we obtain
\[
\gamma_j=F_j-F_0=c_A(jK_A).
\]

\end{proof}

\begin{proposition}[Vanishing of canonical correction terms]\label{obs:CSRC}
Let \(A\) be the local algebra of a rational normal surface singularity admitting a strong \(\mathbb{Q}\)-Gorenstein smoothing. Then \(c_A(jK_A)=0\) for every \(j\in\mathbb{Z}\). In particular, every noncyclic normal surface singularity admitting a rational homology disk smoothing satisfies
\begin{equation}\label{obs:CSRC-equations}
\gamma_j=c_A(jK_A)=0\qquad(j\in\mathbb{Z}),
\end{equation}
where \(\gamma_j\) is computed from its Seifert data by \eqref{obs:canonical-correction}.
\end{proposition}
\begin{proof}
If \(A\) is Gorenstein, every \(jK_A\) is Cartier and its local correction is zero. We may therefore assume that the canonical index is greater than one.

Let \(R\) be the total local algebra, let \(s\) be the smoothing parameter, and put \(M_j=\omega_R^{[j]}\) and \(N_j=\omega_A^{[j]}\). Formula~\eqref{foundation:strong-milnor} gives \(\mu=K^2+N\). The Riemann--Roch specialization formula of Park--Shin~\cite[Theorem~8.3 and Corollary~8.4]{PS} therefore yields
\[
\dim_{\mathbb{C}}\operatorname{coker}(M_j/sM_j\to N_j)
=\binom j2(\mu-N-K^2)+c_A(jK_A)=c_A(jK_A).
\]
Every specialization map is an isomorphism by the strong \(\mathbb{Q}\)-Gorenstein hypothesis, so \(c_A(jK_A)=0\) for every integer \(j\). For a QHD smoothing, Theorem~\ref{foundation:PS} gives rationality, weighted homogeneity, and the strong \(\mathbb{Q}\)-Gorenstein property. Proposition~\ref{obs:correction-identification} then identifies \(c_A(jK_A)\) with \(\gamma_j\).

For this QHD smoothing, choose compatible trivializations at the common canonical index \(r\). The cokernel \(\mathcal{C}_{\varphi}\) of specialization of the canonical-cover algebras decomposes as
\[
[\mathcal{C}_{\varphi}]
=\sum_{j=0}^{r-1}\gamma_j[\mathbb{C}_{\kappa^{-j}}],
\]
where \(\kappa\) is the canonical character. Thus \eqref{obs:CSRC-equations} expresses the vanishing of this actual representation in the character basis.
\end{proof}

The vanishing from Park--Shin~\cite{PS} concerns the cyclic subgroup generated by \(K_A\); formula~\eqref{obs:canonical-correction} expresses it in Seifert invariants. Both \(\gamma_j\) and \(p_j\) are independent of the attachment points, since line-bundle cohomology on \(\mathbb{P}^1\) depends only on degree. Section~\ref{sec:cross-ratios} treats the four-arm attachment points.

For the datum in Example~\ref{obs:hodge-example}, formula~\eqref{obs:canonical-correction} gives \(\gamma_9=1\), so it also violates canonical correction vanishing.

Under the reduced equations \eqref{obs:reduced-graph-equations} and in the range \(0\le \frac{\eta}{e}<1\), we shall also use the following expansion of \eqref{obs:canonical-correction} for \(j\ge0\):
\begin{equation}\label{obs:correction-expanded}
\begin{split}
\gamma_j={}&p_j-(j-f)-\frac{j(j-1)}2\frac{\eta^2}{e}+jf\eta-\frac{f(f+1)}2e\\
&+\sum_i\sum_{v=1}^{j-1}\frac{[-va_i^*]_{n_i}}{n_i}-\sum_i\sum_{h=1}^{f}\frac{[ha_i-j]_{n_i}}{n_i}.
\end{split}
\end{equation}
The identities
\[
T(a,n;z+1)-T(a,n;z)=n\left(n-1-2[-a^*z]_n\right),
\]
\[
T(a,n;z-a)-T(a,n;z)=-n(n-1)+2n[a-z]_n
\]
give this expansion by summing the successive differences.

\begin{proposition}\label{obs:singleton}
Suppose \(t=3\), \(d=2\), and \(\eta>0\). If \eqref{obs:reduced-graph-equations}, \eqref{obs:hodge}, and \eqref{obs:CSRC-equations} hold, exactly two arms have one vertex.
\end{proposition}
\begin{proof}
Put \(\alpha=\eta/e\), so \(0<\alpha<1\) by Proposition~\ref{obs:orbifold-sign}. Set
\[
z=\#\{i:a_i=1\},\quad s_2=\#\{i:a_i=2\},\quad A_2=\#\{i:2a_i^*>n_i\}.
\]
If \(a_i=2\), then \(a_i^*=(n_i+1)/2\), so \(A_2\ge s_2\).

For \(0<\alpha<1/3\), the genus formula gives \(p_3=0\), and \eqref{obs:correction-expanded} at \(j=3\) gives \(A_2=0\). Thus \(s_2=0\), and
\[
0<e-3\eta=-1+\sum_i\frac{3-a_i}{n_i}\le-1+z.
\]
Hence \(z\ge2\). For \(1/3\le\alpha\le1/2\), put \(A=z+s_2-2\). The same formulas give
\[
p_3=[A]_+,\quad 0=[-A]_++A_2.
\]
Thus \(A_2=s_2=0\) and \(z\ge2\).

For \(1/2<\alpha<1\), the equation for \(j=2\) gives \(p_2=z-1\). Write \(J=z-1\); the Hodge bound gives \(0\le J\le1\). If \(\alpha<2/3\), then \(p_3=p_{-2}=0\), and the equation for \(j=3\) is \(0=1-J+A_2-s_2\). If \(\alpha\ge2/3\), put \(F=\#\{i:2a_i\ge n_i+3\}\). Then \(p_3=[z-F]_+\), \(p_{-2}=[F-z]_+\), and the equation becomes \(0=1-J+p_{-2}+A_2-s_2\). Nonnegativity forces \(J=1\) in both cases, so \(z=2\).

Finally, three arms with one vertex would give \(e\eta=\eta^2\), whereas \(e-\eta=1\). Thus \(z=2\) in every case.
\end{proof}

\section{The three-arm classification}\label{sec:three-arm-classification}

We determine the possible three-arm resolution graphs for central weights two, three, and four. The case of central weight two uses the numerical and character restrictions; the cases of central weights three and four use the projective compactification.

\subsection{Three arms and central weight two}\label{d2:section}
Proposition~\ref{obs:general-equations} gives the numerical identities and \(0\le\eta/e<1\).
\begin{theorem}\label{d2:classification}
Let a normalized three-arm graph of central weight two satisfy the numerical identities
\[
\sum_{i,j}(b_{ij}-3)=2,\quad e(1-\sum_i\frac{a_i^*}{n_i})=\eta^2,
\]
and the equations \eqref{obs:CSRC-equations} and Hodge inequalities \eqref{obs:hodge}. Then, after permuting the arms, its Seifert data belong to the following families \((2;(p,1),(q,1),(n,n-h))\). In the fourth row put \(U=pq-p-q\), and choose \(s=p\) or \(s=q\).
\[
\begin{array}{c|c|c}
(p,q)&\text{parameter}&n\\\hline
(2,3)&h\ge5&h^2-4h+1\\
(2,q),q\ge4&h=1+k(q-2),k\ge4&\dfrac{h^2-(2q-2)h+1}{q-2}\\
(2,6)&h\ge3&h^2-h+1\\
3\le p\le q&h=s-1+kU,k\ge3&\dfrac{h^2-(pq-2)h+1}{U}\\
(3,6)&h\ge2&h^2+1\\
(4,4)&h\ge5\text{ odd}&(h-1)^2/2
\end{array}
\]
The entry \(h=5\) in the first row is the graph \((2;(2,1),(3,1),(6,1))\). Every other entry has \(\eta>0\). The third, fifth, and sixth rows give additional families for the indicated orders of the arms with one vertex.
\end{theorem}

\begin{proof}
If \(\eta=0\), Proposition~\ref{obs:orbifold-sign} gives the datum \((2;(2,1),(3,1),(6,1))\). For \(\eta>0\), Proposition~\ref{obs:singleton} reduces the problem to two arms with one vertex. The numerical and character equations in Section~\ref{d2:equations}, together with Lemma~\ref{d2:reciprocity-closure}, proved below, determine the remaining arm. The case analysis below treats \(p=2\), coprime \(p,q\), \(\gcd(p,q)>1\) with \(p\nmid q\), and \(p\mid q\), and gives precisely the displayed families.
\end{proof}

\subsubsection{The numerical equations and \(\gamma_j=0\)}\label{d2:equations}
Assume \(\eta>0\), and write the data as \((2;(p,1),(q,1),(n,a))\), with \(2\le p\le q\). Put
\[
T=pq,\quad U=pq-p-q,\quad h=n-a,\quad b=a^*,\quad m=(hb+1)/n,\quad z=b-1.
\]
The condition \(\eta>0\) implies \(U>0\). Since \(ab\equiv1\pmod n\), \(m\) is a positive integer and \(\gcd(h,m)=1\). The numerical equations are equivalent to
\begin{equation}
b=h+2-\frac{T}{U}m,\quad n=\frac{h(h+2-(T/U)m)+1}{m},\quad 12s(h,m)=p+q-9+\frac{T}{U}-\frac2m.
\label{d2:base}
\end{equation}
We use the classical Dedekind sum
\[
s(h,m)=\sum_{v=1}^{m-1}\left(\frac vm-\frac12\right)\left(\frac{[hv]_m}{m}-\frac12\right),\quad s(h,1)=0.
\]
For the last arm, \(W(n,a)=\sum_j(b_j-3)=8-p-q\). Dedekind reciprocity gives
\[
W(n,a)=12s(m,h)-\frac{n+m}{h}+2=\frac{h-b}{m}-1-12s(h,m),
\]
which proves the last equation of \eqref{d2:base}. The normalization conditions \(a\ge2\), \(1<b<n\), and integrality are retained throughout. Put \(Q=h+2-b\), \(G=\gcd(Q,h+1)\), \(u=Q/G\), \(v=(h+1)/G\), and \(L=\operatorname{lcm}(p,q)\). Then
\[
\frac{U}{T}un=Gv(v-u)+u,\quad \frac{\eta}{e}=\frac{v-u}{v},\quad n>Gv,\quad \gcd(u,v)=1,
\]
and the canonical index is \(r=Lv/\gcd(L,u)\). Define
\[
q_j=\sum_{0\le k+j\frac{\eta}{e}\le \frac{\eta}{e}}[D_j(k)+1]_+.
\]
The identity \(D_j(k)+D_{1-j}(-k)=-2\) gives \(q_j=p_{1-j}\). For an integer \(\ell\) such that \(\ell u/v\) is not integral, write
\[
P=\lfloor\ell u/v\rfloor,\quad \tau=\ell u-vP,\quad E=\{P/p\}+\{P/q\}.
\]
Subtracting the equations \(\gamma_j=0\) gives
\begin{equation}
q_{1+\ell}-q_{1-\ell}=\Phi_\ell,
\quad \Phi_\ell=P-1-\lfloor P/p\rfloor-\lfloor P/q\rfloor-\left\lfloor\frac{P+\ell G(v-u)}n\right\rfloor
=\left\lceil E-\frac{\eta\tau}{v}\right\rceil-1.
\label{d2:recursion}
\end{equation}
We obtain \eqref{d2:recursion} from the equations \(\gamma_j=0\). Reindexing the defining sum by \(v\mapsto n-1-v\) gives
\[
T(a,n;a-c)-T(a,n;c)=n\left(n-1-2[-a^*c]_n\right).
\]
For the two characters \(\ell\) and \(-\ell\), the floor terms are \(\ell-P-1\) and \(-\ell+P\), and the fractional parts are \(1-\tau/v\) and \(\tau/v\). Subtracting the equations for the corresponding character multiplicities and using \(q_{1+\ell}=p_{-\ell}\), \(q_{1-\ell}=p_\ell\), therefore gives
\[
q_{1+\ell}-q_{1-\ell}
=1+\eta-\frac{\eta\tau}{v}
-\sum_{i=1}^3\frac{[\ell-P-1-\ell a_i^*]_{n_i}}{n_i}.
\]
For the two arms with one vertex the residues are \(p-1-[P]_p\) and \(q-1-[P]_q\). For the third arm the residue is \([-P-1-\ell z]_n\), where \(z=b-1=G(v-u)\). The numerical equation gives
\[
\frac{P+\ell z}{n}=\left(1-\frac1p-\frac1q\right)P+\frac{\eta\tau}{v}.
\]
Substitute these identities into the difference of the two equations, and use
\[
\left\lceil\frac{P+1+\ell z}{n}\right\rceil
=\left\lfloor\frac{P+\ell z}{n}\right\rfloor+1.
\]
The result is both expressions for \(\Phi_\ell\) in~\eqref{d2:recursion}.

For the degree calculation, put \(z_0=k+j\), \(w=vz_0-uj\), and \(\epsilon(z_0)=\lceil z_0/p\rceil+\lceil z_0/q\rceil-(1/p+1/q)z_0\). Then
\[
D_j(k)=\left\lfloor\frac{Gv}{un}w-\epsilon(z_0)\right\rfloor.
\]
The terms contributing to \(q_j\) satisfy \(0\le w\le v-u\). For \(j=1-\ell\), a section is possible only when \(\tau\ge u\), and then \(z_0=-P\), \(w=\tau-u\). In particular,
\begin{equation}
q_{1-\ell}=0\quad\text{if }\tau<u,
\quad q_{1-\ell}=\left[\left\lfloor\frac{\eta(\tau-u)}{v-u}-E\right\rfloor+1\right]_+\quad\text{if }\tau\ge u.
\label{d2:section-threshold}
\end{equation}
If \(u\nmid L\), take \(\ell=\lceil Lv/u\rceil<r\). Then \(P=L\), \(0<\tau<u\), and \(\Phi_\ell=-1\), whereas \eqref{d2:section-threshold} gives \(q_{1-\ell}=0\). This contradicts \eqref{d2:recursion}. Hence
\begin{equation}
u\mid L,\quad\text{equivalently}\quad \gcd(p,q)m\mid U(h+1).
\label{d2:divisibility}
\end{equation}
The conditions on the character multiplicities also give a fractional-part restriction. For an integer \(k\in\{0,\ldots,L-1\}\), set \(E=\{k/p\}+\{k/q\}\), and suppose \(0<E<\eta\). If an integer \(\ell\) belonged to
\begin{equation}
\left[\frac{vk}{u}+\frac{vE}{u\eta},\quad 1+\frac{vk}{u}+\frac{(v-u)E}{u\eta}\right),
\label{d2:forbidden}
\end{equation}
then \(P=k\), \(vE/\eta\le\tau<u+(v-u)E/\eta<v\), so \(\Phi_\ell=-1\) and \(q_{1-\ell}=0\), again a contradiction. Thus this interval contains no integer. If \(A\) is its lower endpoint, its length is \(1-E/\eta\), and consequently
\begin{equation}
0<\{A\}\le E/\eta.
\label{d2:fractional}
\end{equation}
If \(A=R/W+c/(Wz)\), \(E/\eta=c/W+c/(Wz)-c/(WH)\), and \(z>c>0\), where \(R\) is an integer, \(W\) is a positive integer, and \(H>0\), then \eqref{d2:fractional} implies \(0\le[R]_W<c\): since \(c/(Wz)<1/W\), taking the fractional part introduces no carry, and cancellation gives \([R]_W\le c-c/H<c\). We use this observation below with \(H=h+1\).

We also use the normalization bound. Put \(\lambda=T/U\). Then
\begin{equation}
h+1=z+\lambda m,\quad n=\frac{z^2}{m}+\lambda(z+1),\quad z\ge m.
\label{d2:z-bound}
\end{equation}
Indeed \(a=z^2/m+(\lambda-1)z+\lambda+1-\lambda m\) is strictly increasing for \(z\ge0\), and its value at \(z=m-1\) is \(1/m\), whereas \(a\ge2\). We use the following three-term Dedekind reciprocity identity:
\begin{equation}
12s(Ad-1,Kd)=3-\frac dK-\frac2{Kd}+12s(A,K)+12s(d-A^{-1},K),
\label{d2:three-term}
\end{equation}
where \(\gcd(A,K)=\gcd(Ad-1,Kd)=1\), and \(A^{-1}\) is taken modulo \(K\). This identity is an application of the Rademacher cocycle for Dedekind sums; see Rademacher--Grosswald~\cite[Chapter~4]{RademacherGrosswald1972}. We include the matrix calculation in the normalization used here. Choose \(s\) with \(As\equiv1\pmod K\), put \(M_1=\left(\begin{smallmatrix}A&(As-1)/K\\K&s\end{smallmatrix}\right)\), and choose a determinant-one matrix \(M_2\) with first column \((d-s,K)^{\mathsf{t}}\). The first column of \(M_1M_2\) is \((Ad-1,Kd)^{\mathsf{t}}\). The cocycle identity, whose three lower-left entries are positive, gives \eqref{d2:three-term}; its rational correction is \((-d^2-2)/(Kd)\). The coprimality conditions hold in each application below.

\begin{lemma}\label{d2:reciprocity-closure}
Let \(p,q\ge3\), write \(p=gP\), \(q=gQ\), with \(\gcd(P,Q)=1\), and set \(W=gPQ-P-Q>1\). Then
\[
12s(P,W)+12s(Q,W)=p+q-11+\frac{g+P+Q}{W}.
\]
For normalized data satisfying \eqref{d2:base}, suppose that positive integers \(b_0,r\) satisfy
\[
m=Wb_0,\quad h+1=b_0r,\quad r\equiv P\text{ or }Q\pmod W,\quad b_0\equiv g\pmod W.
\]
Then \(b_0=g\), and the data are of the form
\begin{equation}
\begin{gathered}
m=U,\quad h=s_0-1+kU,\quad s_0\in\{p,q\},\quad k\ge3,\\
n=\frac{h^2-(pq-2)h+1}{U},\quad a=n-h.
\end{gathered}
\label{d2:ordinary-data}
\end{equation}
\end{lemma}
\begin{proof}
Both \(P,Q\) are units modulo \(W\), and \((g-P^{-1})^{-1}\equiv Q\pmod W\). The continued fraction with two entries
\[
\frac{gW}{gP-1}=[gQ-1,gP-1]
\]
gives \(12s(gP-1,gW)=p+q-8+(p+q-2)/(gW)\). Applying \eqref{d2:three-term} with \((A,K,d)=(P,W,g)\) proves the first identity. For \(r\equiv P\pmod W\), apply it again with \(d=b_0\); the congruence \(b_0\equiv g\pmod W\) gives
\[
12s(h,m)=p+q-8+\frac{g+P+Q-b_0}{W}-\frac2m.
\]
The required value in \eqref{d2:base} is \(p+q-8+(P+Q)/W-2/m\), so \(b_0=g\). The other residue is symmetric. Thus \(m=U\), \(h\equiv p-1\) or \(q-1\pmod U\), and \eqref{d2:base} gives the formula for \(n\). Since \(U>p,q\), writing \(h=s_0-1+kU>0\) gives \(k\ge0\). For \(s_0=q\),
\[
a=Uk(k-2)+(q-p)k-2q+1.
\]
This is negative at \(k=0,1,2\) and at least two for every \(k\ge3\); interchange \(p,q\) for the other residue. Thus the normalization conditions hold precisely for \(k\ge3\). Conversely these formulas give integral normalized data satisfying the numerical equations: \(\gcd(h,U)=1\), and \(nm=hb+1\) gives \(\gcd(n,a)=1\) and \(ab\equiv1\pmod n\).
\end{proof}

For the cases \(p,q\ge3\), use \(p=gP\), \(q=gQ\), \(\gcd(P,Q)=1\), and set \(s=P+Q\), \(W=gPQ-s\), \(L=gPQ\), and \(\delta=\gcd(g,s)\). Thus \(U=gW\), \(T=gL\), and \(\gcd(W,L)=\delta\). Integrality of \(Lm/W\), together with \eqref{d2:divisibility}, gives an integer \(d>0\) such that
\begin{equation}
 m=Wd/\delta,\quad h+1=dr/\delta,\quad
 \delta Wn=dr(r-L)+\delta L,\quad Wn-L=rz,
\label{d2:nd1}
\end{equation}
where \(r\) is the canonical index. If \(W>1\) and \(r\equiv P\) or \(Q\pmod W\), one has \(\gcd(r,\delta)=1\), since \(\delta\mid W\) and \(\gcd(PQ,\delta)=1\). Thus \(d=\delta b_0\). Reducing the equation for \(n\) modulo \(W\) gives \((g-b_0)PQ\equiv0\), hence \(b_0\equiv g\pmod W\), since \(\gcd(PQ,W)=1\). Lemma~\ref{d2:reciprocity-closure} then gives \eqref{d2:ordinary-data}. Thus these two residues suffice for the ordinary family whenever \(W>1\).

\subsubsection{The pairs with integral ratio}\label{d2:integral}
Suppose \(T/U\) is an integer \(k\). The equation \(1/p+1/q=1-1/k\) has precisely the following six unordered solutions:
\[
(2,3),(2,4),(3,3),(2,6),(3,6),(4,4).
\]
Indeed \(2\le k\le6\), and the equation becomes \(((k-1)p-k)((k-1)q-k)=k^2\), which gives these six pairs. Integrality in \eqref{d2:base} gives \(m\mid(h+1)^2\). For \((2,3),(2,4),(3,3)\), equation \eqref{d2:divisibility} gives \(m\mid h+1\). Thus \(12s(h,m)=-m+3-2/m\), and \eqref{d2:base} gives \(m=1,2,3\), respectively. For \((2,6),(4,4)\), the divisibility is \(m\mid2(h+1)\). If \(m\mid h+1\), the same calculation gives \(m=1,2\), respectively. Otherwise write \(m=2d'\) and \(h+1\equiv d'\pmod{2d'}\). The condition \(m\mid(h+1)^2\) forces \(d'\) to be even, so \(m=4d\) and \(h\equiv2d-1\pmod{4d}\). In this case
\[
12s(2d-1,4d)=3-d-\frac2{4d}.
\]
For \(d=1\) the negative continued fraction is \([4]\); for \(d\ge2\) it is \([3,2^{d-2},3]\). Equation \eqref{d2:base} gives \(d=1\) for \((2,6)\), and \(d=2\) for \((4,4)\). For \((3,6)\), the divisibility is \(m\mid3(h+1)\). Either \(m\mid h+1\), giving \(m=1\), or \(m=9d\) and \(h\equiv3d-1\) or \(6d-1\pmod{9d}\). The extra factor of three follows from \(m\mid(h+1)^2\). For each of these two residue classes,
\[
12s(h,9d)=3-d-\frac2{9d}.
\]
The fractions are \([5,2]\) and \([2,5]\) for \(d=1\), and \([4,2^{d-2},3,2]\) and its reverse for \(d\ge2\). Thus \eqref{d2:base} gives \(d=1\). The solutions with \(m=U\) and \((p,q)\ne(2,3)\) give the second and fourth rows of Theorem~\ref{d2:classification}: their residues are \(h\equiv1\pmod{q-2}\) when \(p=2\), and \(h\equiv p-1\) or \(q-1\pmod U\) otherwise. Substitution in \eqref{d2:base} gives the stated formulas, and \(a=n-h\ge2\) gives \(k\ge4\) and \(k\ge3\), respectively. The remaining solutions are
\[
\begin{array}{c|c|c|c}
(p,q)&m&h&n\\\hline
(2,3)&1&h\ge6&h^2-4h+1\\
(2,6)&1&h\ge3&h^2-h+1\\
(3,6)&1&h\ge2&h^2+1\\
(4,4)&2&h\ge5\text{ odd}&(h-1)^2/2.
\end{array}
\]
\subsubsection{One arm with Seifert invariants \((2,1)\)}\label{d2:order-two}
The orders \(q=3,4,6\) have been treated. For the remaining orders put \(U=q-2\), \(T=2q\), \(\lambda=T/U\), and retain \(z=b-1\). Equations \eqref{d2:base} and \eqref{d2:z-bound} become
\begin{equation}
n=\frac{z^2}{m}+\lambda(z+1),\quad 12s(h,m)=U-3+\frac4U-\frac2m.
\label{d2:st1}
\end{equation}
Set
\begin{equation*}
\alpha=\frac{n-2}{z}=\frac{h+1}{m}+\varepsilon,\quad \varepsilon=\frac4{Uz},\quad \rho=\frac2{q\eta}=\frac4U+\varepsilon-\frac4{U(h+1)}.
\end{equation*}
These identities follow from \(Un-T=U(h+1)z/m\). We apply the fractional-part condition \eqref{d2:fractional} to suitable integers modulo \(L\).
\paragraph{Odd values of \(q\ge5\).}
Here \(U\) is odd and \(\gcd(T,U)=1\). Integrality of \(Q=Tm/U\) gives \(m=Ud\) for an integer \(d>0\). The index formula gives \(h+1=dr\), and hence
\begin{equation}
 Un=dr(r-T)+T.
\label{d2:st5}
\end{equation}
Use \(k=q+1\), which is even, so \(E=1/q\). The inequality \(E<\eta\) holds: it is equivalent to \(n>2q/(q-4)\), and follows from \(z\ge m\ge U\) and the first equation in \eqref{d2:st1}, including \(q=5\). The lower endpoint of \eqref{d2:forbidden} is
\[
 A=\frac{r+\alpha}{2}
 =\frac{(U+1)r}{2U}+\frac\varepsilon2.
\]
Since \(z\ge U\ge3\), the observation following \eqref{d2:fractional}, with \(W=U\), \(R=(U+1)r/2\), and \(c=2\), gives \(j=[(U+1)r/2]_U<2\). The case \(j=0\) would give \(r\equiv0\pmod U\), contradicting \eqref{d2:st5}, since \(T\equiv4\pmod U\) and odd \(U\ge3\) does not divide four. Thus \(j=1\), equivalently \(r\equiv2\pmod U\). Equation \eqref{d2:st5} then gives \(d\equiv1\pmod U\). Modulo \(m=Ud\), we have \(h\equiv2d-1\). The three-term Dedekind reciprocity identity, with \(A=2,K=U\), gives
\[
 12s(h,m)=3-\frac dU-\frac2m
 +12s(2,U)+12s(d-2^{-1},U).
\]
Because \(d\equiv1\pmod U\), the last Dedekind sum equals \(s(2,U)\). Using
\[
 12s(2,U)=\frac U2-3+\frac5{2U}\quad(U\text{ odd})
\]
therefore gives
\[
 12s(h,m)=U-3+\frac{5-d}{U}-\frac2m.
\]
Comparison with \eqref{d2:st1} forces \(d=1\). Hence \(m=U\) and \(h\equiv1\pmod U\).

\paragraph{Even values of \(q\ge8\).}
Write \(q=2Q_0\), \(V=Q_0-1\), so \(U=2V\) and \(V\ge3\). Integrality of \(Q=2Q_0m/V\) forces \(d=2m/V\) to be an integer, since \(\gcd(Q_0,V)=1\). The index formula gives
\begin{equation}
 h+1=\frac{dr}{2},\quad m=\frac{Vd}{2},\quad
 2Vn=dr(r-2Q_0)+4Q_0.
\label{d2:st6}
\end{equation}
Use \(k=2\), so \(E=2/q=1/Q_0\). Again \(E<\eta\): it is equivalent to \(n>2q/(q-6)\); the bounds in \eqref{d2:st1} imply this for all \(q\ge8\). The lower endpoint in \eqref{d2:forbidden} is \(A=\alpha=r/V+\varepsilon\). Here \(z>2\). Indeed \(z\ge m\), while \(V\mid2m\), so \(z\le2\) could only give \(V=4,m=2,z=2\); \eqref{d2:st1} would give the noninteger value \(n=19/2\). The observation following \eqref{d2:fractional}, with \(W=V\), \(R=r\), and \(c=2\), gives \(j=[r]_V<2\). If \(j=0\), \eqref{d2:st6} first forces \(V\mid4\), hence \(V=4\). But \(r\equiv0\pmod4\) makes \(dr(r-10)\) divisible by eight, whereas \(4Q_0=20\) is not, contradicting \eqref{d2:st6}. Therefore \(r\equiv1\pmod V\). The integer \(d\) must be even: if odd, integrality of \(m=Vd/2\) forces \(V\) even, and integrality of \(h+1=dr/2\) forces \(r\) even, contrary to \(r\equiv1\pmod V\). Write \(d=2b_0\). Then \(m=Vb_0\), \(h+1=b_0r\), and \eqref{d2:st6} gives \(b_0\equiv2\pmod V\). Modulo \(m\), \(h\equiv b_0-1\). Three-term reciprocity, now with \(A=1,K=V\), gives
\[
 12s(h,m)=3-\frac{b_0}{V}-\frac2m
 +12s(1,V)+12s(b_0-1,V)
 =2V-3+\frac{4-b_0}{V}-\frac2m.
\]
The required value \eqref{d2:st1} is \(2V-3+2/V-2/m\); hence \(b_0=2\). Thus \(m=2V=U\), and \(h+1=2r\) gives \(h\equiv1\pmod U\).

Thus every order \(q\ge4\), apart from the additional solution already recorded at \(q=6\), has
\[
m=q-2,\quad h\equiv1\pmod{q-2},\quad n=\frac{h^2-(2q-2)h+1}{q-2}.
\]
Writing \(h=(q-2)k+1\), normalization is equivalent to \(k\ge4\), and
\[
n=(q-2)k(k-2)-2,\quad a=(q-2)k(k-3)-3.
\]

\subsubsection{The case of coprime \(p\) and \(q\)}
Let \(3\le p<q\) and \(\gcd(p,q)=1\). In \eqref{d2:nd1}, one has \(g=\delta=1\), \(W=U\), \(L=T\), and \(m=Ud\). Normalization \(a\ge2\) gives \(z\ge m\), by \eqref{d2:z-bound}. In particular \(z\ge U>q>p\), and
\begin{equation*}
 n=\frac{z^2}{m}+\frac TU(z+1)>T.
\end{equation*}
Recall that \(\eta=U/T-1/n\).

Choose \(j\in\{1,\ldots,p-1\}\) so that \(k=jq+1\equiv0\pmod p\). Then \(0<k<L\) and
\[
 E=\{k/p\}+\{k/q\}=1/q<\eta.
\]
The last inequality follows from \(n>T\) and \(U-p\ge1\). The interval \eqref{d2:forbidden} has lower endpoint
\[
 A=\frac{vk}{u}+\frac{v}{uq\eta}
   =\frac{jr}{p}+\frac{\alpha_p}{p},
 \quad \alpha_p=\frac{(p-1)n-p}{z}.
\]
The numerical equations simplify this to
\begin{equation*}
 \alpha_p=\frac{(p-1)r}{U}+\varepsilon_p,
 \quad \varepsilon_p=\frac{p^2}{Uz},
 \quad
 \frac1{q\eta}=\frac pU+\frac{\varepsilon_p}{p}
                      -\frac p{U(h+1)}.
\end{equation*}
Set \(C=(jU+p-1)/p\), an integer since \(U\equiv-q\pmod p\). Then \(A=Cr/U+p/(Uz)\), so the observation following \eqref{d2:fractional}, with \(W=U\), \(R=Cr\), and \(c=p<z\), gives \([Cr]_U<p\). The residue zero is impossible: \(C\equiv(p-1)p^{-1}\pmod U\) is a unit, so it would imply \(r\equiv0\pmod U\), contrary to \eqref{d2:nd1} and \(\gcd(T,U)=1\). Since \((p-1)(q-1)=U+1\), we obtain
\begin{equation*}
 r\equiv k_1q\pmod U,\quad 1\le k_1\le p-1.
\end{equation*}
Indeed \(C^{-1}\equiv p(q-1)\equiv q\pmod U\). Repeating the same argument with \(p,q\) interchanged gives
\begin{equation*}
 r\equiv k_2p\pmod U,\quad 1\le k_2\le q-1.
\end{equation*}
The integer \(k_1q-k_2p\) is divisible by \(U\) and lies in \([-U,U]\). It cannot be zero, since \(\gcd(p,q)=1\) and \(0<k_1<p\). Its value \(U\) occurs only at \((k_1,k_2)=(p-1,1)\), and its value \(-U\) only at \((1,q-1)\). Therefore
\begin{equation*}
 r\equiv p\text{ or }q\pmod U.
\end{equation*}
The reduction following \eqref{d2:nd1} gives \eqref{d2:ordinary-data}.

\subsubsection{The case \(\gcd(p,q)>1\) and \(p\nmid q\)}
Suppose \(g\ge2\) and \(2\le P<Q\), retaining the notation and equations \eqref{d2:nd1}.

We first prove
\begin{equation}
 z>s,\quad W(W-s)>\delta L.
\label{d2:nd2}
\end{equation}
For the first, \(z\ge W/\delta\). If \(\delta=1\), then \(W\ge2PQ-P-Q>s\). If \(\delta\ge2\), then \(W/\delta\ge PQ-s/\delta\). For \(s\ge8\), the inequality \(PQ\ge2s-4\) gives \(PQ-s/\delta\ge3s/2-4\ge s\); equality would require \(s=8\), \((P,Q)=(2,6)\), which is not coprime. For \(s=7\), one has \(\delta=7\) and \(PQ\ge10\), so the bound is strict. The only smaller coprime pair is \((P,Q)=(2,3)\), with \(s=5\). Here \(\delta>1\) means \(\delta=5\), and \(W/\delta\ge5\), with equality only if \(g=5\). The remaining possibility \(m=z=5\) would give \(n=61/5\), impossible. This proves \(z>s\). For the second bound, if \(g\ge3\), then
\[
 \frac{W(W-s)}{L}>\frac23(g-2)PQ
 \ge4(g-2)\ge g\ge\delta.
\]
If \(g=2\), set \(v=PQ-P-Q\ge1\); then \(W=PQ+v\) and \(W-s=2v\). For \(\delta=1\), their product is at least \(2PQ+2>L\). For \(\delta=2\), both \(P,Q\) are odd and \(v\ge2\), so their product is at least \(4PQ+8>2L\). This proves \eqref{d2:nd2}. We also have \(W>q>p\).

First take the residue index \(k=1\), for which \(E=1/p+1/q=s/L\). Since \(n>m\ge W/\delta>L/(W-s)\), we have \(0<E<\eta=W/L-1/n\). The lower endpoint of the forbidden interval \eqref{d2:forbidden} is
\begin{equation}
 A=\frac{n-1}{z}=\frac rW+\frac{s}{Wz},
 \quad \frac E\eta=\frac sW+\frac{s}{Wz}-\frac{s}{W(h+1)}.
\label{d2:nd3}
\end{equation}
The observation following \eqref{d2:fractional}, with \(c=s<z\), gives \([r]_W<s\). The residue zero would imply \(W\mid L\) by \eqref{d2:nd1}, impossible since \(W>\delta=\gcd(W,L)\). Thus
\begin{equation}
 [r]_W\in\{1,\ldots,s-1\}.
\label{d2:nd4}
\end{equation}

Choose \(j\in\{1,\ldots,P-1\}\) with \(jQ+1\equiv0\pmod P\), and put \(k=g(jQ+1)\). Then \(E=1/Q\), and \(0<E<\eta\). Indeed the numerical equation and \(z\ge m\) give \(n>2z\), while \eqref{d2:nd2} gives \(z>L/(W-s)\). Since
\[
 2(W-q)\ge W-s>0
\]
(the difference is \(L-2q=gQ(P-2)\)), we obtain \(n>L/(W-q)\). Thus \(E<\eta\). The same argument applies after interchanging \(P,Q\), when \(E=1/P\). The lower endpoint \(A\) and the ratio \(E/\eta\) are
\begin{equation*}
\begin{aligned}
 A&=\frac{Cr}{W}+\frac{p}{Wz},&
 C&=\frac{jW+p-1}{P}\in\mathbb{Z},\quad C^{-1}\equiv Q\pmod W,\\
 \frac E\eta&=\frac pW+\frac p{Wz}-\frac p{W(h+1)}.
\end{aligned}
\end{equation*}
Since \(p/(Wz)<1\), there is an integer \(i\equiv Cr\pmod W\) such that
\[
 \{A\}=\frac{i+p/z}{W}.
\]
Condition \eqref{d2:fractional} gives \(-p/z<i<p\). If \(i\le0\), then
\[
 -W<-L/z-[r]_W<iQ-[r]_W<0
\]
by \eqref{d2:nd2} and \eqref{d2:nd4}, contradicting \(iQ-[r]_W\equiv0\pmod W\). Hence \(1\le i\le p-1\), and
\begin{equation*}
 r\equiv iQ\pmod W.
\end{equation*}
Interchanging \(P,Q\) gives \(r\equiv j'P\pmod W\) with \(1\le j'\le q-1\). The difference \(iQ-j'P\) is a multiple of \(W\) lying in \([-W,W]\). At the two endpoints it yields \(r\equiv P\) or \(Q\). If the difference is zero, then \(i=cP\), \(j'=cQ\), with \(1\le c\le g-1\). But in that case \([r]_W=cPQ\), which lies strictly between \(s\) and \(W\), contradicting \eqref{d2:nd4}. We conclude
\begin{equation*}
 r\equiv P\text{ or }Q\pmod W.
\end{equation*}

The reduction following \eqref{d2:nd1} gives \eqref{d2:ordinary-data}.

\subsubsection{The case \(p\mid q\)}
The Dedekind sum satisfies
\[
 12s(h,m)\le m-3+2/m.
\]
Together with \eqref{d2:base}, this gives
\begin{equation}
 m+4/m\ge p+q-6+T/U.
\label{d2:de1}
\end{equation}
By Cauchy--Schwarz, \(s(h,m)\le\sum_{j=1}^{m-1}(j/m-1/2)^2=(m-1)(m-2)/(12m)\). The bound also holds at \(m=1\). As before \(z=b-1\ge m\), and the numerical equations imply
\begin{equation}
 n=\frac{z^2}{m}+\lambda(z+1)>2z,
 \quad \lambda=T/U>1.
\label{d2:de2}
\end{equation}

\paragraph{The case \(p\mid q\) and \(p<q\).}
Let \(P=1\), \(g\ge3\), and \(Q\ge2\), excluding \((p,q)=(3,6)\). Then \(W=(g-1)Q-1\), \(L=gQ\), and \(s=Q+1\). Inequality \eqref{d2:de1} forces \(m>s\). Indeed, on \(1\le m\le s\), the maximum of \(m+4/m\) is \(\max(5,s+4/s)\), whereas the required lower bound is \(gs-6+L/W\). It strictly exceeds that maximum for \(g\ge3,s\ge4\), and for \(g\ge4,s=3\). The excluded pair is \(g=3,s=3\), that is, \((p,q)=(3,6)\). Hence
\begin{equation*}
 z>s.
\end{equation*}
We also have
\begin{equation}
 z>\frac{L}{W-s}.
\label{d2:de5}
\end{equation}
For \(g\ge5,Q\ge2\), for \(g=4,Q\ge3\), and for \(g=3,Q\ge5\), this follows from \(z>s\) and
\[
 s(W-s)-L=(g-2)Q^2-4Q-2\ge0.
\]
The three remaining pairs are \((p,q)=(3,9),(3,12),(4,8)\). Divisibility of \(m\) and \eqref{d2:de1} respectively give \(m\ge10,14,10\), whereas \(L/(W-s)\) is respectively \(9,6,4\). Thus \eqref{d2:de5} holds for every pair under consideration.

The character \(k=1\) has \(E=s/L\). By \eqref{d2:de2}--\eqref{d2:de5}, \(n>L/(W-s)\), so \(E<\eta\). Its threshold is \eqref{d2:nd3}, so the observation following \eqref{d2:fractional}, with \(c=s<z\), gives
\begin{equation}
 1\le [r]_W\le Q.
\label{d2:de6}
\end{equation}
The residue zero is impossible because it would imply \(W\mid L\), while \(\gcd(W,L)=\delta<W\).

Next use \(k=g\), so \(E=1/Q\). This also satisfies \(E<\eta\): since
\[
 2(W-g)\ge W-s>0
\]
(the difference is \(g(Q-2)\)), equations \eqref{d2:de2} and \eqref{d2:de5} give \(n>L/(W-g)\). The lower endpoint of the excluded interval is
\begin{equation*}
 A=\frac{(g-1)r}{W}+\frac g{Wz},
 \quad
 \frac E\eta=\frac gW+\frac g{Wz}-\frac g{W(h+1)}.
\end{equation*}
Choose the integer \(i\equiv(g-1)r\pmod W\) such that \(\{A\}=(i+g/z)/W\). Condition \eqref{d2:fractional} gives \(-g/z<i<g\). Since \((g-1)Q=W+1\), we have \(iQ\equiv r\pmod W\). If \(i\le0\), \eqref{d2:de5}--\eqref{d2:de6} would give
\[
 -W<-L/z-[r]_W<iQ-[r]_W<0,
\]
a contradiction. Thus \(1\le i\le g-1\). If \(i\le g-2\), then \(0<iQ<W\), so \([r]_W=iQ\le Q\) forces \(i=1\). If \(i=g-1\), then \(iQ=W+1\), giving \([r]_W=1\). Therefore
\begin{equation*}
 r\equiv1\text{ or }Q\pmod W.
\end{equation*}
The reduction following \eqref{d2:nd1} gives \eqref{d2:ordinary-data}; the excluded pair \((3,6)\) was treated above.

\paragraph{The case \(p=q=g\ge5\).}
Now \(W=g-2\), \(L=g\), \(\delta=\gcd(g,2)\). Equation \eqref{d2:de1} gives \(m>2\), hence \(z>2\) and \(n>6\). The character \(k=1\) has \(E=2/g<\eta\), since \(n>g/(g-4)\). Its threshold is
\[
 0<\left\{\frac rW+\frac2{Wz}\right\}
 \le\frac2W+\frac2{Wz}-\frac2{W(h+1)}.
\]
The same observation with \(c=2<z\) gives \([r]_W<2\). Zero is impossible because \(W\mid L\) would imply \(g-2\mid2\), contrary to \(g\ge5\). Thus \(r\equiv1\pmod W\). The reduction following \eqref{d2:nd1}, with \(P=Q=1\), gives \eqref{d2:ordinary-data}.

The cases above exhaust all pairs \(2\le p\le q\). Together with the case \(\eta=0\), they prove Theorem~\ref{d2:classification}.

\subsubsection{The continued fractions}\label{d2:comparison}
The data in Theorem~\ref{d2:classification} give the rows with central self-intersection \(-2\) of Table~\ref{table:threearm}. This follows by multiplying the continued-fraction matrices \(M(b)=\left(\begin{smallmatrix}b&-1\\1&0\end{smallmatrix}\right)\), using \(M(2)^k=\left(\begin{smallmatrix}k+1&-k\\k&1-k\end{smallmatrix}\right)\). In particular, the family with \(p=2\) and \(m=U\) has last arm \([2^{k-4},3,2^{q-4},k+1]\). For \(3\le p\le q\), the residue class \(h\equiv q-1\pmod U\) gives the last arm
\[
[2^{k-3},3,2^{q-4},3,2^{p-3},k+1]
\]
when \(q\ge4\); the other residue is obtained by interchanging \(p,q\). If the order of the arm with one vertex defining the residue is three, the continued fraction is \([2^{k-3},4,2^{q-3},k+1]\). The third, fifth, and sixth rows of Theorem~\ref{d2:classification} give the last three families.
\subsection{Three arms and central weight three}\label{sec:d3}

Use the compactification of Section~\ref{sec:compactification}, and write \(D_i=S_i+T_i\), where \(S_i\) meets \(C\) and \(T_i\) may be empty. Here \(C^2=0\), \(C\cdot S_i=1\), and \(C\cdot T_i=0\).

\begin{theorem}\label{d3:classification}
Let a weighted homogeneous normal surface singularity with three arms and central self-intersection \(-3\) admit a rational homology disk smoothing. Up to permutation of the arms, its resolution graph belongs to row~1(b) or~1(c) of Table~\ref{table:threearm}.
\end{theorem}

\begin{proof}
By Lemma~\ref{lem:rational-curve-system}, \(|C|\) defines a \(\mathbb{P}^1\)-fibration over \(\mathbb{P}^1\). The surface \(Y\) is rational, the curves \(S_1,S_2,S_3\) are disjoint sections, and each nonempty \(T_i\) is contained in one fiber.

We first contract vertical \((-1)\)-curves while preserving the disjointness of the sections. In a reducible fiber, suppose that every \((-1)\)-curve meets at least two sections. Such a curve has fiber multiplicity one. Since each section has intersection number one with the fiber, at most one \((-1)\)-curve in that fiber can meet two sections. For a fiber \(\sum_v m_vE_v\), with \(E_v^2=-b_v\), adjunction gives
\[
\sum_v m_v(2-b_v)=2.
\]
Under this supposition, the left side is at most one. Thus every reducible fiber contains a \((-1)\)-curve meeting at most one section. Contracting such a curve preserves the disjointness of the sections. Repeating the contraction gives a ruled surface with three disjoint sections. Write their classes on \(\mathbb{F}_n\) as \(S_0+a_iF\), where \(S_0^2=-n\) and \(F\) is a fiber. Their pairwise intersections give \(a_i+a_j=n\) for \(i\ne j\), so \(a_i=n/2\) and each section has square zero. If \(n>0\), an irreducible section different from \(S_0\) must have \(a_i\ge n\), a contradiction. Hence \(n=0\), and the three sections are constant sections of \(\mathbb{F}_0\). We obtain a morphism
\[
Y\to\mathbb{F}_0
\]
which contracts only vertical curves.

Let \(u_F\) be the number of components of a reducible fiber \(F\) which do not belong to \(D\). Since \(\rho(Y)=N_D+1\) and \(D\) has \(N_D-3\) vertical components, the formula for the Picard number of a rational fibration gives
\[
\sum_F(u_F-1)=2.
\]
Every reducible fiber contains a component outside \(D\): its intersection matrix is negative semidefinite, whereas every subconfiguration of \(D\) is negative definite.

Suppose that \(u_F=1\), and denote the unique component outside \(D\) by \(E\). It is the unique \((-1)\)-curve in the fiber. If its multiplicity is \(m\), adjunction gives
\[
m=2+\sum_{v\in D\cap F}(b_v-2)m_v\ge2.
\]
No section meets \(E\). Each section therefore meets the nonempty subchain \(T_i\) attached to it, and the component adjacent to the section has multiplicity one. Since the fiber has a tree as its dual graph, \(E\) meets each \(T_i\) at one component. Let \(A_i\) be the positive intersection matrix of \(T_i\), let \(y_i\) be its multiplicity vector, and let \(j_i\) index the component meeting \(E\). Then
\[
A_iy_i=me_{j_i},\quad (y_i)_1=1.
\]
Thus \(\det A_i=m\delta_i\), where \(\delta_i\) is a positive integral cofactor of \(A_i\). The adjugate of the negative fiber intersection matrix is the outer product of the primitive multiplicity vector with itself. This identity holds for a smooth fiber and is preserved by a point blowup. Deleting the row and column indexed by \(E\) therefore gives
\[
m^2=\prod_{i=1}^{3}\det A_i=m^3\delta_1\delta_2\delta_3,
\]
which contradicts \(m\ge2\). Hence \(u_F\ge2\) for every reducible fiber. There is either one reducible fiber with \(u_F=3\), or two reducible fibers with \(u_F=2\) for each fiber.

Let \(q\) be the number of fibers of \(\mathbb{F}_0\) whose strict transforms belong to \(D\). The morphism \(Y\to\mathbb{F}_0\) is a composition of \(N_D-1\) point blowups. Among the \(N_D-3\) vertical components of \(D\), precisely \(q\) are strict transforms of fibers of \(\mathbb{F}_0\). Consequently the number of exceptional prime curves outside \(D\) is
\[
(N_D-1)-(N_D-3-q)=q+2.
\]
Each exceptional fiber contains a \((-1)\)-curve created by its last blowup, and this curve does not belong to \(D\). Each section has square zero on \(\mathbb{F}_0\) and square at most \(-2\) on \(Y\), so each section contains a blowup center. These centers lie in three distinct exceptional fibers. Thus \(q\ge1\). Since there are at most two reducible fibers, \(q\le2\).

Suppose that \(q=1\) and there is one reducible fiber. Let \(F\) be the fiber of \(\mathbb{F}_0\) whose strict transform belongs to \(D\), and assume that \(F\subset T_1\). There are exactly three positive-dimensional fibers of \(Y\to\mathbb{F}_0\), lying over one point on each section, and each contains exactly one exceptional curve outside \(D\). Lemma~\ref{lem:one-outside-cluster} applies at \(F\cap S_2\) and \(F\cap S_3\). Since \(S_2^2,S_3^2\le-2\), the exceptional subchain adjacent to \(F\) in each exceptional fiber is nonempty. Both subchains belong to \(D_1\). The curve \(F\) is also joined to \(S_1\) by a path in \(D_1\). Hence \(F\) has at least three neighbors in \(D_1\), contrary to the linearity of \(D_1\).

Suppose that \(q=1\) and there are two reducible fibers. Denote the fiber whose strict transform belongs to \(T_1\) by \(F\), and denote the other fiber of \(\mathbb{F}_0\) by \(G\). The points \(F\cap S_2\) and \(F\cap S_3\) must be blown up to separate the corresponding chains. There are only three exceptional prime curves outside \(D\). The reducible fiber over \(G\) requires another exceptional fiber, and the condition \(S_1^2\le-2\) requires a center on \(S_1\). Thus the third exceptional fiber is over \(G\cap S_1\). No exceptional component in this exceptional fiber belongs to \(D\): it cannot belong to \(T_1\), which lies over \(F\), and it is not connected to either of the other sections. The exceptional fiber therefore consists of its unique exceptional curve outside \(D\). It is a single blowup, which gives \(S_1^2=-1\), a contradiction.

It follows that \(q=2\) and there are exactly two reducible fibers. Their strict transforms belong to different subchains \(T_i\); label them so that
\[
F\subset T_1,\quad G\subset T_2.
\]
The points
\[
F\cap S_2,\quad F\cap S_3,\quad G\cap S_1,\quad G\cap S_3
\]
must be blown up. These four exceptional fibers account for all four exceptional prime curves outside \(D\), and each contains precisely one such curve. Lemma~\ref{lem:one-outside-cluster} implies that each exceptional fiber has a linear dual graph. There are no centers over \(F\cap S_1\) or \(G\cap S_2\).

Since \(T_2\) lies over \(G\), the exceptional fiber over \(F\cap S_2\) has no exceptional subchain adjacent to \(S_2\). By Lemma~\ref{lem:one-outside-cluster}, its dual graph is
\[
F-[2^r]-E-S_2.
\]
The section \(S_2\) has no other blowup center, so \(S_2^2\le-2\) implies \(r\ge1\). Similarly, the exceptional fiber over \(G\cap S_1\) has dual graph
\[
G-[2^s]-E'-S_1,\quad s\ge1.
\]
Thus \(S_1^2=-(s+1)\) and \(S_2^2=-(r+1)\).

The curve \(F\) meets \(S_1\) and the nonempty exceptional subchain over \(F\cap S_2\). Since \(D_1\) is linear, the exceptional fiber over \(F\cap S_3\) has no exceptional subchain adjacent to \(F\). The same argument applies at \(G\cap S_3\). The unique exceptional curve outside \(D\) in each of these exceptional fibers therefore meets the strict transform of the original fiber. The subchain \(T_3\) is contained in one fiber. After interchanging \(F,G\) and the corresponding section labels, let \(k\ge0\) be its length over \(F\); there is no such subchain over \(G\). The self-intersections are
\[
S_3^2=-2,\quad F^2=-(k+2),\quad G^2=-2.
\]
The complementary chains, ordered starting at their section components, are
\[
D_1=[s+1,k+2,2^r],\quad D_2=[r+1,2^{s+1}],\quad D_3=[2^{k+1}],\quad r,s\ge1,\quad k\ge0.
\]
Hirzebruch--Jung duality gives row~1(c) when \(k=0\), with table parameters \(q=s-1\) and \(r_{\mathrm{table}}=r-1\). For \(k\ge1\), it gives row~1(b), with \((p,q,r_{\mathrm{table}})=(k-1,s-1,r-1)\).
\end{proof}

\subsection{Three arms and central weight four}
\label{sec:d4-plane-cap}

For central weight four, \(C^2=1\), and \(|C|\) maps the compactification birationally to \(\mathbb{P}^2\).

\begin{theorem}\label{d4:classification}
Let a weighted homogeneous normal surface singularity with three arms and central self-intersection \(-4\) admit a rational homology disk smoothing. Up to permutation of the arms, its resolution graph belongs to row~1(a) of Table~\ref{table:threearm}.
\end{theorem}

\begin{proof}
Use the compactification of Section~\ref{sec:compactification}, with \(C^2=1\), and denote the component of \(D_i\) meeting \(C\) by \(H_i\).

Lemma~\ref{lem:rational-curve-system} gives a birational morphism \(\phi:Y\to\mathbb{P}^2\) with \(\phi^*\mathcal{O}_{\mathbb{P}^2}(1)=\mathcal{O}_Y(C)\). Its exceptional curves are disjoint from \(C\). Thus \(\phi\) is an isomorphism near \(C\), and its image is a line \(\ell_\infty\).

Since \(C\cdot H_i=1\), the restriction of \(\phi\) maps \(H_i\) isomorphically to a line \(\ell_i\). These three lines are distinct because their intersection points with \(\ell_\infty\) are distinct. Every component of \(D-H_1-H_2-H_3\) is contracted. Since \(\rho(Y)=N_D+1\), the morphism has \(N_D\) exceptional prime curves, of which \(N_D-3\) belong to \(D\). Thus exactly three lie outside \(D\).

By Lemma~\ref{lem:one-outside-cluster}, every nonempty exceptional fiber contains a prime curve outside \(D\). All blowup centers lie outside \(\ell_\infty\).

Suppose that \(\ell_1,\ell_2,\ell_3\) pass through a common point \(P\). Their strict transforms are disjoint, so \(\phi\) factors through the blowup of \(P\). The pencil of lines through \(P\) induces a morphism \(g:Y\to\mathbb{P}^1\). The strict transform \(E\) of the first exceptional curve is a section of \(g\).

The fibers containing \(H_1,H_2,H_3\) are reducible: an irreducible fiber has square zero, whereas \(H_i^2\le-2\). Each contains a \((-1)\)-curve, and every such curve lies outside \(D\). It is also \(\phi\)-exceptional, since the only component in that fiber which maps to a curve in \(\mathbb{P}^2\) is \(H_i\). The three \(\phi\)-exceptional prime curves outside \(D\) therefore consist of one curve \(F_i\) in each of these fibers. In particular, \(E\) belongs to \(D\), and each of the three fibers has exactly one component outside \(D\), with \(F_i^2=-1\).

Write one of these fibers as \(G=\sum_jm_jG_j\). Its components are smooth rational curves, and \(-K_Y\cdot G=2\). Adjunction gives
\[
2=\sum_jm_j(2+G_j^2)\le m(F_i),
\]
because every component other than \(F_i\) belongs to \(D\). Thus \(m(F_i)\ge2\). A section meets one component of multiplicity one in each fiber, so \(E\) does not meet \(F_i\). It meets a component of \(D\) in each of the three fibers. Hence \(E\) has at least three neighbors in \(D\), contrary to the linearity of its connected component. The three lines therefore form a triangle.

Let \(P_{ij}=\ell_i\cap\ell_j\). Since \(D_i\) and \(D_j\) are disjoint, the exceptional fiber over \(P_{ij}\) is nonempty. Each of the three fibers contains a curve outside \(D\), so each contains precisely one such curve \(F_{ij}\), and there are no exceptional fibers over other points. Moreover, \(F_{ij}^2=-1\).

By Lemma~\ref{lem:one-outside-cluster}, the exceptional configuration over \(P_{ij}\), together with the two strict transforms, is a linear chain. The curve \(F_{ij}\) separates the subchains belonging to \(D_i,D_j\).

Put \(\varepsilon_{ij}=1\) if \(D_i-H_i\) has a nonempty subchain over \(P_{ij}\), and put \(\varepsilon_{ij}=0\) otherwise. Since \(H_i\) is an end component of \(D_i\),
\[
\sum_{j\ne i}\varepsilon_{ij}\le1.
\]
Lemma~\ref{lem:one-outside-cluster} implies that, if the blowups over \(P_{ij}\) decrease \(H_i^2\) by more than one, then \(\varepsilon_{ji}=1\).

The first blowup at each of the two vertices on \(\ell_i\) decreases its self-intersection from \(1\) to \(-1\). Since \(H_i^2\le-2\), a further decrease occurs at one of those vertices. Thus
\[
\sum_{j\ne i}\varepsilon_{ji}\ge1.
\]
Summing these inequalities and the preceding upper bounds gives equality in every inequality. The directed graph defined by \(i\to j\) when \(\varepsilon_{ij}=1\) has one incoming and one outgoing edge at each of its three vertices, and has no loops. It is therefore a directed three-cycle. At each \(P_{ij}\), exactly one of the two exceptional subchains is nonempty.

By Lemma~\ref{lem:one-outside-cluster}, this subchain is \([2^r]\) for some \(r\ge1\). The blowups decrease the self-intersection of the line component adjacent to this subchain by one, and that of the other line component by \(r+1\).

Choose the cyclic ordering and lengths of the subchains \(r_1,r_2,r_3\ge1\) so that
\[
D_i=[r_{i+1}+1,\underbrace{2,\ldots,2}_{r_i}],\quad i\in\mathbb{Z}/3\mathbb{Z}.
\]
These complementary fractions equal \(n_i/(n_i-a_i)\). Hirzebruch--Jung duality gives
\[
\frac{n_i}{a_i}=[\underbrace{2,\ldots,2}_{r_{i+1}-1},r_i+2].
\]
Putting \(p=r_1-1\), \(q=r_2-1\), and \(r=r_3-1\) proves the assertion.
\end{proof}

\section{The four-arm classification}
\label{four:section}

Contract the complementary divisor \(D\), and denote the resulting surface by \(Z\). It has Picard number one and four cyclic quotient points. Let \(\overline{C}\) be its ample boundary, and let \(\pi:Y\to Z\) be the minimal resolution. The identities
\[
\overline{C}^2=e,\quad K_Z\equiv\left(\frac{\eta}{e}-1\right)\overline{C}
\]
imply that \(Z\) is log del Pezzo if \(\eta<e\). We use the following consequence of the classification of Belousov--Hwang.

\begin{proposition}[Belousov--Hwang]\label{four:BH}
Let \(Z\) be a log del Pezzo surface with Picard number one and four cyclic quotient singularities. Suppose that at least one of its singularities is not a rational double point. Then exactly one is not a rational double point. On the minimal resolution \(\pi:Y\to Z\), write \(D=D'+D''\), where \(D''\) lies over that point. There is a smooth rational curve \(F\), not contracted by \(\pi\), such that
\begin{equation}\label{four:anticanonical-cycle}
F^2=-1,\quad F\cdot D''=2,\quad K_Y+F+D''\sim0,
\end{equation}
and \(D'\) is disjoint from \(F+D''\). The three rational double points have one of the following collections of types:
\[
2A_1+A_3,\quad 3A_2,\quad A_1+2A_3,\quad A_1+A_2+A_5.
\]
For the last three collections, if \(n\) is the order of the point corresponding to \(D''\), then
\begin{equation}\label{four:degree-invariant}
nK_Z^2=3,\quad 2,\quad 1,
\end{equation}
respectively.
\end{proposition}

\begin{proof}
Belousov--Hwang~\cite[Theorem~1.1]{BH} gives the singularity types. Cyclicity excludes their Section~3, whose hypothesis requires an exceptional component of valency three. Their Section~5 therefore gives a minimal curve with nonempty adjoint system. By their Lemmas~2.9 and~4.1, this is a smooth rational \((-1)\)-curve \(F\), and \(D=D'+D''\) with \(K_Y+F+D''\sim0\), \(F\cdot D''=2\), and \(F\cdot E=D''\cdot E=K_Y\cdot E=0\) for every component \(E\) of \(D'\). Hence \(D'\) consists of \((-2)\)-curves and is disjoint from \(F+D''\). Adjunction gives \(F\cdot E=2-\operatorname{val}_{D''}(E)\) on \(D''\), so each connected chain meets \(F\) twice. Thus \(D''\) is the full chain over the unique point which is not a rational double point, proving \eqref{four:anticanonical-cycle}.

Write \(D''=\sum_{i=1}^{s}E_i\), let \(M=-(E_i\cdot E_j)\), and put \(v=(F\cdot E_i)_{i=1}^{s}\). The Hirzebruch--Jung determinant formula gives \(\det M=n\). Pushing \eqref{four:anticanonical-cycle} to \(Z\) and taking the numerical pullback gives
\[
\pi^*(-K_Z)=F+\sum_{i=1}^{s}(M^{-1}v)_iE_i.
\]
Consequently,
\[
K_Z^2=-1+v^{\mathsf{T}}M^{-1}v,\quad
nK_Z^2=\left|\det\begin{pmatrix}-1&v^{\mathsf{T}}\\v&-M\end{pmatrix}\right|.
\]
The terminal cycles in Belousov--Hwang~\cite[Section~4, cases~(c) and~(d)]{BH} have squares \((-1,-2,-2)\), \((-1,-2)\), and \((-1,-3)\) for the respective collections \(3A_2\), \(A_1+2A_3\), and \(A_1+A_2+A_5\). Contracting a \((-1)\)-curve preserves the absolute determinant by the Schur complement formula, and preserves the anticanonical relation by pushforward. For a terminal cycle \(\Theta\), adjunction gives \((\Theta-E)\cdot E=2\) on every component \(E\). Thus distinct components meet once in the three-component cycle and twice in each two-component cycle. The terminal matrices are therefore
\[
\begin{pmatrix}-1&1&1\\1&-2&1\\1&1&-2\end{pmatrix},\quad
\begin{pmatrix}-1&2\\2&-2\end{pmatrix},\quad
\begin{pmatrix}-1&2\\2&-3\end{pmatrix}.
\]
Their determinants are \(3,-2,-1\), respectively. Thus their absolute determinants are \(3,2,1\), proving \eqref{four:degree-invariant}.
\end{proof}

\begin{theorem}\label{four:classification}\label{completion:intrinsic-bound}\label{four:central-weight}
Let a weighted homogeneous rational surface singularity with four arms admit a rational homology disk smoothing. Its compactification \(Z\) is log del Pezzo and has precisely one singular point which is not a rational double point. Moreover,
\[
d=3,\quad \sum_{i,j}(b_{ij}-3)=0,\quad h^1(S_X)=1.
\]
Up to permutation of the arms, its Seifert data are those in Table~\ref{table:fourarm}.
\end{theorem}

\begin{proof}
By Lemma~\ref{lat:star-reduction}, three arms are arms with one vertex and orders \(l_1,l_2,l_3\), with \(\sum_i1/l_i=1\). The fourth arm \(n/a\) has at least two components, so \(a>1\). Since \(d\ge3\),
\[
\eta=1-\frac1n,\quad e=d-1-\frac an,\quad e-\eta=d-2-\frac{a-1}{n}>0.
\]
Hence \(-K_Z\) is ample. The complementary chains of the arms with one vertex consist of \((-2)\)-curves. The complementary chain of \(n/a\) contains a component of self-intersection at most \(-3\), since \(a>1\). Proposition~\ref{four:BH} applies. Write \(D''=[b_1,\ldots,b_s]\) for this fourth complementary chain.

Put \(r_0=\sum_i(l_i-1)\). Noether's formula and \eqref{four:anticanonical-cycle} give
\[
9-r_0-s=K_Y^2=(F+D'')^2=1-\sum_{j=1}^{s}b_j+2s.
\]
Thus
\begin{equation*}
W(D'')=r_0-8=l_1+l_2+l_3-11.
\end{equation*}
Using \(W(n/a)+W(n/(n-a))=-2\), we obtain
\[
\sum_{i,j}(b_{ij}-3)=\sum_{i=1}^{3}(l_i-3)-2-W(D'')=0.
\]
Equation~\eqref{foundation:dimension-eq} gives \(h^1(S_X)=4-d\), whereas Proposition~\ref{prop:three-arm-lattice-vanishing} gives \(h^1(S_X)=1\). Therefore \(d=3\).

Set \(h=n-a\), and let \(m=3,2,1\) for the respective triples in the table. The boundary identities and \eqref{four:degree-invariant} give
\[
e=\frac{n+h}{n},\quad \eta=\frac{n-1}{n},\quad nK_Z^2=\frac{(h+1)^2}{n+h}=m.
\]
Consequently
\begin{equation*}
(h+1)^2=m(n+h).
\end{equation*}
Since \(m\) is squarefree, \(h+1=mv\) for an integer \(v\). It follows that
\[
n=mv(v-1)+1,\quad a=mv(v-2)+2.
\]
The inequalities \(0<a<n\) imply \(v\ge2\). Conversely, these expressions give normalized Seifert data for every \(v\ge2\), since \(\gcd(n,a)=\gcd(n,h)=\gcd(v,mv-1)=1\). Substituting \(p=v-2\) gives Table~\ref{table:fourarm}.
\end{proof}

\section{The cross ratios of the attachment points}\label{sec:cross-ratios}

We give another proof of the cross ratios determined by Fowler~\cite[Theorem~5(a)]{Fowler}. His proof uses the uniqueness of configurations on rational surfaces and explicit plane-curve constructions; see Fowler~\cite[Theorem~3.3.1 and Sections~3.5.4--3.5.6]{Fowler}. Here the attachment points become the singular parameters of anticanonical elliptic pencils.

The anticanonical cycle in Proposition~\ref{four:BH} allows us to contract the parameter-dependent part of the fourth complementary chain while preserving the marked central curve. The resulting elliptic pencils have fibers \(I_3^4\), \(I_4^2I_2^2\), or \(I_6I_3I_2I_1\). Beauville's classification~\cite{Beauville} fixes their singular parameters, explaining the independence of the cross ratios from the graph parameter.

Let \(P_1,P_2,P_3\) be the attachment points of the arms with one vertex on the central rational curve, in the order of Table~\ref{table:fourarm}, and let \(P_4\) be the attachment point of the fourth arm.

\begin{theorem}\label{thm:cross-ratio}
Let a weighted homogeneous surface singularity with a graph in Table~\ref{table:fourarm} admit a rational homology disk smoothing. Its ordered attachment points satisfy
\[
\begin{array}{c|c|c}
\text{orders}&\text{normalization of }(P_1,P_2,P_3)&\lambda=z(P_4)\\ \hline
(3,3,3)&(0,1,\infty)&\lambda^2-\lambda+1=0\\
(2,4,4)&(\infty,0,1)&\lambda=\frac12\\
(2,3,6)&(0,1,\infty)&\lambda=9.
\end{array}
\]
These values are independent of the parameter \(p\ge0\) of the graph. The two values in the first row are equivalent under permutation of the three arms with Seifert invariants \((3,1)\). Conversely, each indicated analytic structure admits a rational homology disk smoothing.
\end{theorem}

\begin{proof}
Propositions~\ref{prop:degree-one} and~\ref{prop:fibers}, proved below, identify the ordered attachment points with the singular parameters of the three elliptic pencils. The singular-parameter computations below give the stated cross ratios. For a weighted homogeneous surface singularity, the Seifert invariants and the attachment points determine the graded section ring \eqref{cqg:section-algebra}; see Pinkham~\cite{Pinkham1977}. Thus they determine the analytic structure under consideration. The converse follows from Theorem~\ref{cqg:main}, proved by the constructions in Sections~\ref{cqg:invariant-sections}--\ref{cqg:m}.
\end{proof}

\subsection{The projective compactification and a weak del Pezzo surface of degree one}

Use the compactification of Section~\ref{sec:compactification}, with \(C^2=-1\), and write \(D=D'+D''\) as in Proposition~\ref{four:BH}. The ordered attachment points on \(C\) have the same cross ratio as those on the original central curve. Continued-fraction duality gives
\[
\begin{array}{c|c|c}
\text{orders}&D'&D''\text{, oriented from }C\\ \hline
(3,3,3)&3A_2&[p+2,2,3,2^p]\\
(2,4,4)&A_1+2A_3&[p+2,3,2^p]\\
(2,3,6)&A_1+A_2+A_5&[p+3,2^p].
\end{array}
\]
By Proposition~\ref{four:BH}, there is a curve \(F\) satisfying \eqref{four:anticanonical-cycle} and disjoint from \(D'\).

\begin{proposition}\label{prop:degree-one}\label{lem:disjoint}
Contracting \(F\) and then the end subchain \(2^p\) gives a weak del Pezzo surface \(Y_0\) of degree one. The contraction is an isomorphism near \(C\) and preserves its four attachment points, its square \(-1\), and the chains \(D'\). The image \(D_0\) of \(F+D''\) is a reduced anticanonical divisor with the following components:
\begin{enumerate}
\item for \((3,3,3)\), three rational curves with squares \(-1,-2,-2\) and pairwise intersection numbers one;
\item for \((2,4,4)\), two rational curves with squares \(-1,-2\) and intersection number two;
\item for \((2,3,6)\), one irreducible rational curve with arithmetic genus one and square one.
\end{enumerate}
\end{proposition}

\begin{proof}
Adjunction and \eqref{four:anticanonical-cycle} give
\[
F\cdot C=(-K_Y-D'')\cdot C=1-1=0,\quad F\cdot D_j=2-\operatorname{val}_{D''}(D_j)
\]
for every component \(D_j\) of \(D''\). Since \(C^2=-1\), the first equality implies \(F\ne C\), and hence \(F\) is disjoint from \(C\). The second equality implies that \(F\) meets each end once and no interior component, or meets the sole component twice when \(D''\) has length one.

Contract \(F\), and then contract the successive \((-1)\)-curves arising from the end subchain \(2^p\). These curves are disjoint from \(C\). In the first two families, the end component with initial square \(-(p+2)\) has final square \(-1\), and the component with initial square \(-3\) has final square \(-2\). In the first family, the intervening component retains square \(-2\). This gives the stated intersection numbers. In the third family, the curve contracted at the last step has intersection number two with the surviving component. Its contraction increases the square of that component from \(-3\) to \(1\). When \(p=0\), the last step is the contraction of \(F\). The image is an irreducible rational anticanonical curve of arithmetic genus one.

In each case, \(D\) has \(p+9\) components. Hence \(\rho(Y)=p+10\) and \(K_Y^2=-p\). The \(p+1\) contractions give \(K_{Y_0}^2=1\), and \eqref{four:anticanonical-cycle} gives \(-K_{Y_0}\sim D_0\). In the first two cases, \(D_0\) has intersection number one with the component of square \(-1\) and zero with every other component. In the third case, \(D_0^2=1\). Every curve outside the support of \(D_0\) has nonnegative intersection with \(D_0\). Thus \(-K_{Y_0}\) is nef and big.
\end{proof}

\subsection{The elliptic pencil and its singular fibers}

The anticanonical pencil on a degree-one weak del Pezzo surface has a unique simple base point, which is a smooth point of every member. Blowing up that point gives \(\beta:T\to Y_0\) and a relatively minimal rational elliptic surface \(f:T\to\mathbb{P}^1\), with exceptional section \(O\); see Dolgachev~\cite[Theorem~8.3.2 and Section~8.8.1]{Dolgachev}.

\begin{proposition}\label{prop:fibers}\label{lem:section}
The base point does not lie on \(C\), and the strict transform of \(C\) is a section of \(f\). Its attachment points lie over the four singular fibers, whose types in the order \(P_1,P_2,P_3,P_4\) are
\[
\begin{array}{c|c}
\text{orders}&\text{singular fibers}\\ \hline
(3,3,3)&I_3,I_3,I_3,I_3\\
(2,4,4)&I_2,I_4,I_4,I_2\\
(2,3,6)&I_2,I_3,I_6,I_1.
\end{array}
\]
\end{proposition}

\begin{proof}
The divisor \(D'\) is disjoint from the anticanonical member \(D_0\), so it is disjoint from the base point. Its components remain vertical \((-2)\)-curves on \(T\). Suppose that the base point lies on \(C\). The strict transform of \(C\) then has square \(-2\) and fiber degree \((-K_{Y_0})\cdot C-1=0\). It is a fiber component of multiplicity one, since it meets \(O\) once and \(O\) has intersection number one with a fiber. The three end components of \(D'\) meeting \(C\) belong to the same fiber. Intersecting that fiber with \(C\) gives
\[
0=-2+\sum_{\text{other components }A}m_A(A\cdot C)\ge-2+3=1,
\]
a contradiction. Thus the base point does not lie on \(C\). Its fiber degree is one, so \(f|_C\) is an isomorphism.

The three chains in \(D'\) lie in distinct fibers, and these fibers are distinct from the fiber corresponding to \(D_0\). A fiber containing an \(A_{l-1}\)-chain has at least \(l\) components, because the negative definite chain cannot be the whole fiber. Its topological Euler characteristic is therefore at least \(l\). By Proposition~\ref{prop:degree-one}, the fourth fiber has topological Euler characteristic at least three, two, or one, respectively. The lower bounds sum to
\[
3+3+3+3=12,\quad 2+4+4+2=12,\quad 2+3+6+1=12.
\]
The surface \(T\) is rational and \(K_T^2=0\), so Noether's formula gives \(c_2(T)=12\). Equality holds in each lower bound, and there are no other singular fibers. By the Kodaira classification, an additive fiber has topological Euler characteristic one greater than its number of components. Thus every reducible fiber is multiplicative, and the irreducible fiber with topological Euler characteristic one is nodal. This proves the table.
\end{proof}

\subsection{The cross ratios of the singular parameters}

Beauville's classification of nonisotrivial semistable elliptic fibrations over \(\mathbb{P}^1\) with four singular fibers~\cite{Beauville} determines each fibration in Proposition~\ref{prop:fibers} up to isomorphisms of the surface and the base. The configurations \(I_3^4\), \(I_4^2I_2^2\), and \(I_6I_3I_2I_1\) have the following models:
\begin{equation*}
\begin{array}{c|l}
(3,3,3)&X^3+Y^3+Z^3-3sXYZ=0\\
(2,4,4)&y^2=x(x-1)(x-s^2)\\
(2,3,6)&(X+Y)(Y+Z)(Z+X)+sXYZ=0.
\end{array}
\end{equation*}
Let \(\omega^3=1\), with \(\omega\ne1\), and write \(s_i=s(P_i)\). The ordered singular parameters and the normalizing coordinates are
\[
\begin{array}{c|c|c|c}
\text{orders}&(s_1,s_2,s_3,s_4)&z(s)&z(P_4)\\ \hline
(3,3,3)&(1,\omega,\omega^2,\infty)&\dfrac{(\omega-\omega^2)(s-1)}{(\omega-1)(s-\omega^2)}&\dfrac{\omega-\omega^2}{\omega-1}\\[6pt]
(2,4,4)&(1,0,\infty,-1)&\dfrac{s}{s-1}&\dfrac12\\[6pt]
(2,3,6)&(1,0,\infty,-8)&1-s&9.
\end{array}
\]
The four singular fibers of the Hesse pencil are triangles. The value of \(\lambda\) in the table satisfies \(\lambda^2-\lambda+1=0\), and permutations of the arms with the same Seifert invariants identify the two roots. For the second model,
\[
\Delta_{\mathrm{ell}}=16s^4(s^2-1)^2.
\]
The fibers at \(0,\infty\) are of type \(I_4\), and those at \(1,-1\) are of type \(I_2\). 

For the third pencil, \(s=0\) gives three lines, and \(s=1\) gives
\[
(X+Y+Z)(XY+YZ+ZX)=0.
\]
At \(s=-8\), the cubic has a node at \([1:1:1]\), with affine quadratic term \(2(x^2-xy+y^2)\). Resolving the base points gives an \(I_6\)-fiber at infinity.

\subsection{Numerical properties of the listed graphs}

\begin{proposition}\label{completion:candidate-data}
Every weighted homogeneous surface singularity with a graph in Tables~\ref{table:threearm} and~\ref{table:fourarm} is rational and satisfies \(K^2+N=0\), independently of the existence of a smoothing and of the attachment points.
\end{proposition}
\begin{proof}
For \(k\ge0\), put \(D(k)=kd-\sum_i\lceil ka_i/n_i\rceil\). The graded formula for the section ring gives \(p_g=\sum_{k\ge0}h^1(\mathbb{P}^1,\mathcal{O}(D(k)))\). In the three-arm rows with \(d\ge3\), one has \(D(k)\ge0\); in every row with \(d=2\), two arms have one vertex and order at least two, so \(D(k)\ge k-2\lceil k/2\rceil\ge-1\). In the four-arm rows, the reciprocal orders of the three arms with one vertex sum to one, whence their three ceilings sum to at most \(k+2\), and \(D(k)\ge k-2\ge-1\) for \(k\ge1\). Since \(D(0)=0\), rationality follows in all cases.

Set \(S=\sum_{i,j}(b_{ij}-3)\), \(\eta=t-2-\sum_i1/n_i\), and \(C=\sum_i a_i^*/n_i+\eta^2/e\). The proof using the Schur complement of \eqref{obs:canonical-square} applies to any number of arms, with \((K+E)E_0=t-2\), and gives \(K^2+N=5-d-S-C\). Reading the entries of the two tables gives
\[
\begin{array}{c|ccc}
\text{row}&d&S&C\\ \hline
1\mathrm{(a)}&4&0&1\\
1\mathrm{(b)},1\mathrm{(c)}&3&1&1\\
1\mathrm{(d)}\text{--}1\mathrm{(j)}&2&2&1\\
2\mathrm{(a)}\text{--}2\mathrm{(c)}&3&0&2
\end{array}
\]
We give the arithmetic verifying the last column. The continued fractions \(n_i/a_i\) and the modular inverses \(a_i^*\), defined by \(a_ia_i^*\equiv1\pmod{n_i}\), are computed by multiplying \(M(b)=\left(\begin{smallmatrix}b&-1\\1&0\end{smallmatrix}\right)\), using \(M(2)^j=\left(\begin{smallmatrix}j+1&-j\\j&1-j\end{smallmatrix}\right)\).

For row~1(a), put \((a,b,c)=(p+2,q+2,r+2)\), \(\alpha=a(b-1)+1\), \(\beta=b(c-1)+1\), \(\gamma=c(a-1)+1\), \(L=abc+1\), and \(R=abc-ab-ac-bc+a+b+c-2\). For rows~1(b),(c), respectively take \((a,b,c)=(p+2,q+2,r+2)\) and \((1,r+2,q+2)\), and put \(h=bc-b+c\), \(k=abc+b-c\), \(L=a(b+1)c+b\), and \(R=abc-ab-c\). The matrix products give
\[
\begin{array}{c|c|c}
\text{row}&(n_1,n_2,n_3)&(a_1^*,a_2^*,a_3^*)\\ \hline
1\mathrm{(a)}&(\alpha,\beta,\gamma)&(b-1,c-1,a-1)\\
1\mathrm{(b)},1\mathrm{(c)}&(a+1,h,k)&(1,c-1,ab-1).
\end{array}
\]
Writing \(Q=n_1n_2n_3\), direct substitution gives
\[
e=\frac{L^2}{Q},\quad \eta=\frac{RL}{Q},\quad R^2+\sum_{i=1}^3a_i^*\frac{Q}{n_i}=Q,
\]
and hence \(C=1\).

For rows~1(d)--(g), take \((P,Q,k)=(r+4,2,q+3),(3,p+3,q+3),(3,2,q+3),(r+4,p+3,q+3)\), respectively, and set \(\ell=PQ-P-Q\), \(U=P+k\ell\), \(n=(k-1)U-kQ\). The Seifert pairs are \((P,1),(Q,1),(n,n-U+1)\), and \(a_3^*=U-PQ+1\). In fact,
\[
\begin{gathered}
(n-U+1)(U-PQ+1)-1=((k-2)\ell-Q+1)n,\\
e=\frac{U^2}{PQn},\quad
\eta=\frac{U(U-PQ)}{PQn},\quad
n\ell=U(U-PQ)+PQ.
\end{gathered}
\]
Consequently
\[
C=\frac{n(P+Q)+PQ(U-PQ+1)+(U-PQ)^2}{PQn}
=\frac{n(P+Q)+n\ell}{PQn}=1.
\]
For rows~1(h),(i),(j), put \(m=q+2,q+3,q+4\), respectively. If \(\sigma\) is the sum of the reciprocal orders of the first two arms, the last-arm data are
\[
\begin{array}{c|c|c|c|c}
\text{row}&\sigma&n&a^*&e\\ \hline
1\mathrm{(h)}&\frac12&2m^2&2m-1&\frac{(m+1)^2}{2m^2}\\[2pt]
1\mathrm{(i)}&\frac12&m^2-2m+2&m-1&\frac{m^2}{2n}\\[2pt]
1\mathrm{(j)}&\frac23&m^2-3m+3&m-2&\frac{m^2}{3n}.
\end{array}
\]
Substituting \(\eta=1-\sigma-1/n\) gives \(C=\sigma+a^*/n+\eta^2/e=1\).

Finally, for rows~2(a),(b),(c), put \(P=p+2\), and take \(q=3,2,1\), respectively. Set \(m=qP(P-1)+1\). The fourth arm has numerator \(m\), denominator \(m-qP+1\), and modular inverse \(a_4^*=q(P-1)+1\); the product of the last two integers is \(1+(q(P-2)+1)m\). Since \(e=qP^2/m\) and \(\eta=(m-1)/m\), one obtains
\[
C=1+\frac{q(P-1)+1+q(P-1)^2}{m}=2.
\]
Substitution into the displayed table proves \(K^2+N=0\).
\end{proof}

\begin{corollary}\label{completion:strong-smoothability}
Let \((V,0)\) be a weighted homogeneous surface singularity whose graph occurs in Table~\ref{table:threearm} or~\ref{table:fourarm}. Every strong \(\mathbb{Q}\)-Gorenstein smoothing of \((V,0)\) is a rational homology disk smoothing. In the four-arm case, \((V,0)\) admits a strong \(\mathbb{Q}\)-Gorenstein smoothing if and only if its ordered attachment points satisfy the conditions in Theorem~\ref{thm:cross-ratio}.
\end{corollary}
\begin{proof}
Proposition~\ref{completion:candidate-data} gives rationality and \(K^2+N=0\) for every choice of attachment points. Formula~\eqref{foundation:strong-milnor} therefore gives \(\mu=0\) for every strong \(\mathbb{Q}\)-Gorenstein smoothing. Its Milnor fiber has zero first rational homology and the homotopy type of a finite complex of dimension at most two, so it is a rational homology disk. In the four-arm case, Theorem~\ref{thm:cross-ratio} gives the necessary conditions on the attachment points. Their sufficiency follows from the constructions in Theorem~\ref{cqg:main}.
\end{proof}

Thus, within each four-arm graph family, the special cross ratios characterize strong \(\mathbb{Q}\)-Gorenstein smoothability itself. This is a consequence of the numerical properties of the graphs and the analytic classification, rather than an additional classification theorem.

\section{Equivariant deformations of canonical covers}\label{cqg:section}

We use lifting of relations and canonical-cover descent to verify the deformations. Proposition~\ref{uqg:regular-section} treats invariant sections and regular sequences, used in Sections~\ref{cqg:invariant-sections} and~\ref{cqg:toric-sections}. Lemmas~\ref{uqg:fractions} and~\ref{uqg:fraction-syzygies} treat the fraction-field adjunctions used in Section~\ref{cqg:m}. Corollary~\ref{cqg:qhd-verification} supplies the rational homology disk property.

\subsection{Canonical covers and descent}\label{cqg:general-foundation}\label{uqg:presentations}

The construction applies to any normal complex \(\mathbb{Q}\)-Gorenstein surface algebra \(A\). Let \(r\) be its index, fix a trivialization of \(\omega_A^{[r]}\), and write \(B=\bigoplus_{j=0}^{r-1}\omega_A^{[j]}\) for the resulting canonical cover. Write \(G=\mu_r\), and normalize its faithful character \(\kappa\) so that a generator of \(\omega_B\) has character \(\kappa\); the summand \(\omega_A^{[j]}\) then has character \(\kappa^{-j}\). The cover \(B\) is normal and Gorenstein, and the action is free away from the closed point. We use \(t\) for the deformation parameter throughout this section; all finite groups act trivially on \(t\).

We recall the canonical-cover description of \(\mathbb{Q}\)-Gorenstein deformations, including the character and index identifications needed below; see Lee--Nakayama~\cite[Section~7]{LeeNakayama2018} for general base-change results.

\begin{proposition}\label{cqg:descent}
Let \(\mathcal{B}\) be a flat analytic deformation of \(B\) over a smooth pointed curve. Suppose that the canonical action of \(G\) extends to \(\mathcal{B}\), and that its nonzero fibers are smooth and carry a free \(G\)-action. Then \(R=\mathcal{B}^G\) is a strong \(\mathbb{Q}\)-Gorenstein smoothing of \(A\), of index \(r\), and
\[
\mathcal{B}_{\kappa^{-j}}\cong\omega_R^{[j]},\qquad
\omega_R^{[j]}/t\omega_R^{[j]}\cong\omega_A^{[j]}
\qquad(j\in\mathbb{Z}).
\]
The same assertion holds for an algebraic family after restriction to a neighborhood of the central point.
\end{proposition}

\begin{proof}
Exactness of invariants gives flatness and \(R/tR=A\). The local total algebra \(\mathcal{B}\) is Gorenstein because \(t\) is regular and \(B\) is Gorenstein. It is regular in codimension one, by smoothness off the central fiber and regularity at its generic point, so Serre's criterion gives normality of \(\mathcal{B}\) and hence of \(R\). Parameter sequences of \(R\) are regular on its finite extension \(\mathcal{B}\); thus the character summands are maximal Cohen--Macaulay and reflexive. Character projection and Nakayama's lemma lift a canonical generator with character \(\kappa\). Descent on the complement of the central closed point identifies the character summands with the canonical powers, and reflexivity extends these identifications. Character projection commutes with reduction modulo \(t\), giving specialization. The index divides \(r\); a smaller index would specialize to a smaller index of \(A\).
\end{proof}

\begin{corollary}\label{cqg:larger-cover}
Let \(\widetilde B\) be a normal Gorenstein surface algebra with a finite group \(\Gamma\) acting freely off its closed point. Suppose a canonical generator has character \(\chi\), and put \(H=\ker\chi\). Then \(\widetilde B^H\), with group \(\Gamma/H\), is the canonical cover of \(A=\widetilde B^\Gamma\). A flat equivariant deformation of \(\widetilde B\) with smooth nonzero fibers on which the group acts freely induces a deformation of this canonical cover and a strong \(\mathbb{Q}\)-Gorenstein smoothing of \(A\).
\end{corollary}

\begin{proof}
Descent of the canonical generator identifies the character summands with the reflexive canonical powers as above. Since \(H\) fixes this generator, \(\widetilde B^H\) is Gorenstein. The index is the order of \(\chi\): if a smaller canonical power were free, its pullback would be \(u\rho^j\) with \(u\) a unit and \(g(u)=\chi(g)^{-j}u\); evaluation at the fixed closed point forces \(\chi^j=1\). Taking invariants commutes with reduction modulo \(t\) and preserves flatness. The nonzero \(H\)-quotients are smooth, with free residual action, so Proposition~\ref{cqg:descent} applies.
\end{proof}

For the cyclic row of Table~\ref{table:smoothing-covers}, the canonical cover and its equivariant smoothing are
\[
W=\{xy-z^p=0\},\qquad \mathcal{W}=\{xy-z^p=t\},\qquad \Gamma=\mu_p(1,-1,q),
\]
where \(p>q>0\) and \(\gcd(p,q)=1\); see Wahl~\cite{Wahl1981}. Writing \(x=u^p\), \(y=v^p\), and \(z=uv\) identifies the central quotient with \(\frac1{p^2}(1,pq-1)\). The residue canonical form \(dx\wedge dz/x\) has faithful character \(q\). The nonzero fibers are smooth and the action is free, since all three coordinate characters are faithful. Their Euler characteristic is \(p\), so the quotient has Euler characteristic one and is QHD. Proposition~\ref{cqg:descent} gives the strong \(\mathbb{Q}\)-Gorenstein property.

\subsection{Equivariant equations and lifted relations}\label{uqg:equivariant-relations}

Choose a finite-dimensional \(G\)-representation \(W\), put \(e=\dim W\), and fix an equivariant presentation
\[
P_0\otimes Q_2\xrightarrow{\Sigma_0}P_0\otimes Q_1\xrightarrow{f}P_0\to B\to0,
\]
exact at the middle and right terms. Here \(P_0=\mathbb{C}\{x_1,\ldots,x_e\}\), or \(P_0=\operatorname{Sym}(W^\vee)\) followed by localization or analytification. The columns of \(\Sigma_0\) generate the first syzygy module. Throughout, a syzygy matrix means a matrix whose columns generate the entire module of relations; no minimality is assumed. The representations \(Q_1,Q_2\) on the chosen equation and relation generators need not be trivial.

We use the one-parameter lifting-of-relations criterion~\cite[Proposition~I.1.91]{GreuelLossenShustin2007}, with equivariant generators.

\begin{samepage}
\begin{lemma}[Lifting relations]\label{uqg:flat-relations}
Suppose equivariant polynomial or convergent matrices \(F,\Sigma\) satisfy
\[
F\bmod t=f,\qquad \Sigma\bmod t=\Sigma_0,\qquad F\Sigma=0.
\]
Then the quotient by \((F)\) is flat over the local parameter ring and its central fiber has coordinate ring \(B\). Conversely, a flat equivariant deformation of this presentation admits such lifts of all central syzygies.
\end{lemma}
\end{samepage}

Equivariant lifts are obtained by averaging.

\begin{proposition}\label{uqg:criterion}
Put
\[
\mathfrak{J}=(F)+I_{e-2}(J_xF),\qquad
\mathfrak{F}_g=(F)+(g(x_i)-x_i:1\le i\le e).
\]
If the matrices of Lemma~\ref{uqg:flat-relations} are convergent and
\[
t^N\in\mathfrak{J},\qquad t^{N_g}\in\mathfrak{F}_g\quad(g\ne1)
\]
for positive integers \(N,N_g\), their invariant quotient is a strong \(\mathbb{Q}\)-Gorenstein smoothing of \(A\). Conversely, every such smoothing supplies these data after choosing an equivariant presentation. Here \(I_0\) means the unit ideal.
\end{proposition}

\begin{proof}
Flatness and the Cohen--Macaulay central surface give pure two-dimensional fibers near the central point. The first ideal defines the relative nonsmooth locus, by the Jacobian criterion, and the other ideals define the fixed loci. The indicated powers exclude these loci from the nonzero fibers, so Proposition~\ref{cqg:descent} applies. Conversely, the equivariant canonical cover of a strong \(\mathbb{Q}\)-Gorenstein smoothing is a flat deformation of \(B\), with smooth nonzero fibers on which the covering group acts freely. Lift generators, equations, and central syzygies, using Lemma~\ref{uqg:flat-relations}. The nonsmooth and fixed loci are supported on \(t=0\), so the analytic nullstellensatz gives the powers. Only convergent or polynomial lifts are asserted.
\end{proof}

\begin{corollary}\label{cqg:qhd-verification}
Let \(A\) be the local algebra of a weighted homogeneous surface singularity whose resolution graph occurs in Table~\ref{table:threearm} or~\ref{table:fourarm}. The data of Proposition~\ref{uqg:criterion}, or a deformation of a larger cover as in Corollary~\ref{cqg:larger-cover}, give a strong \(\mathbb{Q}\)-Gorenstein rational homology disk smoothing as soon as the central quotient is identified with \(A\).
\end{corollary}

\begin{proof}
The stated criteria give a strong \(\mathbb{Q}\)-Gorenstein smoothing of \(A\). The first assertion of Corollary~\ref{completion:strong-smoothability}, which follows from the graph identity \(K^2+N=0\), gives the rational homology disk property.
\end{proof}

\begin{lemma}\label{uqg:graded-flatness}
Suppose the central polynomial sequence in Lemma~\ref{uqg:flat-relations} is exact over \(P_0\). If its lifted polynomial quotient has a nonnegative grading with \(\deg t>0\), it is flat over \(\mathbb{C}[t]\).
\end{lemma}

\begin{proof}
Multiplication by \(t\) is injective by Lemma~\ref{uqg:flat-relations}. A polynomial in \(t\) with nonzero constant term acts injectively by comparison of the lowest nonzero homogeneous terms. Factor a power of \(t\) from any nonzero polynomial. The quotient is torsion-free, hence flat over \(\mathbb{C}[t]\).
\end{proof}

\subsection{Regular sequences and adjoining elements of the fraction field}\label{uqg:operations}

\begin{proposition}[Deforming a regular sequence]\label{uqg:regular-section}
Let \(P_0\) be polynomial, \(a=(a_1,\ldots,a_s)\), and let \(U\) generate \(\ker a\). If \(h_1,\ldots,h_c\) form a regular sequence modulo \((a)\), a syzygy matrix of \((a,h)\) consists of
\[
\begin{pmatrix}U\\0\end{pmatrix},\qquad
h_j e_\alpha-a_\alpha e_{s+j},\qquad
h_j e_{s+i}-h_i e_{s+j}
\quad(\alpha\le s,\ j\le c,\ i<j).
\]
Replacing \(h\) by a polynomial lift \(H\equiv h\pmod t\), respecting the representation spanned by \(h_1,\ldots,h_c\) in the equivariant case, gives a polynomial syzygy lift for \(F=(a,H)\). After localization at \(t=0\), the resulting quotient is flat over \(\mathbb{C}[t]_{(t)}\), with central fiber \(P_0/(a,h)\). If the quotient admits a nonnegative grading with \(\deg t>0\), it is flat over \(\mathbb{C}[t]\). For one invariant function, the lifted equations and syzygies are
\begin{equation}\label{uqg:cartier-matrix}
F=(a,h-t),\qquad
\Sigma=\begin{pmatrix}U&(h-t)I_s\\0&-a\end{pmatrix}.
\end{equation}
\end{proposition}

\begin{proof}
Reduce a relation \(av+hb=0\) modulo \((a)\). Exactness of the first Koszul homology expresses \(\overline b\) by the last columns. After subtracting their lifts, every component of \(b\) lies in \((a)\); the mixed columns remove these components. The remaining relation is generated by \(U\). Each column remains a relation upon replacing \(h\) by \(H\). Equivariance follows by using the tensor product of the representations on the two sets of equations for the mixed columns and the exterior square of the representation spanned by \(h_1,\ldots,h_c\) for the Koszul columns. Flatness and the assertion about the central fiber follow from Lemma~\ref{uqg:flat-relations}; the graded assertion follows from Lemma~\ref{uqg:graded-flatness}.
\end{proof}

Adjoining an element of the fraction field by equations of the form \(Qh_i-u_i\) is closely related to unprojection; see Papadakis--Reid~\cite[Definition~1.2 and Theorem~1.5]{PapadakisReid2004}. Resolutions in this setting are studied in Papadakis~\cite{Papadakis2004}. We prove the statements needed below under their stated hypotheses, including compatibility with specialization and generators for the first syzygy module.

\begin{lemma}\label{uqg:fractions}
Let \(S\) be a domain, \(0\ne I\subset S\), and \(q\in\operatorname{Frac}S\) with \(qI\subset S\). For
\[
T=S[Q]/(Qu-qu:u\in I)
\]
the \(Q\)-degree filtration has \(\operatorname{gr}T=S[Q]/IQ\). If \(S/I\) is a domain and \(qu_0\notin I\) for some \(u_0\in I\), then \(T=S[q]\). If \(S\) and \(S/I\) are flat over a ring \(R\), then so is \(T\). For \(R=\mathbb{C}[t]\), flatness of \(S/I\) identifies the specialized ideal with \(I/tI\), so the presentation specializes exactly.
\end{lemma}

\begin{proof}
Inside \(\operatorname{Frac}(S)[Q]\), the relation ideal is \((Q-q)I[Q]\). The leading coefficient of \((Q-q)b(Q)\) is that of \(b\); hence its initial ideal is \(IQ\). Thus \(S\) embeds in \(T\), and the filtration quotients are \(S\) in degree zero and \(S/I\) in positive degrees. Put \(v_0=qu_0\). A representative of least degree \(n>0\) has leading coefficient nonzero modulo \(I\). Multiplication by \(u_0\) and the relation \(u_0Q=v_0\) give a coefficient nonzero modulo \(I\) in degree \(n-1\); lower terms contribute elements of \(I\) there. For degree zero use that \(S\) is a domain. Thus \(u_0\) is regular on \(T\), and
\[
T\hookrightarrow T[u_0^{-1}]=S[u_0^{-1}],\qquad Q=v_0/u_0=q.
\]
This proves \(T=S[q]\). Flatness follows from the filtration and passage to its union. The last assertion follows from the exact sequence for \(I\subset S\) and flatness of \(S/I\).
\end{proof}

\begin{lemma}\label{uqg:fraction-syzygies}
Suppose \(S=P_0/(a)\) in Lemma~\ref{uqg:fractions}, with \(P_0\) polynomial. Suppose \(I\) is generated by a row \(h=(h_1,\ldots,h_v)\) of distinct coordinate variables and \((a)\subset(h)\). Choose \(C\) with \(a=hC\), representatives \(u_i\) of \(qh_i\), and a syzygy matrix \(U\) of \(a\). Put \(g=Qh-u\), and solve
\[
u_i h_j-u_j h_i=a d^{ij},\qquad -uC_\alpha=a d^\alpha.
\]
A syzygy matrix of \((a,g)\) consists of
\[
\begin{pmatrix}U\\0\end{pmatrix},\qquad
\begin{pmatrix}d^{ij}\\h_j e_i-h_i e_j\end{pmatrix},\qquad
\begin{pmatrix}Qe_\alpha+d^\alpha\\-C_\alpha\end{pmatrix},
\]
and the mixed columns \((-g_i e_\alpha,a_\alpha e_i)\). For parameter-dependent data the same construction specializes to a matrix generating all syzygies after setting \(t=0\) whenever the identification of the specialized algebra with the subalgebra generated by the specialized fraction holds.
\end{lemma}

\begin{proof}
The coefficient identities hold modulo \((a)\), since multiplication by \(q\) is \(S\)-linear on \(I\); polynomial division supplies the displayed solutions. The columns are relations by direct multiplication. Over \(S[Q]\), multiplication by \(Q-q\) is injective, so \(\operatorname{Syz}(g)=\operatorname{Syz}_S(h)[Q]\). The latter is generated by the coordinate Koszul columns and the columns of \(C\): if \(hb=ac=hCc\), then \(b-Cc\) is a coordinate Koszul relation in \(P_0\). Reduce a relation of \((a,g)\) modulo \((a)\) and remove these generators. The mixed columns remove the remaining last entries, which lie in \((a)\); the first entries then belong to \(\ker a\). This proves completeness, and the same argument applies after specialization.
\end{proof}

\subsection{Presentations of invariant rings and identification of the central fiber}\label{uqg:invariant-presentations}

\begin{proposition}\label{uqg:invariant-presentation}
Let a finite group \(\Gamma\) act linearly on a polynomial algebra \(P\), let \(H\triangleleft\Gamma\), and put \(G=\Gamma/H\). Suppose
\[
P[t]\otimes Q_2\xrightarrow{d_2(t)}P[t]\otimes Q_1\xrightarrow{d_1(t)}P[t]
\]
is an equivariant polynomial complex whose specialization at \(t=0\) is exact at the middle term. There is a finite construction of an equivariant polynomial presentation \(S[t]/(F(t))\) of \((P[t]/\operatorname{im}d_1(t))^H\), together with \(\Sigma(t)\), such that \(F\Sigma=0\) and \(\Sigma(0)\) generates \(\ker F(0)\). In particular, this quotient is flat near \(t=0\), with the prescribed invariant algebra as its central fiber.
\end{proposition}

\begin{proof}
Put \(C=P^H\) and \(M_i=(P\otimes Q_i)^H\). These are finite \(C\)-modules. Exactness of invariants gives a complex \(M_2[t]\to M_1[t]\to C[t]\), exact at \(M_1\) after setting \(t=0\), with the stated invariant quotient as cokernel. Choose equivariant presentations \(S/(k)\cong C\) and \(p:S^s\to M_1\), and matrices \(U_k\) and \(L\) whose columns generate \(\ker k\) and \(\ker p\), respectively. Choose module generators \(n_j\) of \(M_2\). Lift \(d_1(t)p\) to a polynomial row \(H(t)\), and lift the columns \(d_2(t)n_j\) through \(p\) to a matrix \(D(t)\). Put \(A(t)=[L,D(t)]\). Since \(HA\) vanishes in \(C[t]\), solve \(kB=-HA\) in the ideal \((k)\), and set
\[
F=(k,H),\qquad
\Sigma=\begin{pmatrix}U_k&B\\0&A\end{pmatrix}.
\]
Then \(F\Sigma=0\). For a central relation \(kv+H(0)a=0\), exactness at \(t=0\) shows that \(p(a)\) is an image from \(M_2\). Express a preimage in the generators \(n_j\). It follows that \(a\) is a combination of the columns of \(D(0)\) and \(L\), hence of \(A(0)\). Subtract the corresponding columns of \(\Sigma(0)\); the remaining relation lies in \(\ker k\) and is generated by \(U_k\). Thus the columns of \(\Sigma(0)\) generate the entire syzygy module. The quotient by \((k,H)\) is the required cokernel, and Lemma~\ref{uqg:flat-relations} gives flatness without assuming it in the construction. All choices can be made equivariantly by averaging the finitely many coefficient lifts over \(G\).
\end{proof}

\begin{remark}\label{uqg:fixed-presentations}
We use the following finite computation throughout. For diagonal \(H=\mu_d\) with coordinate weights \(a_i\), the invariant algebra is generated by \(x_i^d\) and the nonconstant monomials \(x^u\) with \(0\le u_i<d\) and \(\sum a_i u_i\equiv0\pmod d\). The module of covariants of character \(b\) is generated by the monomials with \(0\le u_i<d\) and \(\sum a_i u_i+b\equiv0\pmod d\). Division of the exponents by \(d\) proves both assertions; for \(d=1\) use the coordinates and representation basis. For a toric algebra with specified monomial generators, its relation ideal is obtained by eliminating the ambient monomials from the equations identifying them with the new variables. Module elimination gives the corresponding covariant presentations. Fix the finite group-stable equation list before computing its full kernel by the Schreyer algorithm, retaining the change-of-generator matrices. Polynomial division then supplies every coefficient identity used above. These calculations terminate for each input and do not require a complete intersection or a minimal presentation.
\end{remark}

\begin{lemma}\label{uqg:central-components}
Let a torus \(T\) act on a reduced affine algebra \(D\). Suppose its finitely many irreducible components are \(T\)-stable and contain a common fixed point \(o\). Let \(C\) be one component, and suppose that every other component has only constant invariant functions. Restriction induces an injection \(D^T\hookrightarrow\mathcal{O}(C)^T\). If specified elements of \(D^T\) restrict to generators of \(\mathcal{O}(C)^T\), this map is an isomorphism.
\end{lemma}

\begin{proof}
An invariant function vanishing on \(C\) has value zero at \(o\). On each other component it is constant and therefore zero. Reducedness implies that it is zero in \(D\). The generator condition gives surjectivity.
\end{proof}

\section{Smoothings defined by an invariant function}\label{cqg:invariant-sections}

We use the equations, group actions, and central quotient descriptions of Stipsicz--Szab\'{o}--Wahl, Wahl, and Artal Bartolo--Wahl~\cite{SSW,Wahl2011,ArtalBartoloWahl2022} for Figures~1(a)--(c), 1(h), and~2. We determine the canonical characters and invariant presentations and apply Proposition~\ref{uqg:regular-section}.

For the noncyclic families, we retain the notation of Subsections~\ref{sec:notation} and~\ref{obs:characters}. In particular, the Seifert invariants together with the attachment points determine the central graded algebra
\begin{equation}\label{cqg:section-algebra}
A=\bigoplus_{k\ge0}H^0\left(\mathbb{P}^1,\mathcal{O}_{\mathbb{P}^1}(\lfloor k\mathcal{D}\rfloor)\right).
\end{equation}
For three arms, the cited quotient graphs and Corollary~\ref{cor:three-arm-taut} identify the prescribed analytic structures. In the four-arm case, we first construct the smoothing of the actual quotient and then use the necessity assertion of Theorem~\ref{thm:cross-ratio} to identify its attachment points.

\subsection{The common lifting argument}\label{uqg:wnh:matrices}

Let \(C=Q/(a)\) be an equivariant presentation of a three-dimensional domain and \(h\in Q\) an invariant lift, chosen by averaging, of a nonzero function \(f\in C\). Proposition~\ref{uqg:regular-section} gives \(B=Q/(a,h)\), the deformation \(Q[t]/(a,h-t)\), and the full syzygy lift~\eqref{uqg:cartier-matrix}. All gradings below have \(\deg t>0\), so flatness holds over \(\mathbb{C}[t]\). For a larger cover \(D\) with group \(\Gamma\), take \(C=D^{\ker\kappa}\), using Proposition~\ref{uqg:invariant-presentation}. Exactness of invariants identifies the deformation with \((D[t]/(f-t))^{\ker\kappa}\). Smoothness and freeness then give descent by Corollary~\ref{cqg:larger-cover}.

\subsection{The three-arm families}\label{cqg:wn:section}\label{cqg:n:section}\label{cqg:h:section}\label{uqg:wnh:invariants}

\begin{proposition}\label{cqg:w}\label{cqg:n}\label{cqg:h}
Every singularity in Figures~1(a)--(c) and~1(h) admits a strong \(\mathbb{Q}\)-Gorenstein smoothing.
\end{proposition}
\begin{proof}
We give the equations and canonical characters in three cases, followed by their common invariant presentation. In Table~\ref{table:smoothing-covers}, the surface cover is \(W=\Spec(D/(f))\) with the group \(\Gamma\) specified in each case.

\emph{Figure~1(a).} Put \(a=p+2\), \(b=q+2\), and \(c=r+2\), so that \(a,b,c\ge2\), and set
\[
L=abc+1,
\qquad
\alpha=a(b-1)+1,
\qquad
\beta=b(c-1)+1,
\qquad
\gamma=c(a-1)+1.
\]
Let \(D=\mathbb{C}[x,y,z]\), and let \(\Gamma=\mu_L\) act with characters \((1,bc,-c)\). The monomials \(xy^a\), \(yz^b\), and \(zx^c\) have characters \(L\), \(0\), and \(0\), respectively. Hence
\begin{equation*}
f=xy^a+yz^b+zx^c,
\qquad
\mathscr{A}=(D[t]/(f-t))^\Gamma
\end{equation*}
defines the proposed family. The weights \((\alpha,\beta,\gamma)\) give all three monomials degree \(L\), since \(\alpha+a\beta=\beta+b\gamma=\gamma+c\alpha=L\). The exponent matrix of these monomials has determinant \(L\).

The integers \(1\), \(bc\), and \(-c\) are units modulo \(L\). Thus \(\Gamma\) acts freely on \(\mathbb{A}^3\setminus\{0\}\). The partial derivatives of \(f\) are
\[
y^a+czx^{c-1},
\qquad
axy^{a-1}+z^b,
\qquad
byz^{b-1}+x^c.
\]
If one coordinate vanishes at a common zero of these derivatives, all coordinates vanish. If all coordinates are nonzero, multiplying the three derivative equations and cancelling \(x^cy^az^b\) gives \(1=-abc\), a contradiction. Consequently \(Y=(f=0)\) is an isolated normal surface singularity. For \(t\ne0\), Euler's identity gives \(df(\alpha x\partial_x+\beta y\partial_y+\gamma z\partial_z)=Lt\), so the fiber \(f=t\) is smooth.

By Stipsicz--Szab\'{o}--Wahl~\cite[Example~8.9]{SSW}, the central quotient has the Seifert data in Figure~1(a).

The canonical character of \(D\), and also of \(D/(f)\), is
\[
\kappa\equiv1+bc-c\pmod L.
\]
The canonical group therefore has order \(L/\gcd(L,1+bc-c)\).

\emph{Figures~1(b) and~1(c).} Take \(a\ge1\) and \(b,c\ge2\), and put
\[
L=a(b+1)c+b,
\quad
h=bc-b+c,
\quad
k=abc+b-c,
\quad
d_x=c(a+1),
\quad
d_z=b(a+1).
\]
Let
\begin{equation*}
\begin{aligned}
D&=\mathbb{C}[x,y,z,w]/(F),&F&=x^b+z^c-yw,\\
f&=xw-y^az,&\mathscr{A}&=(D[t]/(f-t))^\Gamma,
\end{aligned}
\end{equation*}
where \(\Gamma=\mu_L(1,b+1,-a(b+1),-1)\). Each monomial in \(F\) has character \(b\): for \(z^c\), this follows from \(-ac(b+1)=b-L\). Both monomials in \(f\) have character zero. The weights \((d_x,h,d_z,k)\) give \(F\) degree \(bc(a+1)\) and \(f\) degree \(L\).

The gradient of \(F\) is \((bx^{b-1},-w,cz^{c-1},-y)\), so \(\operatorname{Spec}D\) is smooth outside its vertex. If a nonidentity element of \(\Gamma\) fixes a point of this threefold, the faithful characters of \(x,w\) force \(x=w=0\). The equation gives \(z=0\), while \(\gcd(L,b+1)=1\), since \(L\equiv-1\pmod{b+1}\), forces \(y=0\). Hence \(\Gamma\) acts freely on the punctured threefold.

The central complete intersection is isolated. On \(x\ne0\), eliminate \(w\) to obtain
\[
H=x^{b+1}+xz^c-y^{a+1}z=0.
\]
Here \(z\ne0\). If \(\partial H/\partial y=-(a+1)y^az\) vanishes, then \(y=0\), and \(\partial H/\partial z=cxz^{c-1}\ne0\). On \(x=0\), the equations imply \(y^az=0\) and \(z^c=yw\), so every nonzero point lies on the \(y\)- or \(w\)-axis. At the former, \(dF=-ydw\) and \(df=-y^adz\); at the latter, \(dF=-wdy\) and \(df=wdx\). These covectors are independent. The central complete intersection is therefore normal by Serre's criterion. The weighted Euler field is tangent to \(F=0\) and has \(df\)-value \(Lt\) on \(f=t\), proving smoothness for \(t\ne0\).

With parameters \((a-1,b-2,c-2)\) in Stipsicz--Szab\'{o}--Wahl~\cite[Example~8.10]{SSW}, the quotient \((D/(f))^\Gamma\) has the graph in Figure~1(b) for \(a\ge2\), and Figure~1(c) for \(a=1\). In Table~\ref{table:threearm}, these substitutions are \((a,b,c)=(p+2,q+2,r+2)\) and \((a,b,c)=(1,q+2,r+2)\), respectively.

Adjunction on the hypersurface threefold gives canonical character
\[
\kappa=1+(b+1)-a(b+1)-1-b=1-a(b+1)\pmod L.
\]
The canonical group has order \(L/\gcd(L,1-a(b+1))\).

\emph{Figure~1(h).} Put \(m=q+2\ge2\) and \(L=2m(m+1)\). Choose a primitive \(L\)-th root \(\omega\), and put \(\zeta=\omega^m\), of order \(2m+2\). Let \(\Gamma=\Gamma_h\) be generated by
\[
S(x,y,z,w)=(\omega^{-1}x,\omega^{2m+1}y,\omega^{-(2m+1)}z,\omega w),
\]
and
\[
T(x,y,z,w)=(\zeta^{-1}y,x,w,\zeta z).
\]
These substitutions satisfy
\[
S^L=1,
\qquad
TST^{-1}=S^{-(2m+1)},
\qquad
T^2=S^m.
\]
The identity \((2m+1)^2=1+2L\) verifies the conjugation relation. The powers of \(S\) are diagonal, whereas \(T\) interchanges the two coordinate pairs, so \(|\Gamma|=2L=4m(m+1)\). Since \(\gcd(2m+1,L)=1\), every nonidentity power of \(S\) has no nonzero fixed vector. Further,
\[
(S^iT)^2=S^{-m(2i-1)}\ne1,
\]
because \(2(m+1)\) does not divide \(2i-1\). An element \(S^iT\) fixing a nonzero vector would have a square fixing that vector, a contradiction. Thus \(\Gamma\) acts freely on \(\mathbb{A}^4\setminus\{0\}\).

Set
\begin{equation*}
\begin{aligned}
F&=zw+x^{2m}+\zeta y^{2m},&D&=\mathbb{C}[x,y,z,w]/(F),\\
f&=xw+yz,&\mathscr{A}&=(D[t]/(f-t))^\Gamma.
\end{aligned}
\end{equation*}
The polynomial \(f\) spans the invariant line in \(\operatorname{Span}\{xw,yz\}\). The polynomial \(F\) has \(S\)-character \(\omega^{-2m}\) and \(T\)-character \(\zeta\): the identities \(\zeta^{-2m}=\zeta^2\) and \(\zeta^{2m+2}=1\) give \(T(F)=\zeta F\). Thus the defining equation has a nontrivial character, whereas the equation being deformed is invariant. The weights \((1,1,m,m)\) give \(F\) degree \(2m\) and \(f\) degree \(m+1\).

The threefold \(F=0\) has an isolated singularity, as is immediate from its gradient \((2mx^{2m-1},2m\zeta y^{2m-1},w,z)\). We also verify that its central surface \(Y=(F=f=0)\) is isolated. On \(x\ne0\), eliminate \(w\), obtaining
\[
yz^2-x^{2m+1}-\zeta xy^{2m}=0.
\]
Here \(y\ne0\). Vanishing of the \(z\)-derivative forces \(z=0\), and then the \(y\)-derivative is \(-2m\zeta xy^{2m-1}\ne0\). If \(x=0\), the equations force \(y=0\) and \(zw=0\). At a nonzero point of the \(z\)-axis, \(dF=zdw\) and \(df=zdy\); at a nonzero point of the \(w\)-axis, \(dF=wdz\) and \(df=wdx\). Hence the two gradients are independent. Serre's criterion gives normality of \(Y\). The weighted Euler field proves that every fiber \(f=t\), \(t\ne0\), is smooth.

By Artal Bartolo--Wahl~\cite[Proposition~1.8]{ArtalBartoloWahl2022}, the quotient \((D/(f))^\Gamma\) has the graph in Figure~1(h), with \(m=q+2\).

Both matrices \(S,T\) have determinant one. Adjunction therefore gives
\[
\kappa(S)=\omega^{2m}=\zeta^2,
\qquad
\kappa(T)=\zeta^{-1}.
\]
The image of \(\kappa\) has order \(2m+2\). Its kernel is \(\langle S^{m+1}\rangle\), of order \(2m\): a power \(S^i\) lies in the kernel exactly when \(m+1\) divides \(i\), whereas \(\kappa(S^iT)=\zeta^{2i-1}\ne1\). The canonical group is \(\mu_{2m+2}\).

\emph{The invariant presentations.} For these four families, let \(P\) be the polynomial coordinate algebra, let \(H\) be the ambient equation when present, and put \(D=P/(H)\); in Figure~1(a), put \(D=P\). The kernel of the canonical character \(H_\Gamma=\ker\kappa\) is diagonal cyclic. Its order \(\ell\), coordinate characters \(\mathbf{c}\), and ambient equation character \(\eta\) are
\[
\begin{array}{c|c|c|c}
\text{Figure}&\ell&\mathbf{c}&\eta\\ \hline
1\mathrm{(a)}&\gcd(L,1+bc-c)&(1,bc,-c)&\text{no equation}\\
1\mathrm{(b),(c)}&\gcd(L,1-a(b+1))&(1,b+1,-a(b+1),-1)&b\\
1\mathrm{(h)}&2m&(-1,1,-1,1)&0
\end{array}
\]
All characters are taken modulo \(\ell\). In Figure~1(h), the kernel generator is \(S^{m+1}\). The character \(\eta\) need not vanish in Figures~1(b),(c), and is retained in the covariant module defining the invariant equations.

Proposition~\ref{uqg:invariant-presentation} and Remark~\ref{uqg:fixed-presentations}, applied to these characters, give \(C=D^{H_\Gamma}=Q/(a)\) with its full first syzygy matrix. The bounded-exponent monomial generators are \(\Gamma\)-stable: the cyclic actions are diagonal, and \(T\) in Figure~1(h) permutes the coordinate pairs. The residual group therefore acts on the presentation. The cited graphs identify the three-arm central quotients with \(A\) by \eqref{cqg:section-algebra}. An invariant lift of \(f\) now gives the canonical-cover deformation by Subsection~\ref{uqg:wnh:matrices}; the verified smoothness and freeness give the asserted smoothing.
\end{proof}

\subsection{The four-arm families}\label{cqg:four:section}

The four-arm canonical covers are invariant sections in three-dimensional semigroup algebras. The equations agree with Stipsicz--Szab\'{o}--Wahl~\cite[Examples~8.12 and~8.14]{SSW} and Wahl~\cite[Theorems~2.3, 3.4, and~7.4]{Wahl2011}.

\begin{proposition}\label{cqg:four:main}
Let \(P\ge2\), and let \((n,q)\) be one of \((3,3),(4,2),(6,1)\). Put
\[
m=qP(P-1)+1,\qquad \beta=m-qP+1.
\]
The three choices of \((n,q)\), in their displayed order, correspond to the three rows of Table~\ref{table:fourarm}, with \(P=p+2\). For the ordering of the attachment points in Theorem~\ref{thm:cross-ratio}, the corresponding singularity admits a strong \(\mathbb{Q}\)-Gorenstein smoothing constructed by Proposition~\ref{uqg:criterion}. Its canonical index is \(nP\).
\end{proposition}

\begin{proof}
\emph{The ambient algebras and their representations.}\par\nobreak

All generators of the following algebras have degree one:
\begin{align*}
D_3&=\mathbb{C}[x_1,x_2,x_3],\\
D_4&=\mathbb{C}[x_1,x_2,x_3,x_4]/(x_1x_3-x_2x_4),\\
D_6&=\mathbb{C}[x_1,\ldots,x_6,z]/I_6,
\end{align*}
where indices on \(x_i\) are cyclic and
\[
I_6=\left(x_ix_{i+2}-x_{i+1}z\ (i\bmod6),\quad x_ix_{i+3}-z^2\ (1\le i\le3)\right).
\]
Write \(Z_n=\operatorname{Spec}D_n\). The projective surfaces underlying these cones are \(\mathbb{P}^2\) with its hyperplane class, the smooth quadric with polarization \(\mathcal{O}(1,1)\), and the degree-six del Pezzo surface with its anticanonical polarization. If \(H\) denotes the respective polarization, then
\[
H^2=\frac nq,\qquad K=-qH.
\]
For the third surface, the monomial map
\[
[A:B:C]\dashrightarrow[AB^2:A^2B:A^2C:AC^2:BC^2:B^2C:ABC]
\]
is resolved by blowing up the three coordinate points. The exponent vectors of the six boundary monomials are the vertices of the polygon defining the anticanonical embedding, and the central monomial corresponds to its unique interior lattice point. The displayed binomials reduce every monomial to \(z^a x_i^b x_{i+1}^c\). The triangles joining the central point to adjacent vertices are unimodular; consequently these reduced monomials have distinct semigroup exponents, except for identical monomials on common faces. This proves that the binomials generate the full kernel and that the resulting semigroup is saturated. The chart \(x_1\ne0\) is
\begin{equation}\label{cqg:four:hexchart}
[x_1,\ldots,x_6,z]=[1,x,x^2y,x^2y^2,xy^2,y,xy],
\end{equation}
and the other five charts are obtained by rotation. Every nonzero point belongs to one of these charts, since the equations with \(x_i=0\) for all \(i\) imply \(z=0\). In particular the projective surface is smooth, and the punctured cone is the complement of the zero section in \(\mathcal{O}(-H)\).

All three algebras are normal and Gorenstein. The algebra \(D_3\) is regular. For \(D_4\), the Jacobian \((x_3,-x_4,x_1,-x_2)\) vanishes only at the vertex, so the hypersurface is normal by Serre's criterion and Gorenstein by adjunction. For \(D_6\), normality follows from the saturated semigroup just described. The section ring is Cohen--Macaulay: \(H^1(\mathcal{O}(jH))=0\) for all integers \(j\), by rationality for \(j=0\), Kodaira vanishing for \(j>0\), and Serre duality for \(j<0\). Since \(K=-H\), its canonical module is \(D_6(-1)\). Each \(Z_n\) is smooth away from its vertex. A canonical generator has scalar degree \(q\).

Put
\[
d=qP,\qquad a=d-1,\qquad r=1-d\pmod m.
\]
Choose primitive roots \(\xi\) and \(\eta\) of orders \(m\) and \(nd\), and put \(\zeta=\eta^d\). We choose \(\zeta=i\) when \(n=4\). Let \(E_{\mathrm{amb}}=\mathbb{C}^n\) for \(n=3,4\), and let \(E_{\mathrm{amb}}=\mathbb{C}^7\) for \(n=6\). On \(E_{\mathrm{amb}}\), define
\begin{equation*}
S(x_i)=\xi^{r^{i-1}}x_i,\qquad T(x_i)=\eta x_{i+1}.
\end{equation*}
When \(n=6\), also set \(S(z)=z\) and \(T(z)=\eta z\). These formulas describe actions on points. The relevant integer identities are
\[
\begin{array}{c|c}
n&\text{identity}\\
\hline
3&r^2+r+1=3m\\
4&r^2+1=2m\\
6&r^2-r+1=m,
\end{array}
\qquad r=1-qP\text{ on the right-hand side}.
\]
Thus \(r^n=1\pmod m\). In the hexagonal case, \(r^2=r-1\) and \(r^3=-1\) show that the two sides of each defining binomial have the same \(S\)-character. Rotation preserves these binomials. In the quadrilateral case, writing \(Q_4=x_1x_3-x_2x_4\), we have
\[
Q_4\circ S=Q_4,\qquad Q_4\circ T=-\eta^2Q_4.
\]
Consequently \(\Gamma=\Gamma_{n,P}=\langle S,T\rangle\) acts on \(Z_n\), and
\[
TST^{-1}=S^r.
\]
Every element has the form \(S^iT^j\), with \(0\le i<m\) and \(0\le j<nd\). These forms are distinct. A power of \(T\) is diagonal only when its exponent is divisible by \(n\), in which case it is scalar. No nonidentity power of \(S\) is scalar. For \(n=3\), compare its first two entries and use \(\gcd(m,3P)=1\); for \(n=4\), compare the first and third entries and use that \(m\) is odd; for \(n=6\), compare the central coordinate with the first coordinate. Hence
\begin{equation*}
|\Gamma|=nqmP,\qquad \Gamma\cap\mathbb{C}^*=\langle T^n\rangle\cong\mu_d.
\end{equation*}
The congruences suffice even when \(r\) has smaller order modulo \(m\). In particular, for \(n=6,P=2\), one has \(m=3\) and \(r=-1\pmod3\); no argument below assumes that \(r\) has order six.

Consider the monomial representation
\[
M=\operatorname{Span}_{\mathbb{C}}\{u_i=x_i^{d-1}x_{i+1}:1\le i\le n\}\subset (D_n)_d.
\]
The monomials are distinct. Each has \(S\)-character \((d-1)r^{i-1}+r^i=0\pmod m\), whereas \(T(u_i)=\zeta u_{i+1}\). The invariant line \(M^\Gamma\) is therefore generated by
\begin{equation*}
f=\sum_{j=0}^{n-1}\zeta^j x_{j+1}^{d-1}x_{j+2}.
\end{equation*}
Set \(Y_t=Z_n\cap\{f=t\}\) and \(Y=Y_0\). Thus \(W=Y\) and \(\Gamma=\Gamma_{n,P}\) in Table~\ref{table:smoothing-covers}. 

\emph{Freeness and the canonical character.}\par\nobreak

The action of \(\Gamma\) is free on \(Z_n\setminus\{0\}\). When \(n=3\), every exponent \(1,r,r^2\) is a unit modulo \(m\), and \(\gcd(m,3P)=1\). If \(3\mid j\), an eigenvalue one for \(S^iT^j\) would be an equality between roots of relatively prime orders dividing \(m\) and \(3P\). It forces both factors to be one, hence the identity element. If \(3\nmid j\), then
\[
(S^iT^j)^3=\eta^{3j}I.
\]
A fixed vector would imply \(3P\mid j\), contrary to \(3\nmid j\).

When \(n=4\), the four \(S\)-weights are \((1,r,-1,-r)\), all units modulo \(m\), and \(\gcd(m,2P)=1\). For \(4\mid j\), the same diagonal argument applies. If \(j\) is odd, then \((S^iT^j)^4=\eta^{4j}I\), so a fixed vector would force \(2P\mid j\), which is impossible. If \(j\equiv2\pmod4\), then \((S^iT^j)^2=\eta^{2j}I\), so it would force \(4P\mid j\), again impossible. Thus in the first two cases the action is free on the punctured ambient vector space.

When \(n=6\), a fixed point with \(z\ne0\) forces \(\eta^j=1\). Some boundary coordinate is then nonzero, and its unit \(S\)-weight forces \(i=0\). If \(z=0\), the binomials imply that the nonzero coordinates consist of one vertex or two adjacent vertices of the hexagon. A rotation preserving such a support has \(6\mid j\). The remaining scalar factors have relatively prime orders dividing \(P\) and \(m\), so fixing a nonzero coordinate forces both to be one. This proves freeness in the third case as well.

Let \(\kappa\) be the character of a canonical generator of \(Z_n\). For \(n=3\), determinants give \(\kappa(S)=1\) and \(\kappa(T)=\eta^3\). For \(n=4\), adjunction gives
\[
\kappa(S)=1,\qquad \kappa(T)=\frac{-\eta^4}{-\eta^2}=\eta^2.
\]
For \(n=6\), use cone coordinates \((s,x,y)\) in \eqref{cqg:four:hexchart}, with \(s=x_1\). Then
\begin{equation}\label{cqg:four:hexaction}
S(s,x,y)=(\xi s,\xi^{-P}x,\xi^{P-1}y),\qquad T(s,x,y)=(\eta xs,xy,1/x).
\end{equation}
The toric canonical form is
\[
d(sxy)\wedge\frac{dx}{x}\wedge\frac{dy}{y}=ds\wedge dx\wedge dy.
\]
The two Jacobian determinants in \eqref{cqg:four:hexaction} are \(1\) and \(\eta\). Thus in all three cases
\begin{equation*}
\kappa(S)=1,\qquad \kappa(T)=\eta^q,\qquad \operatorname{ord}(\kappa)=nP.
\end{equation*}
It follows that
\[
K=\ker\kappa=\langle S,T^{nP}\rangle,\qquad |K|=qm,\qquad \Gamma/K\cong\mu_{nP}.
\]
The subgroup \(K\) is cyclic: \(T^{nP}\) is scalar of order \(q\), it commutes with \(S\), and \(\gcd(q,m)=1\).

For these ambient algebras, group actions, and functions, the quotient calculations in Wahl~\cite[Theorems~2.3, 3.4, and~7.4]{Wahl2011} show that the projective curve \(D=(f=0)\subset\operatorname{Proj}D_n\) is smooth and that \(Y/\Gamma\) has the corresponding graph in Table~\ref{table:fourarm}. We use these identifications to construct its canonical-cover deformation below.

\emph{The central fiber and its canonical cover.}\par\nobreak

The smooth ample curve \(D\) is connected and hence irreducible. The central cone is a Gorenstein hypersurface in \(Z_n\), with smooth irreducible punctured cone; Cohen--Macaulayness excludes embedded components, so it is a domain and normal by Serre's criterion. If \(\vartheta\) is the Euler vector field, then
\[
df(\vartheta)=d\cdot f.
\]
Thus every fiber \(Y_t\), \(t\ne0\), is smooth. By the freeness proved above, its quotients by \(K\) and \(\Gamma\) are smooth as well.

In our notation, the graph calculation gives
\[
e=\frac{qP^2}{m},\qquad \frac m\beta=[2^{P-2},q+1,P].
\]
The reciprocals of the three orders of the arms with one vertex sum to one, and \(e+1+\beta/m=3\), confirming the central weight. Put \(A'=(D_n/(f))^\Gamma\); its attachment points will be identified after constructing the smoothing.

Since \(f\) is invariant, adjunction gives the same canonical character on \(Y\) as on \(Z_n\). By Corollary~\ref{cqg:larger-cover}, the canonical cover of \(A'\) and its faithful group are
\[
B=(D_n/(f))^K=R/(f),\qquad R=D_n^K,\qquad G=\Gamma/K\cong\mu_{nP}.
\]
We construct its equivariant presentation directly from the normal domain \(R\).

\emph{A common semigroup presentation and its lifted syzygies.}\par\nobreak

Let \(\Delta_n\) be the polygon whose cyclically ordered vertices are
\[
\begin{array}{c|l}
n&\text{vertices}\\
\hline
3&(0,0),(1,0),(0,1)\\
4&(0,0),(1,0),(1,1),(0,1)\\
6&(0,0),(1,0),(2,1),(2,2),(1,2),(0,1).
\end{array}
\]
The degree-one monomials of \(D_n\) have exponents \((1,b,c)\) at these vertices, together with \((1,1,1)\) when \(n=6\). Put
\[
\begin{split}
\mathcal{S}_n&=\{(\ell,b,c)\in\mathbb{Z}^3:\ell\ge0,\ (b,c)\in\ell\Delta_n\},\\
\Lambda&=\{(\ell,b,c)\in\mathbb{Z}^3:
\ell\equiv0\pmod q,\ \ell-qPb+q(P-1)c\equiv0\pmod m\}.
\end{split}
\]
Then \(D_n=\mathbb{C}[\mathcal{S}_n]\) and
\begin{equation}\label{uqg:four:semigroup}
R=\mathbb{C}[\mathcal{S}_n\cap\Lambda].
\end{equation}
Indeed, a monomial with exponent \((\ell,b,c)\) has scalar degree \(\ell\) and \(S\)-weight \(\ell-qPb+q(P-1)c\). The first congruence in \(\Lambda\) is invariance under \(T^{nP}\), and the second is invariance under \(S\). The unimodular triangulations described above show that the displayed lattice-point semigroups have precisely the given degree-one generators.

There is a common finite generating set for \eqref{uqg:four:semigroup}. Put
\[
L=qm,\qquad s=d-q+1,\qquad L=1+(d-1)s.
\]
Let \(w\) denote the monomial of exponent \((q,1,1)\). Thus \(w=x_1x_2x_3\) for \(n=3\), \(w=x_1x_3=x_2x_4\) for \(n=4\), and \(w=z\) for \(n=6\). Besides \(w\), take
\begin{equation}\label{uqg:four:generators}
b_{i,0}=x_i^L,\qquad
b_{i,j}=x_i^{d-j}x_{i+1}^{1+(j-1)s}
\quad(1\le i\le n,\ 1\le j\le d-1).
\end{equation}
These \(nd+1\) monomials generate \(R\). To verify the assertion on a boundary face, use exponents \((a,b)\) of \(x_i^ax_{i+1}^b\). Its invariant lattice is
\[
a+b\equiv0\pmod q,\qquad a+(1-d)b\equiv0\pmod m.
\]
This lattice has index \(qm=L\) in \(\mathbb{Z}^2\), since \(\gcd(q,m)=1\). In angular order, the vectors
\[
(L,0),\qquad (d-j,1+(j-1)s)\quad(1\le j\le d)
\]
begin and end at the two primitive boundary rays; the final vector is \((0,L)\). Each belongs to the invariant lattice. Their consecutive determinants are \(L\), since the first determinant is \(L\), and for \(1\le j<d\),
\[
(d-j)(1+js)-(d-j-1)(1+(j-1)s)=1+(d-1)s=L.
\]
Every consecutive pair therefore forms a basis of the invariant lattice. Its cone contains only nonnegative integral combinations of that pair, proving generation on each boundary face. The final ray is already \(b_{i+1,0}\), so it need not be repeated in \eqref{uqg:four:generators}.

For an interior lattice point, subtract \((q,1,1)\). This remains in \(\Lambda\), and it remains in \(\mathcal{S}_n\): the primitive facet functionals are \(b,c,\ell-b-c\) for \(n=3\); \(b,c,\ell-b,\ell-c\) for \(n=4\); and \(b,c,2\ell-b,2\ell-c,\ell-b+c,\ell+b-c\) for \(n=6\). Each takes value one at \((q,1,1)\). Subtracting this vector as many times as the minimum facet value reaches a boundary face. The boundary calculation proves generation of the whole semigroup. In particular, the chosen monomial orbit \(x_i^{d-1}x_{i+1}\) is the list \(b_{i,1}\).

Let \(Q=\mathbb{C}[y_0,y_{i,j}]\), with \(y_0\mapsto w\) and \(y_{i,j}\mapsto b_{i,j}\), and let \(I\) be the kernel of this surjection. It is the toric ideal of the displayed exponent vectors.

The generator list is \(G\)-stable, since
\[
T(w)=\eta^qw,\quad T(b_{i,0})=\eta^Lb_{i+1,0},\quad T(b_{i,j})=\eta^{d+(j-1)q(P-1)}b_{i+1,j}.
\]
These substitutions have order dividing \(nP\), because all displayed degrees are divisible by \(q\). Apply the procedure of Remark~\ref{uqg:fixed-presentations} to obtain an equivariant equation row \(a\) for the toric ideal and a matrix \(U\) generating its entire first syzygy module.

Put
\[
h=\sum_{i=1}^n\zeta^{i-1}y_{i,1}.
\]
It is invariant and maps to \(f\). We obtain the actual presentation of the canonical cover
\[
B=Q/(a,h).
\]
This invariant function in the domain \(R\), together with the displayed equivariant toric presentation, is the input for Subsection~\ref{uqg:wnh:matrices}. It gives the polynomial family
\[
\mathscr{B}=Q[t]/(a,h-t),\qquad \mathscr{B}/t\mathscr{B}=B,
\]
with the full syzygy lift~\eqref{uqg:cartier-matrix}. Its nonzero fibers are the smooth surfaces \(Y_t/K\), with free residual action by \(G=\Gamma/K\). Proposition~\ref{cqg:descent} therefore gives a strong \(\mathbb{Q}\)-Gorenstein smoothing of \(A'\). Since its graph occurs in Table~\ref{table:fourarm}, Corollary~\ref{cqg:qhd-verification} makes this a rational homology disk smoothing. The necessity assertion of Theorem~\ref{thm:cross-ratio} now gives the ordered attachment points. The section-ring description \eqref{cqg:section-algebra} identifies \(A'\) with the prescribed algebra \(A\), proving the assertion.
\end{proof}

\section{Deformations of regular sequences in toric algebras}\label{cqg:toric-sections}

The QHD smoothings of Figures~1(i) and~1(j) are known~\cite{BS}. We construct their canonical covers as quotients of toric algebras by regular sequences of two and three equations, respectively. Exactly one equation character is trivial, allowing a constant equivariant perturbation. Proposition~\ref{uqg:regular-section} supplies flatness and the full polynomial syzygy lift. We identify the central invariant algebras, including the contribution of additional components in the auxiliary zero loci, and verify smoothness and freeness on quotient charts.

\subsection{The family in Figure~1(i)}\label{cqg:i:section}

\begin{proposition}\label{cqg:i:existence}
Let \(m\ge3\), and put
\[
n=m^2-2m+2,\qquad a_m=m^2-3m+3.
\]
The weighted homogeneous surface singularity with Seifert data
\begin{equation*}
\left(2;(6,1),(3,1),(n,a_m)\right)
\end{equation*}
admits an algebraic strong \(\mathbb{Q}\)-Gorenstein smoothing. Taking \(m=q+3\) gives every member of Figure~1(i).
\end{proposition}

\begin{proof}
The calculations in \eqref{cqg:i:canonical-cover} and~\eqref{cqg:i:canonical-character-uac} below identify the canonical cover and its action. Lemmas~\ref{cqg:i:flatness} and~\ref{cqg:i:smooth-free}, proved below, construct an equivariant flat deformation with smooth noncentral fibers and free covering action. Corollary~\ref{cqg:qhd-verification} gives the required smoothing.
\end{proof}

\subsubsection{The canonical cover}

The identity \([2^r,m]=((r+1)m-r)/(rm-r+1)\) gives \(n/a_m=[2^{m-2},m]\). The rational Euler number and discriminant order are
\[
e=2-\frac16-\frac13-\frac{a_m}{n}=\frac{m^2}{2n},\qquad
\lvert D\rvert=18ne=9m^2.
\]
The link is a rational homology sphere, and the universal abelian cover is the Brieskorn hypersurface~\cite{Neumann1983}
\begin{equation*}
C^{\mathrm{ab}}=\mathbb{C}[x,y,z]/(x^6+y^3+z^n).
\end{equation*}
Let \(E_6\) be the component of the arm with one vertex of weight \(-6\). Its row in the negative inverse intersection matrix, restricted to the three end components, is
\[
\left(\frac16+\frac1{36e},\frac1{18e},\frac1{6ne}\right)
=\frac1{9m^2}\left(2m^2-m+1,n,3\right).
\]
Set \(A_1=2m^2-m+1\), and let \(\xi\) be a primitive \(9m^2\)-th root of unity. The meridian of \(E_6\) therefore acts by
\begin{equation*}
\delta(x,y,z)=(\xi^{A_1}x,\xi^ny,\xi^3z).
\end{equation*}
We have \(A_1\equiv1\pmod m\), while \(A_1\not\equiv0\pmod3\). Hence \(\gcd(A_1,9m^2)=1\), so this meridian has order \(9m^2\) and generates the discriminant group. The three terms in the equation have the same character, since
\[
6A_1-3n=9m^2.
\]
The residue canonical form has character
\begin{equation}\label{cqg:i:canonical-character-uac}
A_1+n+3-3n=A_1-2n+3=3m
\quad\pmod{9m^2}.
\end{equation}
Thus the canonical index is \(3m\), and the kernel of the canonical character is \(\langle\delta^{3m}\rangle\), of order \(3m\).

We first take its subgroup of order three, generated by \(\delta^{3m^2}\). Its weights on \(x,y,z\) are \((A_1,n,0)\) modulo three. The first two are nonzero and opposite, so the invariant algebra is generated by \(x^3,y^3,xy,z\). Put \(u=xy\) and \(v=x^3\). Eliminating \(y^3=-v^2-z^n\) gives
\[
u^3+v^3+vz^n=0.
\]
On these coordinates, the remaining subgroup of order \(m\) has weights
\[
(m^2-m+1,2m^2-m+1,1)\equiv(1,1,1)\pmod m.
\]
Consequently the canonical cover is
\begin{equation}\label{cqg:i:canonical-cover}
B_m=C_m^{\mu_m},
\qquad
C_m=\mathbb{C}[u,v,z]/(u^3+v^3+vz^n),
\end{equation}
where \(\mu_m\) acts scalarly. The residual canonical group is \(G=\mu_{3m}\). This is the row \(W=\Spec B_m\), \(\Gamma=G\) of Table~\ref{table:smoothing-covers}.

The polynomial defining \(C_m\) is Eisenstein at \(v\), and its partial derivatives \(3u^2,3v^2+z^n,nvz^{n-1}\) vanish simultaneously only at the origin. Thus \(C_m\) is a normal domain. The scalar action is free off the origin, and its determinant character equals the equation character. Hence \(B_m\) is normal and Gorenstein.

Let \(\epsilon=\xi^{3m}\), a primitive \(3m\)-th root, and put \(\omega=\epsilon^m\). A monomial \(u^av^bz^c\) with \(a+b+c=m\) has residual character
\begin{equation}\label{cqg:i:monomial-character}
1+(m-1)a+(2m-1)b\quad\pmod{3m}.
\end{equation}
Indeed, its numerator before passage to \(G\) is
\[
(m^2-m+1)a+(2m^2-m+1)b+c
=m\left(1+(m-1)a+(2m-1)b\right).
\]
In particular the elements
\begin{equation}\label{cqg:i:XYZ-central}
X=vz^{m-1},\qquad Y=uz^{m-1},\qquad Z=z^m
\end{equation}
have characters \(2m,m,1\), respectively.

\subsubsection{A presentation using a toric algebra}

All indices in the following construction are read modulo three. Factoring \(u^3+v^3\) gives the three linear forms \(v+\omega^iu\). Put
\[
l_i=X+\omega^iY=(v+\omega^iu)z^{m-1}.
\]
Multiplication in the coordinate ring of the canonical cover satisfies
\begin{equation}\label{cqg:i:central-products}
\sum_{i=0}^2\omega^il_i=0,
\qquad
l_0l_1l_2=-XZ^m.
\end{equation}
For the second equality, use \(u^3+v^3=-vz^n\) and \(n+2m-2=m^2\). Choose the cyclic ordering \(0,1,2\), and define
\begin{equation*}
\overline M_{i,j}=(v+\omega^iu)(v+\omega^{i+1}u)^{j-1}z^{m-j},
\qquad 1\le j\le m.
\end{equation*}
These belong to \(B_m\), and after inverting \(Z\) they satisfy \(\overline M_{i,j}=l_il_{i+1}^{j-1}/Z^{j-1}\). We now construct a toric algebra realizing these multiplication identities without inverting \(Z\).

Let
\[
P=\mathbb{C}[S_0,S_1,S_2,h_0,h_1,h_2],
\]
and let \(T=(\mathbb{C}^*)^2\) have character lattice generated by \(e_0,e_1,e_2\), with the sole relation \(e_0+e_1+e_2=0\). Give the variables the weights
\begin{equation*}
\operatorname{wt}_T(S_i)=e_i,
\qquad
\operatorname{wt}_T(h_i)=-(m-1)e_{i+1}-e_{i+2}.
\end{equation*}
The following monomials are invariant:
\begin{equation}\label{cqg:i:toric-generators}
\begin{aligned}
Z&=S_0S_1S_2,\qquad H=h_0h_1h_2,\qquad L_i=S_{i+1}^{m-1}S_{i+2}h_i,\\
M_{i,j}&=S_{i+1}^{m-j}S_{i+2}^{1+(m-2)(j-1)}h_ih_{i+1}^{j-1},
\qquad 1\le j\le m.
\end{aligned}
\end{equation}
Here \(M_{i,1}=L_i\), and direct multiplication gives
\[
L_0L_1L_2=Z^mH,
\qquad
Z^{j-1}M_{i,j}=L_iL_{i+1}^{j-1}.
\]
Set \(Q=P^T\).

\begin{lemma}\label{cqg:i:semigroup}
The algebra \(Q\) is generated by \(Z,H\), and the \(3m\) monomials \(M_{i,j}\) in \eqref{cqg:i:toric-generators}.
\end{lemma}

\begin{proof}
Write an invariant monomial as \(\prod_iS_i^{a_i}h_i^{b_i}\). Its invariance is equivalent to the existence of an integer \(\lambda\) such that
\[
a_i=\lambda+(m-1)b_{i-1}+b_{i-2}\qquad(i=0,1,2).
\]
Divide first by the largest possible power of \(H\), and then by the largest possible power of \(Z\). The remaining monomial has \(\min b_i=\min a_i=0\). By cyclic permutation, suppose \(b_2=0\), and put \(b_0=p,b_1=q\). Then
\[
\lambda=-\min\{q,(m-1)p\},
\]
because \(p+(m-1)q\ge q\). If \(q\le(m-1)p\), the exponents of the \(S_i\) are
\[
(0,(m-1)p-q,p+(m-2)q).
\]
Choose integers \(r_1,\ldots,r_p\) in \(\{0,\ldots,m-1\}\) with sum \(q\). The monomial is then \(\prod_{\nu=1}^p M_{0,r_\nu+1}\). If \(q>(m-1)p\), it is \(M_{0,m}^pL_1^{q-(m-1)p}\). The case \(p=q=0\) gives the constant monomial. Restoring the removed powers of \(H,Z\), and applying the same argument after cyclic permutation, proves the claim.
\end{proof}

Define
\begin{equation}\label{cqg:i:central-equations}
F_1=\sum_{i=0}^2\omega^iL_i,
\qquad
F_0=\sum_{i=0}^2L_i+3H.
\end{equation}
The relation \(F_0=0\) sets \(H=-X\), where \(X=\sum_iL_i/3\). Thus \eqref{cqg:i:central-equations} is obtained directly from the multiplication identities in the central fiber \eqref{cqg:i:central-products}.

\begin{lemma}\label{cqg:i:central-identification}
The sequence \(F_1,F_0\) is regular in \(P\) and in \(Q\), and there is an isomorphism
\begin{equation*}
Q/(F_1,F_0)\simeq B_m
\end{equation*}
sending \(Z\) to \(z^m\), \(H\) to \(-vz^{m-1}\), and \(M_{i,j}\) to \(\overline M_{i,j}\).
\end{lemma}

\begin{proof}
Put \(D_0=P/(F_1,F_0)\). After inverting \(Z\), all \(S_i\) are invertible. The torus weights of any two of them form a lattice basis, so setting \(S_1=S_2=1\) gives an algebraic slice. Define
\[
X=\frac{L_0+L_1+L_2}{3},\qquad
Y=\frac{L_0+\omega^2L_1+\omega L_2}{3}.
\]
The equation \(F_1=0\) gives \(L_i=X+\omega^iY\), while \(F_0=0\) gives \(H=-X\). Therefore
\begin{equation*}
(D_0[Z^{-1}])^T
\simeq
\mathbb{C}[X,Y,Z,Z^{-1}]/(X^3+Y^3+XZ^m).
\end{equation*}
Under \eqref{cqg:i:XYZ-central}, the right side is \(B_m[Z^{-1}]\). Indeed, the scalar invariants in \(C_m[z^{-1}]\) are generated by \(u/z,v/z,z^m,z^{-m}\); multiplying the first two generators by \(Z\) gives \(Y,X\). The right side is a domain and is smooth: at a singular point its derivatives would give \(Y=0\), \(mXZ^{m-1}=0\), and \(3X^2+Z^m=0\), which is impossible when \(Z\ne0\). The full space \(\operatorname{Spec}D_0[Z^{-1}]\) is the product of this smooth irreducible surface and \(T\). Let \(C\) denote its closure in \(\operatorname{Spec}D_0\).

We determine every other component. If exactly one variable \(S_i\) vanishes, then \(L_{i+1}=L_{i+2}=0\), and \(F_1=0\) forces \(L_i=0\). Since the other two \(S\)-variables are nonzero, this forces \(h_i=0\). The resulting component is
\[
P_i=V(S_i,h_i),
\]
which is a four-dimensional coordinate space contained in the zero set. If at least two \(S_i\) vanish, all the \(L_i\) vanish, and \(F_0=0\) gives \(H=0\). Such a stratum has dimension at most three: two \(S\)-coordinates and at least one \(h\)-coordinate vanish. It follows that \((F_1,F_0)\) has height two and that its only possible irreducible components are \(C,P_0,P_1,P_2\). As a height-two ideal generated by two elements in a polynomial ring, it is a complete intersection and is unmixed.

The main component is smooth on \(Z\ne0\). On \(P_0\), the Jacobian minor in the columns \(S_0,h_0\) is
\[
S_2^{m-1}h_1\left((\omega-1)S_1^{m-1}S_2+3\omega h_1h_2\right),
\]
which is generically nonzero; cyclic permutation treats \(P_1,P_2\). Thus the unmixed complete intersection \(D_0\) is generically reduced on every component, hence reduced.

The extra components contribute no invariant functions. On \(P_0\), the surviving variables are \(S_1,S_2,h_1,h_2\). In the basis \(e_1,e_2\), their weights are
\[
(1,0),\quad(0,1),\quad(1,-(m-2)),\quad(m-2,m-1).
\]
The linear functional \((m-1,1)\) is strictly positive on all four weights. Hence every invariant polynomial on \(P_0\) is constant. The same conclusion holds on the other \(P_i\). All these coordinate components contain the origin. The main component \(C\) also contains it, since the locus \(h_0=h_1=h_2=0\), \(S_0S_1S_2\ne0\), specializes to the origin. Lemma~\ref{uqg:central-components} therefore gives an injection
\[
D_0^T\hookrightarrow B_m[Z^{-1}].
\]
By Lemma~\ref{cqg:i:semigroup}, its image is generated by the stated images of \(Z,H,M_{i,j}\), all of which belong to \(B_m\).

We verify that these images generate all of \(B_m\). For \(j=1\), the three forms \(v+\omega^iu\) span the binary linear forms. For \(j\ge2\), the three degree-\(j\) forms
\[
(v+\omega^iu)(v+\omega^{i+1}u)^{j-1}\qquad(i=0,1,2)
\]
form a basis of \(\mathbb{C}[u,v]_j/(u^3+v^3)\mathbb{C}[u,v]_{j-3}\), where the denominator is zero for \(j=2\). To see this, evaluate at the three distinct projective roots of \(u^3+v^3\). Each of the three forms vanishes at two roots and is nonzero at the third, and the quotient has dimension three. Thus, for \(3\le j\le m\), any degree-\(j\) binary form times \(z^{m-j}\) is a linear combination of the \(\overline M_{i,j}\), plus a term
\[
(u^3+v^3)Q_{j-3}(u,v)z^{m-j}
=-Z^{m-2}\left(vQ_{j-3}(u,v)z^{m-j+2}\right).
\]
The expression in parentheses has weighted degree in the two variables \(j-2\) and total auxiliary degree \(m\). Induction on \(j\), starting with \(j=0,1,2\), therefore gives all degree-\(m\) monomials. The standard monomials in \(C_m\) are \(u^av^bz^c\) with \(0\le a\le2\). Scalar invariance means that \(a+b+c\) is divisible by \(m\). Such a monomial factors into monomials of total auxiliary degree \(m\), each still having \(u\)-exponent at most two. Hence the degree-\(m\) monomials generate \(B_m\).

Exactness of torus invariants identifies \(D_0^T\) with \(Q/(F_1,F_0)\). Applying the same functor to the Koszul resolution over \(P\) proves that \(F_1,F_0\) form a regular sequence in \(Q\).
\end{proof}

\subsubsection{The equivariant deformation}

Choose \(c_0=\epsilon,c_1=c_2=1\), and set
\[
d_i=\frac{\omega^2}{c_{i+1}^{m-1}c_{i+2}}.
\]
The transformation
\begin{equation*}
g(S_i)=c_iS_{i-1},\qquad g(h_i)=d_ih_{i-1}
\end{equation*}
normalizes \(T\), cyclically permuting its character lattice. Direct substitution gives
\begin{equation}\label{cqg:i:normalizer-invariants}
g(Z)=\epsilon Z,\qquad
g(L_i)=\omega^2L_{i-1},\qquad
g(H)=\omega^2H.
\end{equation}
For the last equality, use
\[
\prod_i d_i=\frac{\omega^6}{(c_0c_1c_2)^m}=\epsilon^{-m}=\omega^2.
\]
Moreover \(g^3(S_i)=\epsilon S_i\) and \(g^3(h_i)=\omega^2h_i\). Its \(3m\)-th power is the torus element acting as \(S_i\mapsto\omega S_i\), \(h_i\mapsto\omega^{-m}h_i\). Therefore \(g\) induces an action of \(G=\mu_{3m}\) on \(Q\). Its order is exactly \(3m\), since it acts on \(Z\) by the primitive root \(\epsilon\).

The equation characters are
\begin{equation*}
g(F_1)=F_1,\qquad g(F_0)=\omega^2F_0.
\end{equation*}
Thus the action on the equations allows a constant deformation of \(F_1\) over a base with trivial \(G\)-action. Define
\begin{equation*}
D=P[\tau]/(F_1+3\tau,F_0),
\qquad
\mathscr{B}=D^T=Q[\tau]/(F_1+3\tau,F_0),
\end{equation*}
and let \(g(\tau)=\tau\). Both equations are preserved, the first invariantly and the second with character \(\omega^2\). The central action satisfies
\[
g(X)=\omega^2X,\qquad g(Y)=\omega Y,\qquad g(Z)=\epsilon Z.
\]
These agree with \eqref{cqg:i:monomial-character}. Since \(X,Y,Z,Z^{-1}\) generate the localization of the central fiber and \(B_m\) injects into that localization, the action agrees on all of \(B_m\) with its action of the canonical covering group.

\begin{lemma}\label{cqg:i:flatness}
The algebra \(\mathscr{B}\) is flat over \(\mathbb{C}[\tau]\), and \(\mathscr{B}/\tau\mathscr{B}\simeq B_m\).
\end{lemma}

\begin{proof}
Use the toric presentation determined by \eqref{cqg:i:toric-generators}, with equivariant equations \(k\) and their full syzygy matrix chosen as in Remark~\ref{uqg:fixed-presentations}. The action on the monomial generators is
\[
g(Z)=\epsilon Z,\qquad g(H)=\omega^2H,\qquad
g(M_{i,j})=\omega^{2j}\epsilon^{-(j-1)}M_{i-1,j}.
\]
Let \(s_1,s_2\) be the linear equivariant lifts of \(F_1,F_0\). Lemma~\ref{cqg:i:central-identification} gives regularity and the central fiber. Proposition~\ref{uqg:regular-section} applies to \((k,s_1+3\tau,s_2)\), with the positive grading
\[
\deg S_i=\frac23,\qquad\deg h_i=\frac m3,\qquad\deg\tau=m.
\]
It gives the full polynomial syzygy lift and flatness of \(\mathscr{B}\), whose invariant presentation follows by exactness of torus invariants.
\end{proof}

\subsubsection{Smoothness, freeness, and descent}

\begin{lemma}\label{cqg:i:smooth-free}
For every \(\tau\ne0\), the fiber of \(\operatorname{Spec}\mathscr{B}\) is smooth and the covering group of the canonical cover \(G\) acts freely on it.
\end{lemma}

\begin{proof}
The equation \(\sum_i\omega^iL_i=-3\tau\) shows that the invariant open sets \(L_i\ne0\) cover a noncentral fiber. On \(L_0\ne0\), the coordinates \(S_1,S_2,h_0\) are nonzero. The first two torus weights form a lattice basis, so there is a unique torus translate with \(S_1=S_2=1\). This gives an algebraic product decomposition with \(T\), and hence an affine quotient slice. On the slice put
\[
s=S_0,\quad p=h_1,\quad q=h_2,\quad
a=h_0,\quad b=sp,\quad c=s^{m-1}q.
\]
Here \(a\ne0\), and the equations become
\begin{equation}\label{cqg:i:slice-equations}
a+\omega b+\omega^2c=-3\tau,
\qquad
a(3pq+1)+b+c=0.
\end{equation}
Eliminate \(a=-3\tau-\omega b-\omega^2c\), and denote the second left side by \(F(s,p,q)\). Put
\[
U=3pq+1,\qquad T_0=a(U-1),\qquad
B_0=b(1-\omega U),\qquad C_0=c(1-\omega^2U).
\]
Direct differentiation, including the dependence of \(a\) on \(s,p,q\), gives
\begin{equation*}
pF_p=T_0+B_0,\qquad
qF_q=T_0+C_0,\qquad
sF_s=B_0+(m-1)C_0.
\end{equation*}
At a critical point all three right sides vanish. The first two imply \(B_0=C_0=-T_0\); substitution in the third gives \(mT_0=0\). Thus \(T_0=B_0=C_0=0\). Since \(a\ne0\), we obtain \(U=1\). Since \(1-\omega\) and \(1-\omega^2\) are nonzero, we then obtain \(b=c=0\). But the second equation of \eqref{cqg:i:slice-equations} becomes \(a=0\), a contradiction. This argument permits \(s,p,q\) to vanish and therefore omits no coordinate stratum. The other two charts have the same form: after cyclic permutation and multiplication of the first equation by \(\omega^{-i}\), the nonzero parameter is replaced by \(\omega^{-i}\tau\). Hence every noncentral fiber is smooth.

We next check every possible nontrivial stabilizer. Suppose \(g^j\), with \(0<j<3m\), fixes a quotient point. If \(Z\ne0\), the equation \(g(Z)=\epsilon Z\) would give \(\epsilon^j=1\), contrary to this range. Thus \(Z=0\). Some \(L_i\) is nonzero, which forces \(S_{i+1},S_{i+2},h_i\ne0\). Consequently exactly one \(S\)-coordinate vanishes, namely \(S_i\), and \(L_i\) is the only nonzero \(L\)-coordinate. A power of \(g\) not divisible by three permutes the position of this unique nonzero coordinate, so it cannot fix the point. It remains to consider \(j=3r\).

At the point in question, \(H=-L_i/3\ne0\), so all the \(h\)-coordinates are nonzero. The polynomial monomial
\[
M_{i,2}=S_{i+1}^{m-2}S_{i+2}^{m-1}h_ih_{i+1}
\]
is therefore nonzero. From either \eqref{cqg:i:normalizer-invariants} or its monomial formula,
\[
g^3(M_{i,2})=\epsilon^{-3}M_{i,2}.
\]
Thus \(g^{3r}\) can fix the point only if \(m\mid r\), in which case it is the identity. This excludes all nontrivial stabilizers.
\end{proof}

\subsection{The family in Figure~1(j)}\label{cqg:j:section}

\begin{proposition}\label{cqg:j:existence}
Let \(m\ge4\), put \(n=m^2-3m+3\), and let \(A\) be the weighted homogeneous surface algebra with Seifert data
\[
\left(2;(2,1),(6,1),(n,(m-2)^2)\right).
\]
Then \(A\) admits an algebraic strong \(\mathbb{Q}\)-Gorenstein smoothing. Taking \(m=q+4\) gives every member of Figure~1(j).
\end{proposition}

\begin{proof}
Lemma~\ref{cqg:j:cover}, proved below, identifies the canonical cover and its action. Lemmas~\ref{cqg:j:flatness} and~\ref{cqg:j:smooth-free}, also proved below, supply an equivariant flat deformation with smooth noncentral fibers and free covering action. Corollary~\ref{cqg:qhd-verification} gives the required smoothing of \(A\).
\end{proof}

\subsubsection{The canonical cover}

\begin{lemma}\label{cqg:j:cover}
The canonical index of \(A\) is \(2m\). Its canonical cover and covering group are
\[
B=\left(\mathbb{C}[u,v,w]/(u^2+v^4+vw^n)\right)^{\mu_m(2,1,1)},
\qquad G=\mu_{2m}.
\]
The algebra \(B\) is generated by
\begin{equation}\label{cqg:j:generators}
r_j=v^jw^{m-j}\quad(0\le j\le m),
\qquad
s_j=uv^{j-2}w^{m-j}\quad(2\le j\le m).
\end{equation}
For a suitable generator of \(G\), their characters, written additively modulo \(2m\), are
\begin{equation}\label{cqg:j:characters}
\chi(r_j)=1+(m-1)j,
\qquad
\chi(s_j)=m+1+(m-1)j.
\end{equation}
The canonical character is \(1\).
\end{lemma}

\begin{proof}
The third arm is \(n/(m-2)^2=[2^{m-3},m]\). Let \(M\) be the negative intersection matrix and \(E_b\) the component of weight \(-6\). The Schur complement gives
\[
e=2-\frac12-\frac16-\frac{(m-2)^2}{n}=\frac{m^2}{3n}>0,
\qquad \det M=12ne=4m^2.
\]
Thus \(M\) is positive definite. Its row indexed by \(E_b\), restricted to the three end components, is
\begin{equation}\label{cqg:j:inverse-column}
(M^{-1})_{E_b,\mathrm{leaves}}
=\left(\frac1{12e},\frac16+\frac1{36e},\frac1{6ne}\right)
=\frac1{4m^2}\left(n,m^2-m+1,2\right).
\end{equation}
The description of the universal abelian cover of a weighted homogeneous singularity identifies its universal abelian cover with the Brieskorn surface associated to the three arm determinants~\cite{Neumann1983}:
\[
Z=\{x_0^2+y_0^6+w^n=0\}.
\]
By \eqref{cqg:j:inverse-column}, the discriminant action associated to \(E_b\) is
\begin{equation*}
\delta=\frac{1}{4m^2}\left(n,m^2-m+1,2\right).
\end{equation*}
This notation means that the three coordinates are multiplied by the corresponding exponentials. Since \(m^2-m+1\) is odd and is congruent to \(1\) modulo \(m\), it is relatively prime to \(4m^2\). Therefore \(\delta\) has order \(4m^2\) and generates the entire discriminant group. The three terms of the equation have the same character, because
\[
6(m^2-m+1)-2n=4m^2,
\qquad n\cdot2=2n.
\]
Thus the congruence condition is satisfied. The action is free away from the origin: a point with \(y_0\ne0\) has trivial stabilizer, while a point with \(y_0=0\) has \(x_0w\ne0\); fixing \(w\) then leaves only the possible exponent \(2m^2\), which changes the sign of \(x_0\), since \(n\) is odd. The quotient is weighted homogeneous. The three attachment points on its central rational curve have no moduli, so the standard description of a weighted homogeneous algebra by its rational divisor on the orbit curve with its specified attachment points identifies this quotient with \(A\).

The residue canonical form on \(Z\) has character
\[
\frac{n+(m^2-m+1)+2-2n}{4m^2}
=\frac{1}{2m}.
\]
Its kernel has order \(2m\) and is generated by \(\delta^{2m}\). First quotient by its order-two subgroup, generated by \(\delta^{2m^2}\). This subgroup acts by \((x_0,y_0,w)\mapsto(-x_0,-y_0,w)\). Its invariant algebra has generators \(u=x_0y_0\), \(v=y_0^2\), and \(w\): the remaining invariant \(x_0^2\) is \(-v^3-w^n\), and consequently
\[
u^2+v^4+vw^n=0.
\]
The induced action of \(\delta\) on \((u,v,w)\) has angles
\[
\frac{1}{2m^2}\left(m^2-2m+2,m^2-m+1,1\right).
\]
Its \(2m\)-th power is the action \(\mu_m(2,1,1)\), proving the displayed formula for \(B\). Thus \(W=\Spec B\) and \(\Gamma=G\) in Table~\ref{table:smoothing-covers}. The remaining quotient group has order \(2m\), and substitution in \eqref{cqg:j:generators} gives \eqref{cqg:j:characters}.

To verify generation, reduce every polynomial modulo the monic relation in \(u^2\). The resulting standard monomials have \(u\)-exponent zero or one. For the auxiliary degrees \((2,1,1)\), an invariant standard monomial has degree divisible by \(m\). A monomial \(v^bw^c\) of degree \(hm\) is a product of degree-\(m\) monomials \(r_j\). For \(uv^bw^c\), one has \(b+c=hm-2\); choose \(b'\le b\), \(c'\le c\) with \(b'+c'=m-2\), factor off \(uv^{b'}w^{c'}=s_{b'+2}\), and factor the remainder into the \(r_j\). This proves the assertion.

The intermediate hypersurface is a normal domain: its equation is Eisenstein at \(v\), and its partial derivatives vanish simultaneously only at the origin. The auxiliary \(\mu_m\)-action is free off the origin, since \(v=w=0\) forces \(u=0\). Its determinant character and equation character are both \(4\) modulo \(m\). Thus \(B\) is Gorenstein, and the residue computation above identifies its canonical character with \(1\).
\end{proof}

\subsubsection{A presentation using a toric algebra}

Put \(k=m-2\). Let \(P=\mathbb{C}[U,V_1,S,T,h,l]\), and let the torus \(\mathbb{T}=\mathbb{C}^*\) act with weights
\begin{equation*}
(1,-1,-1,1,-k,k)
\end{equation*}
on these six variables. Define
\begin{equation}\label{cqg:j:equations}
F_0=V_1T-US,
\qquad F_1=hl+\frac{US+V_1T}{2},
\qquad F_2=T^kh+S^kl-2UV_1.
\end{equation}
All three polynomials are torus invariant. We write \(C_0=P/(F_0,F_1,F_2)\).

\begin{lemma}\label{cqg:j:invariant-presentation}
The three polynomials in \eqref{cqg:j:equations} form a regular sequence, and \(C_0^{\mathbb{T}}\simeq B\). Under this isomorphism,
\[
z=ST=r_0,
\qquad y=\frac{US+V_1T}{2}=r_1,
\qquad x=UV_1=r_2,
\]
and, for \(2\le j\le m\),
\begin{equation}\label{cqg:j:two-chains}
H_j=U^{j-2}T^{m-j}h,
\qquad L_j=V_1^{j-2}S^{m-j}l,
\qquad r_j=\frac{H_j+L_j}{2},
\qquad s_j=\frac{H_j-L_j}{2i}.
\end{equation}
In particular, the presentation is obtained from the two sequences \(r_j+i s_j\) and \(r_j-i s_j\) and their multiplication relations of the coordinate ring of the canonical cover.
\end{lemma}

\begin{proof}
In \(B\), the generators satisfy
\[
xz=y^2,
\qquad z(r_{j+1}\pm i s_{j+1})=y(r_j\pm i s_j),
\]
and
\[
(r_2+i s_2)+(r_2-i s_2)=2x,
\qquad (r_2+i s_2)(r_2-i s_2)=-yz^k.
\]
The last equality is the hypersurface equation multiplied by \(w^{2m-4}\), since \(n+2m-4=m^2-m-1\). The factorization using the torus variables \eqref{cqg:j:two-chains} expresses these two sequences; the three equations in \eqref{cqg:j:equations} express their common ratio, their sum, and their product.

We check the central invariant algebra scheme-theoretically. On \(ST\ne0\), fix \(T=1\), so \(S=z\), \(V_1=US\), and \(U=\eta\). Put \(\xi=(h-S^kl)/(2iS)\). Then
\[
h=S(\eta^2+i\xi),
\qquad l=S^{1-k}(\eta^2-i\xi),
\]
and the remaining equation is
\begin{equation}\label{cqg:j:central-chart}
\xi^2+\eta^4+\eta z^{m-3}=0.
\end{equation}
Thus the localized quotient is \(\mathbb{C}[z,z^{-1},\eta,\xi]/(\xi^2+\eta^4+\eta z^{m-3})\). This is a domain: in \(\mathbb{C}(z)(\eta)\), the element \(-\eta(\eta^3+z^{m-3})\) has odd valuation at \(\eta=0\), and hence is not a square. The torus torsor on this open set is trivial, so the part of \(\operatorname{Spec}C_0\) meeting \(ST\ne0\) has one irreducible component, of dimension three.

The other possible components are
\[
P_+=(V_1,S,h),
\qquad P_-=(U,T,l).
\]
Indeed, if \(S=0\) and \(T\ne0\), the equations force \(V_1=h=0\); if \(T=0\) and \(S\ne0\), they force \(U=l=0\). On \(S=T=0\), the remaining equations are \(hl=UV_1=0\), so this locus has dimension two. Since an ideal generated by three elements in a polynomial ring of dimension six has no component of dimension less than three, these are exactly the three components. The ideal has height three and is a complete intersection.

For the equivalent equations \(US-V_1T\), \(hl+US\), and \(F_2\), the Jacobian minors in the columns \((V_1,S,h)\) and \((U,T,l)\) are, on \(P_+\) and \(P_-\), respectively,
\[
-U(T^{k+1}+2Ul),\qquad V_1(S^{k+1}+2V_1h).
\]
They are generically nonzero, and the main component is generically smooth by \eqref{cqg:j:central-chart}. Thus the complete intersection has no embedded associated primes and is generically reduced on every component, so \(C_0\) is reduced.

On \(P_+\), the surviving coordinate weights are \(1,1,k\); on \(P_-\), they are \(-1,-1,-k\). Thus each extra component has only constant invariants. Both coordinate components contain the origin, and the main component contains it because it is stable under the positive grading below. Lemma~\ref{uqg:central-components} gives
\[
C_0^{\mathbb{T}}\hookrightarrow(C_0[1/(ST)])^{\mathbb{T}}.
\]
The polynomial invariant ring \(P^{\mathbb{T}}\) is generated by
\[
UV_1,US,V_1T,ST,hl,H_j,L_j\quad(2\le j\le m).
\]
To verify this, factor off every possible \(hl\) from an invariant monomial, then pair variables of weight \(1\) with variables of weight \(-1\). If a power of \(h\) remains, its negative weight is balanced by blocks of \(k\) variables from \(U,T\), giving the monomials \(H_j\); the case of \(l\) gives the \(L_j\). This exhausts the monomials. Exactness of torus invariants shows that their images generate \(C_0^{\mathbb{T}}\). The equations reduce these images to \(z,y,x,r_j,s_j\), with \(r_2=x\).

Under \eqref{cqg:j:central-chart}, these generators become
\[
z=w^m,
\qquad r_j=v^jw^{m-j},
\qquad s_j=uv^{j-2}w^{m-j}.
\]
The same chart describes \(B[1/z]\), using \(\eta=v/w\), \(\xi=u/w^2\). Thus the injective map above identifies \(C_0^{\mathbb{T}}\) with the subalgebra generated by \eqref{cqg:j:generators}, which is exactly \(B\).
\end{proof}

\subsubsection{The equivariant deformation}

Let \(\zeta\) be a primitive \(2m\)-th root of unity. The following transformation normalizes \(\mathbb{T}\):
\begin{equation*}
U\mapsto-\zeta^{-1}V_1,
\quad V_1\mapsto-U,
\quad T\mapsto S,
\quad S\mapsto\zeta T,
\quad h\mapsto\zeta^{-1}l,
\quad l\mapsto-\zeta h.
\end{equation*}
It conjugates a torus element to its inverse. Its square multiplies \(U,V_1\) by \(\zeta^{-1}\), \(S,T\) by \(\zeta\), and \(h,l\) by \(-1\). Consequently its \(2m\)-th power is the torus element \(-1\), because \((-1)^k=(-1)^m\). It therefore induces an action of \(G\) on the invariant algebra. Direct substitution gives
\begin{equation}\label{cqg:j:equation-representation}
F_0\mapsto F_0,
\qquad F_1\mapsto-F_1,
\qquad F_2\mapsto\zeta^{-1}F_2.
\end{equation}
For the last identity, use \(-\zeta^{k+1}=\zeta^{-1}\), since \(k+1=m-1\). Moreover, \(H_j\mapsto\zeta^{1+(m-1)j}L_j\) and \(L_j\mapsto\zeta^{1+(m-1)j}H_j\). Thus the induced action agrees with \eqref{cqg:j:characters}, including the action on the central canonical cover.

Equation~\eqref{cqg:j:equation-representation} allows a constant perturbation in exactly one invariant direction. We use it to define
\begin{equation*}
D=P[\tau]/(F_0-2i\tau,F_1,F_2),
\qquad \mathscr{B}=D^{\mathbb{T}},
\end{equation*}
where the torus and the lifted finite action fix \(\tau\).

\begin{lemma}\label{cqg:j:flatness}
The algebra \(\mathscr{B}\) is a finitely generated flat \(\mathbb{C}[\tau]\)-algebra, and \(\mathscr{B}/\tau\mathscr{B}\simeq B\) equivariantly. The selected equation perturbation has a polynomial lift of the first syzygy matrix of a polynomial presentation of \(B\).
\end{lemma}

\begin{proof}
Use the toric presentation of \(Q=P^{\mathbb{T}}\) by the monomial generators in the proof of Lemma~\ref{cqg:j:invariant-presentation}, with equivariant equations \(k\) and their full syzygy matrix as in Remark~\ref{uqg:fixed-presentations}. Exact torus invariants of the central Koszul resolution make \(F_0,F_1,F_2\) a regular sequence in \(Q\). Let \(s_1,s_2,s_3\) be their linear lifts. The row \((k,s_1-2i\tau,s_2,s_3)\) is equivariant by \eqref{cqg:j:equation-representation}, and homogeneous for
\[
\deg U=\deg V_1=2m-3,\quad\deg S=\deg T=3,\quad
\deg h=\deg l=m,\quad\deg\tau=2m.
\]
Proposition~\ref{uqg:regular-section} gives the full polynomial syzygy lift and flatness; exactness of torus invariants identifies the central fiber with \(B\).
\end{proof}

\subsubsection{Smoothness, freeness, and descent}

\begin{lemma}\label{cqg:j:smooth-free}
For \(\tau\ne0\), the fiber of \(\mathscr{B}\) is smooth and the action of \(G\) is free.
\end{lemma}

\begin{proof}
The equation \(V_1T-US=2i\tau\) implies that \(S,T\) cannot both vanish and that some nonzero coordinate has torus weight \(+1\) and some nonzero coordinate has weight \(-1\). Thus every orbit is closed and has trivial stabilizer. The torus quotient is geometric, and it is covered by the charts obtained by setting \(T=1\) or \(S=1\).

On the first chart, put \(c=i\tau\). The first and third equations eliminate
\[
V_1=US+2c,
\qquad h=2U^2S+4cU-S^kl.
\]
Substitution in the second equation gives the hypersurface
\[
F=-S^kl^2+2U^2Sl+4cUl+US+c=0.
\]
Its partial derivatives are
\[
F_U=4USl+4cl+S,
\qquad F_S=-kS^{k-1}l^2+2U^2l+U,
\qquad F_l=-2S^kl+2U^2S+4cU.
\]
At a critical point put \(q=Ul\) and \(a=US\). The equations \(UF_U=0\) and \(F-lF_l/2=0\) give
\[
a(4q+1)+4cq=0,
\qquad a(q+1)+c(2q+1)=0.
\]
The equation \(lF_l=0\) gives \(S^kl^2=aq+2cq\). Substituting this into \(SF_S=0\) gives
\[
a((2-k)q+1)-2kcq=0.
\]
Subtracting the first displayed equation yields \((k+2)q(a+2c)=0\). If \(q=0\), the first two equations force \(a=c=0\). If \(a=-2c\), the second equation gives \(-c=0\). Both contradict \(c\ne0\). This chart is smooth. The involution interchanging \(U,V_1\), \(T,S\), and \(h,l\), while replacing \(\tau\) by \(-\tau\), preserves the three equations. It identifies the other chart with the same calculation, so the entire fiber is smooth.

By \eqref{cqg:j:characters}, \(x\) and \(z\) have the faithful characters \(-1\) and \(1\). They do not vanish simultaneously on a nonzero fiber. If \(S=0\), then \(T\ne0\) and \(V_1T=2i\tau\), so \(V_1\ne0\). The additional equation \(x=UV_1=0\) gives \(U=0\). Now \(hl=-i\tau\ne0\), but \(T^kh=2UV_1=0\), a contradiction. The case \(T=0\) follows by the same involution. Thus a nonidentity element of \(G\), whose fixed points must have \(x=z=0\), fixes no point of the fiber.
\end{proof}

\section{Adjoining elements of the fraction field}\label{cqg:m}

We treat Figures~1(d)--(g) by one construction, adjoining elements of a fraction field. The existence of their QHD smoothings is known~\cite{SSW,BS}; the purpose here is to give equivariant presentations, identify the central fibers, and lift the full first syzygy module. The same multiplication formulas and normal forms work for all four families. Their only distinction in this construction is the parameter substitution at the end of the section. Let \(P\ge3\), \(Q\ge2\), and \(k\ge3\), and put
\begin{equation}\label{cqg:m:parameters}
\ell=PQ-P-Q,\qquad N=P+k\ell,\qquad s=kQ-k+1,\qquad n=(k-1)N-kQ.
\end{equation}
The identities used below are
\begin{equation}\label{cqg:m:identities}
\begin{aligned}
Ps-kQ&=N,&s(P-1)-k&=N-1,\\
s\ell+Q&=(Q-1)N,&n+kQ&=(k-1)N,\\
n\ell&=N(N-PQ)+PQ,&N-PQ&\ge0.
\end{aligned}
\end{equation}
For the last inequality, \(N-PQ=(k-1)PQ-(k-1)P-kQ\) increases with \(P,Q,k\) in the indicated range and is zero at \((P,Q,k)=(3,2,3)\). Moreover, \(n-N=(k-2)N-kQ\ge Q((k-2)P-k)\ge2Q(k-3)\ge0\). In particular \(\ell\ge1\), \(n\ge N\), and all the integers in \eqref{cqg:m:parameters} are positive.

\begin{proposition}\label{cqg:m:existence}
The weighted homogeneous surface singularity with Seifert data
\begin{equation}\label{cqg:m:seifert}
\left(2;(P,1),(Q,1),(n,n-N+1)\right)
\end{equation}
admits a strong \(\mathbb{Q}\)-Gorenstein smoothing. The smoothing is obtained by lifting equivariant equations and their first syzygy module, using the multiplication in a finite Gorenstein cover of its canonical cover. Its canonical index is \(N/\gcd(P,Q)\).
\end{proposition}

\begin{proof}
We use the notation of Lemmas~\ref{cqg:m:cover}, \ref{cqg:m:rolling}, and~\ref{cqg:m:flat-family}, which are stated and proved below. In particular,
\[
\begin{aligned}
R_0&=\mathbb{C}[a,b,c,t]/(c^k+ca^Q+b^P+tab),\qquad \alpha=a^{Q-1},\\
R_k&=R_0[q_1,\ldots,q_k].
\end{aligned}
\]
The group \(\Gamma=\mu_N\) acts on \(a,b,c\) with characters \((P-1,1,-\ell)\) modulo \(N\).

The cover in Table~\ref{table:smoothing-covers} is \(W=\Spec\widetilde B\). By Lemma~\ref{cqg:m:flat-family}, \(\operatorname{Spec}R_k\to\mathbb{A}^1_t\) is a flat Gorenstein family with coordinate ring of the central fiber \(\widetilde B\). Its localizations at \(b\) and \(c\) are respectively \((R_0)_b\) and \((R_0)_c\). The covering by the two affine open subsets follows from the following ideal containment. The coefficient formula gives \(p_{k-1}-(-t)^{k-1}\in(a,c)R_k\), while \(bq_k=\alpha q_{k-1}-p_{k-1}\) and \(\alpha\in(a)\). Hence \(t^{k-1}\in(a,b,c)R_k\). Since \(a^Q=-bq_1-c^{k-1}\in(b,c)R_k\), raising this inclusion to the \(Q\)-th power gives
\begin{equation}\label{cqg:m:chart-certificate}
t^{Q(k-1)}\in(b,c)R_k.
\end{equation}
Hence the general fiber is exactly
\[
\{c^k+ca^Q+b^P+tab=0\}\setminus\{b=c=0\}.
\]

At a critical point of this hypersurface for fixed \(t\ne0\), multiplication of the three partial derivatives by \(a,b,c\) gives
\[
a^Qc=-kc^k,\qquad tab=kQc^k,\qquad b^P=-\frac{kQ}{P}c^k.
\]
Substitution into the defining equation gives
\[
0=\left(1-k-\frac{kQ}{P}+kQ\right)c^k=\frac NP c^k.
\]
Thus \(c=0\), and the three partial derivatives then force \(a=b=0\). This point has been removed, so every nonzero fiber is smooth. For every \(g\in\Gamma\setminus\{1\}\), its ideal defining the fixed locus contains \(b\), since \(b\) has character one, and contains at least one of \(a,c\), since \(\gcd(P-1,\ell)=1\). If it contains \(c\), use \eqref{cqg:m:chart-certificate}. If it contains \(a\), the initial equation gives \(c^k\in(a,b)R_k\), and the \(k\)-th power of \eqref{cqg:m:chart-certificate} applies. Thus the ideal containment
\begin{equation*}
t^{kQ(k-1)}\in(g(u)-u:u\in R_k)
\end{equation*}
holds for every nonidentity element of \(\Gamma\). Hence \(\Gamma\) acts freely on every nonzero fiber. Quotienting by \(K\) gives smooth fibers with free residual canonical action.

The central canonical character is \(-\ell\), and the lifted deformation of its canonical cover has coordinate ring \(R_k^K\), with central fiber \(B\). Corollaries~\ref{cqg:larger-cover} and~\ref{cqg:qhd-verification} now apply.
\end{proof}

\subsection{The covering algebra and its generators}

We identify the central cover and construct a polynomial lift of its equations and first syzygy module.

\begin{lemma}\label{cqg:m:cover}
Let
\[
Z=\operatorname{Spec}\mathbb{C}[x,y,z]/(x^P+y^Q+z^n),\qquad H=\mu_N(-s,-k,1),\qquad \widetilde B=\mathcal{O}(Z)^H.
\]
There is an action of \(\Gamma=\mu_N\) on \(\widetilde B\) such that \(A=\widetilde B^\Gamma\) has Seifert data \eqref{cqg:m:seifert}. The action of \(\Gamma\) is free off the vertex, and its canonical character is \(-\ell\). If \(K\subset\Gamma\) is the kernel of this character, then \(B=\widetilde B^K\) is the canonical cover of \(A\).
\end{lemma}

\begin{proof}
Write diagonal transformations additively in \((\mathbb{Q}/\mathbb{Z})^3\), and set
\[
h=\frac1N(-s,-k,1),\qquad \gamma=\frac1{N^2}\left(N+s\ell,(P-1)N+k\ell,-\ell\right).
\]
The equation is semi-invariant under \(h\) by \eqref{cqg:m:identities}, and under \(\gamma\) because
\[
P(N+s\ell)-Q((P-1)N+k\ell)=0,\qquad P(N+s\ell)+n\ell=N^2.
\]
Moreover \(\gamma^N=h^{-\ell}\). Set \(\Pi=\langle h,\gamma\rangle\). The functions
\begin{equation}\label{cqg:m:abc}
a=yz^k,\qquad b=xz^s,\qquad c=z^N
\end{equation}
are \(H=\langle h\rangle\)-invariant, and \(\gamma\) acts on them with characters \((P-1,1,-\ell)\) modulo \(N\). Since the character on \(b\) has order \(N\), \(\Pi/H\) has order \(N\). Thus \(|\Pi|=N^2\), and \(\Gamma=\Pi/H\simeq\mu_N\).

For \(h^i\gamma^j\), \(0\le i,j<N\), put \(u=Ni-\ell j\) modulo \(N^2\). Its weights modulo \(N^2\) are
\[
Nj-su,\qquad (P-1)Nj-ku,\qquad u.
\]
Suppose that it fixes a point of \(Z\setminus\{0\}\). If \(z\ne0\), then \(u=0\). If also \(x\ne0\), the first character gives \(N\mid j\). If \(x=0\), then \(y\ne0\), and the second character and \(u=0\) give \(N\mid(P-1)j\) and \(N\mid\ell j\). Since \(\ell\equiv-1\pmod{P-1}\), these again give \(N\mid j\). In either case \(N\mid i\). If \(z=0\), the equation of \(Z\) implies \(xy\ne0\). Taking \(k\) times the first character minus \(s\) times the second gives \((1-N)Nj\), so \(N\mid j\). The remaining two characters give \(N\mid si\) and \(N\mid ki\); since \(s-k(Q-1)=1\), they imply \(N\mid i\). Hence the action is free.

Put \(a_3=n-N+1\), \(a_3^*=N-PQ+1\). Then \(1\le a_3,a_3^*<n\), and
\begin{equation*}
a_3a_3^*-1=\left((k-2)\ell-Q+1\right)n.
\end{equation*}
Thus \(a_3^*\) is the inverse of \(a_3\) modulo \(n\). The rational Euler number is
\begin{equation*}
e=2-\frac1P-\frac1Q-\frac{n-N+1}{n}
=\frac{\ell n+PQ(N-1)}{PQn}
=\frac{N^2}{PQn}>0.
\end{equation*}
For the positive plumbing matrix \(M\), put \((n_1,n_2,n_3)=(P,Q,n)\) and \((a_1^*,a_2^*,a_3^*)=(1,1,N-PQ+1)\). The chain corner entries and the Schur complement give
\[
(M^{-1})_{ij}=\frac{1}{n_i n_j e}\quad(i\ne j),\qquad (M^{-1})_{ii}=\frac{a_i^*}{n_i}+\frac{1}{n_i^2 e}.
\]
Substitution gives
\begin{equation}\label{cqg:m:leaf-block}
(M^{-1})_{\mathrm{leaves}}=\frac1{N^2}
\begin{pmatrix}
s(N-Q)&n&Q\\
n&k(P-1)N-kP&P\\
Q&P&\ell(N+1)
\end{pmatrix}.
\end{equation}
The entries follow from \eqref{cqg:m:identities}.

The generators corresponding to the end vertices of the discriminant representation are the rows of \eqref{cqg:m:leaf-block}, modulo \(\mathbb{Z}^3\). In \(\Pi\) they are
\[
\delta_1=h^{Q-1}\gamma^s,\qquad \delta_2=h\gamma^k,\qquad \delta_3=h^\ell\gamma^{-1}.
\]
Substitution using \eqref{cqg:m:identities} gives these rows. Conversely,
\[
h=\delta_2^s\delta_1^{-k},\qquad \gamma=\delta_1\delta_2^{-(Q-1)}.
\]
Hence \(\Pi\) is exactly the discriminant group with its representation on the coordinates corresponding to the end vertices. Its order agrees with \(\det M=PQne=N^2\).

The negative definite graph has one splice node with edge weights \(P,Q,n\). Its semigroup condition is supplied by \(x^P,y^Q,z^n\), and its congruence condition is their common character under the generators \(h,\gamma\), verified above. The splice-quotient theorem~\cite[Theorem~7.2]{NeumannWahl2005} identifies the resolution graph of \(Z/\Pi\) with \eqref{cqg:m:seifert}. The quotient is weighted homogeneous; its three attachment points normalize to \(0,1,\infty\), so this identifies the required coordinate ring of the central fiber.

The residue characters are
\[
\begin{aligned}
\kappa(h)&=\frac{s(P-1)-k+1}{N}=1,\\
\kappa(\gamma)&=\frac{PN+(n+s+k-1)\ell}{N^2}=1-\frac{\ell}{N}.
\end{aligned}
\]
Thus \(\widetilde B\) is Gorenstein and the induced canonical character is \(-\ell\), of order \(N/\gcd(N,\ell)=N/\gcd(P,Q)\). Freeness identifies \(\widetilde B^K\) with the canonical cover.
\end{proof}

\begin{lemma}\label{cqg:m:generators}
For \(1\le j\le k\), put
\[
e_j=\left\lfloor\frac{j-1}{P}\right\rfloor,\qquad m_j=x^{P-j+Pe_j}y^{(j-1)(Q-1)-Qe_j}z^{k-j}.
\]
Then \(\widetilde B\) is generated by \(a,b,c,m_1,\ldots,m_k\), where \(a,b,c\) are defined in \eqref{cqg:m:abc}.
\end{lemma}

\begin{proof}
The exponent of \(x\) in \(m_j\) lies in \(\{0,\ldots,P-1\}\). If \(j-1=Pe_j+h\), with \(0\le h<P\), its exponent of \(y\) is \(e_j\ell+h(Q-1)\), which is nonnegative. Thus all displayed monomials are polynomials. Their \(H\)-characters are zero by \(Ps-kQ=N\).

First replace \(x^P+y^Q+z^n\) by \(x^P+y^Q\). Let \(x^iy^jz^l\) be an invariant monomial, and write
\[
l-si-kj=Nw,\qquad d=Qi+Pj,\qquad h=\frac{d+\ell l}{N},\qquad u=l-kh.
\]
The identities \(Q+s\ell=(Q-1)N\) and \(P+k\ell=N\) give
\[
h=(Q-1)i+j+\ell w,\qquad u=i+Pw.
\]
In particular \(h,u\) are integers. Writing \(q=-w\), all possible representatives with these coordinates have
\begin{equation*}
i=u+Pq,\qquad j=h-(Q-1)u-Qq,\qquad l=kh+u.
\end{equation*}
The conditions \(i,j\ge0\) are equivalent to
\[
h\ge f(u):=(Q-1)u+Q\left\lceil-\frac uP\right\rceil,
\]
and \(l\ge0\) is equivalent to \(kh+u\ge0\). If \(u=Pq+r\ge0\), with \(0\le r<P\), then
\[
f(u)=\ell q+(Q-1)r.
\]
The generator coordinates are
\[
a:(1,0),\qquad m_j:(1,-j),\qquad b:(Q-1,1),\qquad c:(\ell,P).
\]
Thus an allowable pair with \(u\ge0\) is the sum of \(q\) copies of the coordinate of \(c\), \(r\) copies of that of \(b\), and \(h-f(u)\) copies of that of \(a\).

If \(u=-v<0\), write \(v=Pq+r\), with \(q\ge0\) and \(1\le r\le P\). Then
\[
f(-v)=-\ell q+Q-(Q-1)r\le1.
\]
The condition \(kh-v\ge0\) implies \(h\ge\lceil v/k\rceil\ge1\), so it is the stronger condition. Choose \(h\) integers in \(\{0,\ldots,k\}\) whose sum is \(v\). Their corresponding generators among \(a,m_1,\ldots,m_k\) have total coordinate \((h,-v)\). This proves generation of the semigroup of allowable coordinates.

We also verify that these coordinates describe monomial classes even when \(P,Q\) are not relatively prime. Two invariant monomials with the same \((h,u)\) have the same \(l\) and \(d=Qi+Pj\). Write \(\delta=\gcd(P,Q)\). The difference of their exponents has the form \((Pq/\delta,-Qq/\delta,0)\). Equality of their \(H\)-characters gives \(Nq/\delta\in N\mathbb{Z}\), so \(q\in\delta\mathbb{Z}\). Their exponents therefore differ by an integral multiple of \((P,-Q,0)\). The relation \(x^P=-y^Q\) identifies the corresponding monomials up to sign. Conversely, the component of a fixed \((d,l)\) is spanned by such monomials, and the binomial relation links consecutive representatives. Hence the semigroup calculation proves generation of the invariant binomial quotient.

Return to \(x^P+y^Q+z^n=0\). Given an invariant monomial of weighted degree in the two variables \(d=Qi+Pj\), choose a product of the proposed generators representing the same binomial class. Their difference is divisible by \(x^P+y^Q\). Replacing this factor by \(-z^n\) expresses the difference as a sum of invariant monomials of weighted degree in the two variables \(d-PQ\). The terms remain invariant because \(x^P,y^Q,z^n\) have the same character. Induction on the nonnegative integer \(d\) proves the assertion. Exactness of invariants for the finite group \(H\) ensures that invariant monomial classes span \(\widetilde B\).
\end{proof}

\subsection{Deformed generators and multiplication relations}

We next deform these generators. Put
\begin{equation*}
R_0=\mathbb{C}[a,b,c,t]/(c^k+ca^Q+b^P+tab),\qquad \alpha=a^{Q-1}.
\end{equation*}
The ring \(R_0\) is a domain because the defining polynomial is primitive and linear in \(t\), with \(\gcd(ab,c^k+ca^Q+b^P)=1\). Its central quotient is an irreducible normal hypersurface: its gradient vanishes only at the origin, and a factorization of a positively weighted homogeneous polynomial in three variables would give a positive-dimensional intersection of its factors in the singular locus. The substitution \eqref{cqg:m:abc} gives
\[
c^k+ca^Q+b^P=z^{N+kQ}(x^P+y^Q+z^n).
\]
After inverting \(c\), adjoining a root \(z^N=c\) recovers \(x=bz^{-s}\) and \(y=az^{-k}\), with deck group \(H\). Consequently
\begin{equation}\label{cqg:m:central-localization}
(R_0/(t))_c\simeq\widetilde B_c.
\end{equation}

Introduce an auxiliary variable \(w\), and define
\begin{equation}\label{cqg:m:series}
D(w)=c+tw+a^\ell w^P,\qquad L(w)=\sum_{h=0}^{P-1}\alpha^h b^{P-1-h}w^h+ta,\qquad \frac{L(w)}{D(w)}=\sum_{j\ge1}q_jw^{j-1}.
\end{equation}
The coefficients \(q_j\) initially belong to \((R_0)_c\). The polynomial identity governing the construction is
\begin{equation}\label{cqg:m:series-identity}
(b-\alpha w)L(w)=-c^k-a^QD(w).
\end{equation}
Indeed the left side is \(b^P-\alpha^Pw^P+tab-ta\alpha w\); substituting \(b^P+tab=-c^k-ca^Q\), \(\alpha^P=a^{\ell+Q}\), and \(a\alpha=a^Q\) gives the right side.

\begin{lemma}\label{cqg:m:rolling}
For \(1\le j\le k\), the coefficient \(q_j\) belongs to \((R_0)_b\cap(R_0)_c\), and
\begin{equation}\label{cqg:m:q-central}
q_j|_{t=0}=(-1)^{e_j}m_j.
\end{equation}
Moreover the following identities hold. Put
\[
p_h=[w^h]\frac{c^k}{D(w)},\qquad M(w)=\sum_{h=0}^{P-2}\alpha^h b^{P-2-h}w^h,\qquad r_i=[w^{i-1}]\frac{c^{k-1}M(w)}{D(w)}.
\]
Then
\begin{align}\label{cqg:m:linear-relations}
bq_1&=-a^Q-c^{k-1},& bq_{j+1}&=\alpha q_j-p_j,\\
aq_{i+1}&=-q_1q_i-r_i,& cq_j&=l_{j-1}-tq_{j-1}-a^\ell q_{j-P},
\end{align}
where \(l_h=[w^h]L(w)\), and coefficients with nonpositive index in the last equality are zero. For \(0\le h\le k-1\) and \(1\le i\le k-1\),
\begin{equation}\label{cqg:m:p-r-divisibility}
p_h\in c^{k-1-h}\mathbb{C}[a,c,t],\qquad r_i\in c^{k-1-i}\mathbb{C}[a,b,c,t].
\end{equation}
For \(1\le i<j\le k-1\), set \(\Delta_{i,j}=q_iq_{j+1}-q_{i+1}q_j\), and extend by \(\Delta_{j,i}=-\Delta_{i,j}\), \(\Delta_{i,i}=0\). Then
\begin{equation}\label{cqg:m:minor-divisibility}
\Delta_{i,j}\in a c^{k-1-j}\mathbb{C}[a,b,c,t].
\end{equation}
\end{lemma}

\begin{proof}
Expand
\[
\frac1{D(w)}=\frac1c\sum_{m\ge0}(-1)^m\left(\frac{tw+a^\ell w^P}{c}\right)^m.
\]
A term contributing to \(w^h\) has \(m\le h\). Multiplication by \(c^k\) therefore leaves a factor \(c^{k-1-h}\), proving the first inclusion in \eqref{cqg:m:p-r-divisibility}. The same expansion after multiplication by \(c^{k-1}M(w)\) gives the second inclusion. Dividing \eqref{cqg:m:series-identity} by \(D(w)\) gives the first two relations in \eqref{cqg:m:linear-relations}, which inductively place \(q_1,\ldots,q_k\) in \((R_0)_b\). The last relation is coefficient comparison in \(D(w)\sum q_jw^{j-1}=L(w)\).

For the relation involving \(a\), the identity
\[
wL(w)+ac=aD(w)+bwM(w)
\]
uses \(\alpha^{P-1}=a^{\ell+1}\). Since \(cq_1=b^{P-1}+ta\), it gives
\[
a\frac{L(w)/D(w)-q_1}{w}+q_1\frac{L(w)}{D(w)}
=\frac{M(w)(a^Q+bq_1)}{D(w)}
=-\frac{c^{k-1}M(w)}{D(w)}.
\]
Taking coefficients proves the remaining relation.

The minors \(\Delta_{i,j}\) satisfy the recurrences
\begin{align*}
c\Delta_{i,j}&=-t\Delta_{i,j-1}+\alpha^{j-1}b^{P-1-j}p_i&& (j<P),\\
c\Delta_{i,P}&=-t\Delta_{i,P-1}+a^\ell r_i,\\
c\Delta_{i,j}&=-t\Delta_{i,j-1}-a^\ell\Delta_{i,j-P}&& (j>P).
\end{align*}
These follow by subtracting \(q_{i+1}\) times the equation for \(cq_j\) from \(q_i\) times that for \(cq_{j+1}\), using \(\alpha q_i-bq_{i+1}=p_i\) and \(q_1q_i+aq_{i+1}=-r_i\). For \(j>P\), the index \(j-P\) is positive; when \(j-P<i\), use antisymmetry.

Induct on the larger index \(j\). In the first recurrence \(i<j\) gives \(p_i\in c^{k-j}\mathbb{C}[a,c,t]\); the same bound holds for \(r_i\) in the middle recurrence. Each earlier minor in the other terms is divisible by \(c^{k-j}\). Hence division by \(c\) leaves a polynomial divisible by \(c^{k-1-j}\). The initial nonzero term has the factor \(\alpha^{j-1}\), with \(j\ge2\); the transition term has the factor \(a^\ell\). Both are divisible by \(a\), and the recurrence preserves this factor. This proves \eqref{cqg:m:minor-divisibility}.

Finally, write \(j-1=Pe_j+h\), with \(0\le h<P\). At \(t=0\), exactly one numerator term contributes to the coefficient in \eqref{cqg:m:series}, and
\[
q_j|_{t=0}=(-1)^{e_j}\frac{a^{\ell e_j}\alpha^h b^{P-1-h}}{c^{e_j+1}}.
\]
Under \eqref{cqg:m:abc}, the three exponents are
\[
P-1-h=P-j+Pe_j,\qquad \ell e_j+(Q-1)h=(j-1)(Q-1)-Qe_j,
\]
and
\[
k\ell e_j+k(Q-1)h+s(P-1-h)-N(e_j+1)=k-j.
\]
These are the exponents in \(m_j\), proving \eqref{cqg:m:q-central}.
\end{proof}

\subsection{Adjoining fractions and lifting relations}

We show that adjoining \(q_1,\ldots,q_k\) gives a flat Gorenstein deformation of \(\widetilde{B}\), with an extension of the covering action. The defining relations give lifts of generators of the first syzygy module.

\begin{lemma}\label{cqg:m:flat-family}
The algebra
\[
R_k=R_0[q_1,\ldots,q_k]\subset(R_0)_c
\]
is a finitely generated Gorenstein domain, flat over \(\mathbb{C}[t]\), and \(R_k/(t)\simeq\widetilde B\). The \(\Gamma\)-action on its central fiber extends to \(R_k\), trivially on \(t\).
\end{lemma}

\begin{proof}
Give the variables the positive degrees
\[
\deg a=(k-1)P,\qquad \deg b=kQ,\qquad \deg c=PQ,\qquad \deg t=N.
\]
Assigning \(w\) degree \(PQ-N\) in \eqref{cqg:m:series} gives
\begin{equation*}
\deg q_j=d_j:=Q((k-1)P-k)+(j-1)(N-PQ)>0.
\end{equation*}
We apply Lemma~\ref{uqg:fractions} successively to the multiplication relations. At the first stage, \(I_0=(b,c)\), \(R_0/I_0=\mathbb{C}[a,t]\), and the products \(bq_1=-a^Q-c^{k-1}\), \(cq_1=b^{P-1}+ta\) have nonzero residue \(-a^Q\). The lemma identifies the presentation by these relations with \(R_1=R_0[q_1]\), whose quotient by \((a,b,c)\) is \(\mathbb{C}[t,q_1]\).

For \(2\le j\le k\), put
\begin{equation*}
I_{j-1}=(a,b,c,q_1,\ldots,q_{j-2})\subset R_{j-1},\qquad R_{j-1}=R_0[q_1,\ldots,q_{j-1}].
\end{equation*}
Assume \(R_{j-1}/I_{j-1}=\mathbb{C}[t,q_{j-1}]\). The products by \(a,b,c\) are given in \eqref{cqg:m:linear-relations}; the others are
\begin{equation*}
q_jq_i=q_{i+1}q_{j-1}+\Delta_{i,j-1}\quad(1\le i\le j-2).
\end{equation*}
All lie in \(R_{j-1}\). At \(j=2\), the residue of \(q_2a\) is \(-q_1^2\), since \(r_1=c^{k-2}b^{P-2}\); at \(j\ge3\), the residue of \(q_jq_{j-2}\) is \(q_{j-1}^2\), since \(a\mid\Delta_{j-2,j-1}\). These nonzero residues give a presentation of the domain of \(R_j\) by Lemma~\ref{uqg:fractions}. If \(j<k\), every relation vanishes modulo \(a,b,c,q_1,\ldots,q_{j-1}\): the corrections \(p_{j-1},r_{j-1},\Delta_{i,j-1}\) contain \(c^{k-j}\), and \(l_{j-1}\in(a,b)\). The next quotient is therefore \(\mathbb{C}[t,q_j]\), completing the induction.

At \(t=0\), the same quotient rings are \(\mathbb{C}[a]\) and \(\mathbb{C}[q_{j-1}]\), and the residues remain \(-a^Q\), \(-q_1^2\), and \(q_{j-1}^2\). Applying the lemma to these specialized presentations and using \eqref{cqg:m:q-central} gives
\[
R_j/(t)\simeq\mathbb{C}[a,b,c,q_1|_{t=0},\ldots,q_j|_{t=0}].
\]
The identification occurs in \(\operatorname{Frac}(R_0/(t))\) by \eqref{cqg:m:central-localization}. At \(j=k\), Lemma~\ref{cqg:m:generators} identifies the coordinate ring of the central fiber with \(\widetilde B\).

We make the resulting normal form explicit. The following disjoint union is a \(\mathbb{C}[t]\)-basis of \(R_k\):
\begin{equation*}
\begin{aligned}
\mathcal{N}_0&=\{a^u c^v b^w:u,v\ge0,\quad 0\le w<P\},\\
\mathcal{N}_1&=\{a^u q_1^v:u\ge0,\quad v\ge1\},\\
\mathcal{N}_j&=\{q_{j-1}^u q_j^v:u\ge0,\quad v\ge1\}\quad(2\le j\le k).
\end{aligned}
\end{equation*}
The monic equation eliminates \(b^P\), and the filtration of Lemma~\ref{uqg:fractions}, with quotients \(\mathbb{C}[a,t]\) and \(\mathbb{C}[t,q_{j-1}]\), proves spanning and independence, also at \(t=0\). Reduction of coefficients modulo \(I_{j-1}\), followed by the multiplication relations, lowers the degree in the newly adjoined variable. Iterating and finally eliminating \(b^P\) gives a terminating reduction; its recorded subtractions express each zero normal form in the defining equations.

At each stage write \(R_{j-1}=P_{j-1}/(\mathbf{R})\), let \(\mathbf{h}\) be the coordinate row generating \(I_{j-1}\), and put \(\mathbf{u}=q_j\mathbf{h}\). The old equations lie in \((\mathbf{h})\), so monomial division gives \(\mathbf{R}=\mathbf{h}C\). For each coordinate Koszul column or column \(C_\alpha\), write \(\mathbf{h}\mathbf{a}=\mathbf{R}\mathbf{b}\), with \(\mathbf{b}=0\) or \(\mathbf{b}=e_\alpha\), respectively. The normal-form reduction computes \(\mathbf{d}\) with \(\mathbf{u}\mathbf{a}=\mathbf{R}\mathbf{d}\). Lemma~\ref{uqg:fraction-syzygies} supplies the new columns
\begin{equation}\label{cqg:m:full-syzygies}
\begin{pmatrix}\mathbf{d}-Q_j\mathbf{b}\\\mathbf{a}\end{pmatrix}
\quad\text{for the equation row}\quad
F_j=(\mathbf{R},Q_j\mathbf{h}-\mathbf{u}),
\end{equation}
together with the old and mixed Koszul columns. They form \(\Sigma_j\) with \(F_j\Sigma_j=0\), and generate all central syzygies after specialization, by the specialized domain identifications above. Lemmas~\ref{uqg:flat-relations} and~\ref{uqg:graded-flatness} give flatness over \(\mathbb{C}[t]\).

For clarity, the first matrix is explicit. Put \(B_1=b^{P-1}+ta\) and \(C_1=c^{k-1}+a^Q\), so that \(R_0\) has equation \(bB_1+cC_1\). For the equation row
\[
F_1=(bB_1+cC_1,\ Q_1b+C_1,\ Q_1c-B_1),
\]
a matrix whose columns generate the first syzygy module is
\[
\Sigma_1=
\begin{pmatrix}
-Q_1&-1\\
B_1&c\\
C_1&-b
\end{pmatrix}.
\]
These columns come from the equation from the preceding stage and the Koszul relation of \((b,c)\); their specialization at \(t=0\) generates all syzygies. Formula~\eqref{cqg:m:full-syzygies} gives every subsequent matrix.

Extend the \(\Gamma\)-characters by assigning \(t\) character zero and \(w\) character \(-\ell\). Then \(D(w)\) has character \(-\ell\), \(L(w)\) has character \(P-1\), and
\begin{equation*}
\operatorname{char}_{\Gamma}(q_j)=P-1+j\ell\pmod N.
\end{equation*}
Every equation and multiplication map is equivariant. Monomial division respects characters; projecting the syzygy columns to their prescribed characters if necessary gives equivariant matrices throughout. 

The coordinate ring of the central fiber \(\widetilde B\) is Gorenstein. Since \(t\) is a nonzerodivisor, the local algebra of \(R_k\) at its homogeneous maximal ideal is Gorenstein. The Gorenstein locus is open and invariant under the positive grading. Every orbit closure contains the vertex, so its closed invariant complement is empty. Hence \(R_k\) is Gorenstein everywhere.

Proposition~\ref{uqg:invariant-presentation}, applied to \((F_k,\Sigma_k)\) and \(K\), gives an equivariant polynomial presentation of \(R_k^K\), with central fiber \(B\) and a lift of its full first syzygy module. Its monomial inputs are determined by the coordinate characters modulo \(\gcd(P,Q)\).
\end{proof}

\subsection{The four families}

The construction and all its verifications are common to the four families. To identify their Seifert data, let \(p,q,r\ge0\) be the table parameters and make the substitutions
\[
\begin{array}{c|ccc}
\text{Figure}&P&Q&k\\\hline
1(d)&r+4&2&q+3\\
1(e)&3&p+3&q+3\\
1(f)&3&2&q+3\\
1(g)&r+4&p+3&q+3
\end{array}
\]
The identity
\begin{equation*}
\frac n{N-1}=[k-1,P-1,Q,2^{k-1}]
\end{equation*}
follows successively from \([Q,2^{k-1}]=s/k\), \([P-1,Q,2^{k-1}]=(N-1)/s\), and \((k-1)(N-1)-s=n\). Hirzebruch--Jung duality therefore gives
\[
\frac n{n-N+1}=
\begin{cases}
[2^{k-3},3,2^{P-4},k+2],&P\ge4,\ Q=2,\\
[2^{k-3},4,2^{Q-3},k+1],&P=3,\ Q\ge3,\\
[2^{k-3},k+3],&P=3,\ Q=2,\\
[2^{k-3},3,2^{P-4},3,2^{Q-3},k+1],&P\ge4,\ Q\ge3.
\end{cases}
\]
These are precisely the third arms in Figures~1(d)--(g), so Proposition~\ref{cqg:m:existence} covers every parameter in all four families.

\raggedbottom


\begin{thebibliography}{99}
\bibitem{ArtalBartoloWahl2022} E. Artal Bartolo and J. Wahl, \textit{Fundamental group of rational homology disk smoothings of surface singularities}, J. Singul. \textbf{24} (2022), 126--144.

\bibitem{Artin1966} M. Artin, \textit{On isolated rational singularities of surfaces}, Amer. J. Math. \textbf{88} (1966), 129--136.

\bibitem{Beauville} A. Beauville, \textit{Les familles stables de courbes elliptiques sur \(\mathbb{P}^1\) admettant quatre fibres singuli\`eres}, C. R. Acad. Sci. Paris S\'er. I Math. \textbf{294} (1982), 657--660.

\bibitem{BPSS2026}
M. Beke, H. Park, D. Shin, A. I. Stipsicz, \textit{Normal complex surface singularities with rational homology disk smoothings}, preprint 2026.

\bibitem{BH} G. Belousov and D. Hwang, \textit{Del Pezzo surfaces with four log terminal singularities}, Taiwanese J. Math. \textbf{29} (2025), no. 6, 1133--1163; arXiv:2411.01957v2.

\bibitem{BS} M. Bhupal and A. I. Stipsicz, \textit{Weighted homogeneous singularities and rational homology disk smoothings}, Amer. J. Math. \textbf{133} (2011), no. 5, 1259--1297.

\bibitem{Blache1995} R. Blache, \textit{Riemann--Roch theorem for normal surfaces and applications}, Abh. Math. Sem. Univ. Hamburg \textbf{65} (1995), 307--340.

\bibitem{Deligne1971} P. Deligne, \textit{Th\'{e}orie de Hodge: II}, Publ. Math. Inst. Hautes \'{E}tudes Sci. \textbf{40} (1971), 5--57.

\bibitem{Dolgachev} I. V. Dolgachev, \textit{Classical Algebraic Geometry: A Modern View. II}, author's revised version, 10 August 2025, Theorem 8.3.2 and Section 8.8.1. \href{https://sites.lsa.umich.edu/idolga/wp-content/uploads/sites/1467/2025/10/CAG-2.pdf}{Author's PDF}.

\bibitem{EisenbudHarris1987} D. Eisenbud and J. Harris, \textit{On varieties of minimal degree (a centennial account)}, Algebraic geometry, Bowdoin, 1985, Proc. Sympos. Pure Math. \textbf{46}, Part 1, Amer. Math. Soc., Providence, RI, 1987, 3--13.

\bibitem{Fowler} J. R. Fowler, \textit{Rational homology disk smoothing components of weighted homogeneous surface singularities}, Ph.D. thesis, University of North Carolina at Chapel Hill, 2013.

\bibitem{GreuelLossenShustin2007} G.-M. Greuel, C. Lossen, and E. Shustin, \textit{Introduction to Singularities and Deformations}, Springer Monographs in Mathematics, Springer, Berlin, 2007.

\bibitem{GreuelSteenbrink1983} G.-M. Greuel and J. H. M. Steenbrink, \textit{On the topology of smoothable singularities}, Proc. Sympos. Pure Math. \textbf{40}, Part 1 (1983), 535--545.

\bibitem{HwangKeum2011} D. Hwang and J. Keum, \textit{The maximum number of singular points on rational homology projective planes}, J. Algebraic Geom. \textbf{20} (2011), 495--523.

\bibitem{Laufer1973} H. B. Laufer, \textit{Taut two-dimensional singularities}, Math. Ann. \textbf{205} (1973), 131--164.

\bibitem{LeeNakayama2018} Y. Lee and N. Nakayama, \textit{Grothendieck duality and \(\mathbb{Q}\)-Gorenstein morphisms}, Publ. Res. Inst. Math. Sci. \textbf{54} (2018), no. 3, 517--648.

\bibitem{Neumann1983} W. D. Neumann, \textit{Abelian covers of quasihomogeneous surface singularities}, Singularities, Part 2 (Arcata, Calif., 1981), Proc. Sympos. Pure Math. \textbf{40}, Part 2, Amer. Math. Soc., Providence, RI, 1983, 233--243.

\bibitem{NeumannWahl2005} W. D. Neumann and J. Wahl, \textit{Complete intersection singularities of splice type as universal abelian covers}, Geom. Topol. \textbf{9} (2005), 699--755.

\bibitem{Papadakis2004} S. Papadakis, \textit{Kustin--Miller unprojection with complexes}, J. Algebraic Geom. \textbf{13} (2004), no. 2, 249--268.

\bibitem{PapadakisReid2004} S. Papadakis and M. Reid, \textit{Kustin--Miller unprojection without complexes}, J. Algebraic Geom. \textbf{13} (2004), no. 3, 563--577.

\bibitem{PS} H. Park and D. Shin, \textit{Positivity of Wahl defects and specialization of reflexive powers of the canonical sheaf}, \href{https://doi.org/10.48550/arXiv.2610.21092}{arXiv.2610.21092}.

\bibitem{Pinkham} H. Pinkham, \textit{Deformations of normal surface singularities with \(\mathbb{C}^*\)-action}, Math. Ann. \textbf{232} (1978), 65--84.

\bibitem{Pinkham1977} H. Pinkham, \textit{Normal surface singularities with \(\mathbb{C}^*\) action}, Math. Ann. \textbf{227} (1977), 183--193.

\bibitem{RademacherGrosswald1972} H. Rademacher and E. Grosswald, \textit{Dedekind Sums}, Carus Mathematical Monographs \textbf{16}, Mathematical Association of America, Washington, DC, 1972.

\bibitem{SSW} A. I. Stipsicz, Z. Szab\'o, and J. Wahl, \textit{Rational blowdowns and smoothings of surface singularities}, J. Topol. \textbf{1} (2008), no. 2, 477--517.

\bibitem{Wahl1981} J. Wahl, \textit{Smoothings of normal surface singularities}, Topology \textbf{20} (1981), no. 3, 219--246.

\bibitem{Wahl1988} J. Wahl, \textit{Deformations of quasi-homogeneous surface singularities}, Math. Ann. \textbf{280} (1988), 105--128.

\bibitem{Wahl2011} J. Wahl, \textit{On rational homology disk smoothings of valency 4 surface singularities}, Geom. Topol. \textbf{15} (2011), 1125--1156.

\bibitem{Wahl2013} J. Wahl, \textit{Log-terminal smoothings of graded normal surface singularities}, Michigan Math. J. \textbf{62} (2013), 475--489.
\bibitem{Wahl2015} J. Wahl, \textit{Milnor and Tjurina numbers for smoothings of surface singularities}, Algebraic Geometry \textbf{2} (2015), no. 3, 315--331.
\end{thebibliography}
\end{document}